\documentclass[12pt,reqno]{amsart}
     \makeatletter
     \def\section{\@startsection{section}{1}%
     \z@{.7\linespacing\@plus\linespacing}{.5\linespacing}%
     {\bfseries
     \centering
     }}
     \def\@secnumfont{\bfseries}
     \makeatother

\usepackage{float}

\usepackage[body=18cm,top=3cm,bottom=3cm]{geometry}
\usepackage[english]{babel}
\usepackage{tikz}
\usetikzlibrary{mindmap,trees}  

\usepackage{tocloft}

\renewcommand{\div}{\mathrm{div}}

\newtheorem{theorem}{Theorem}[section]
\newtheorem{lemma}[theorem]{Lemma}
\newtheorem{proposition}[theorem]{Proposition}
\newtheorem{corollary}[theorem]{Corollary}
\theoremstyle{definition}
\newtheorem{definition}[theorem]{Definition}

\theoremstyle{remark}
\newtheorem{remark}[theorem]{Remark}
\numberwithin{equation}{section}
\usepackage[utf8]{inputenc}
\usepackage{amsmath}
\usepackage{amsfonts}
\usepackage{amssymb}

\usepackage{lipsum}
\usepackage{indentfirst}
\usepackage{calc}
\usepackage{changepage}

\usepackage{titlesec}

\newlength{\anIndent}
\titleformat{\section}
{\normalfont\Large\bfseries}{\thesection}{2\anIndent-\widthof{\thesection}}{}

\titleformat{\subsection}
{\normalfont\large\bfseries}{\thesubsection}{3\anIndent-\widthof{\thesubsection}}{}

\titleformat{\subsubsection}
{\normalfont\normalsize\bfseries}{\thesubsubsection}{4\anIndent-\widthof{\thesubsubsection}}{}

\input{tcilatex}

\begin{document}

\title[Well-posedness of Deterministic Transport Eq. with Singular Vector Field]{Well-posedness of the Deterministic Transport Equation with Singular Velocity Field Perturbed along Fractional Brownian Paths}

\author{Oussama Amine}
\address{Oussama Amine: Department of Mathematics, University of Oslo, P.O. Box 1053, Blindern, N--0316 Oslo, Norway}
\email{oussamaa@math.uio.no}

\author{Abdol-Reza Mansouri}
\address{Abdol-Reza Mansouri: 416 Jeffery Hall,
	Queen's University,
	Mathematics and Statistics,
	Kingston,  ON K7L 3N6, Canada}
\email{mansouri@queensu.ca}

\author{Frank Proske}
\address{Frank Proske: Department of Mathematics, University of Oslo, P.O. Box 1053, Blindern, N--0316 Oslo, Norway}
\email{proske@math.uio.no}

\subjclass[2010] {--, --.}


    \begin{abstract}
			In this article we prove path-by-path uniqueness in the sense of Davie \cite{5Davie07} and Shaposhnikov \cite{5Shaposhnikov16} for SDE's driven by a fractional Brownian motion with a Hurst parameter $H\in(0,\frac{1}{2})$, uniformly in the initial conditions, where the drift vector field is allowed to be merely bounded and measurable.\par
		Using this result, we construct weak unique regular solutions in $W_{loc}^{k,p}\left([0,1]\times\mathbb{R}^d\right)$, $p>d$ of the classical transport and continuity equations with singular velocity fields perturbed along fractional Brownian paths.\par
		The latter results provide a systematic way of producing examples of singular velocity fields, which cannot be treated by the regularity theory of DiPerna-Lyons \cite{5DiPernaLions89}, Ambrosio \cite{5Ambrosio04} or Crippa-De Lellis \cite{5CrippaDeLellis08}.\par
		Our approach is based on a priori estimates at the level of flows generated by a sequence of mollified vector fields, converging to the original vector field, and which are uniform with respect to the mollification parameter. In addition, we use a compactness criterion based on Malliavin calculus from \cite{5DMN92} as well as a supremum estimate in time of moments of the derivative of the flow of SDE solutions. 
		
		\emph{keywords}: Transport equation, Compactness criterion, Singular vector fields, Regularization by noise.
		
		
	\end{abstract}

\maketitle

\section{Introduction}

Consider the following linear Continuity Equation (CE):
\begin{equation}\label{5CE}
\partial_t u + \div(bu) = 0;\quad (t,x)\in[0,T]\times\mathbb{R}^d ,
\end{equation}

where $b(t,x)\in\mathbb{R}^d$ and $u(t,x)\in\mathbb{R}$. Though it has been one of the most elementary partial differential equations (PDE) it plays an essential role in fluid mechanics and the theory of conservation laws, hence a deeper understanding of it is crucial. This equation has a close link with the following dynamical system:

\begin{equation}\label{5ODE}
\begin{cases}
\partial_t X(t,x,s) = b(t,X(t,x,s)) \\
X(s,x,s) = x
\end{cases}
\end{equation}

The link is based on a duality between a Lagrangian and Eulerian points of view, namely the (CE) represents a pointwise, in $(t,x)$, description of the dynamics. This is in contrast to the particle-wise Lagrangian description which traces single particle paths through $($time$\times$space$)$. In the smooth case these two view points are in fact equivalent and one has that the solution of (CE) with initial datum $\bar{u}$ is given by: \[ u_t = X(t,.)_{\#} \bar{u} \]
i.e. the pushforward of the initial datum by the flow generated through (\ref{5ODE}). The existence of this flow, in this smooth setting, is guaranteed by the classical Cauchy-Lipschitz theory. The previous explicit solution formula holds in fact even when $b$ is not within the range of the classical theory, as long as one can give a compatible meaning to both (\ref{5CE}) and (\ref{5ODE}). Indeed taking the Eulerian viewpoint, and starting from the the closely related Transport Equation (\ref{5TE}):
\begin{equation}\label{5TE}
\partial_t u + b\cdot\nabla u = 0;\quad (t,x)\in\left(0,T\right)\times\mathbb{R}^d 
\end{equation}
DiPerna and Lions in their ground breaking work \cite{5DiPernaLions89} provided a positive answer to the question of well-posedness of both (\ref{5CE} )and (\ref{5TE}) and as a consequence to (\ref{5ODE}) as well. Their work is based on a new solution concept called \textit{renormalized solutions}. Briefly, a renormalized solution to (\ref{5CE}) is a distributional solution with the property that
\[\partial \beta(u) + \div(b\beta(u)) = (\beta(u) - u\beta^{'}(u))\div(b),
\]
where the equality is in the sense of distributions, for every $\beta\in C^1_b(\mathbb{R}^d)$ such that $\beta(0)=0$. Renormalized solutions have the property of obeying a "weak" version of the chain rule and this permits one to prove energy estimates rigorously. Indeed, and up to an approximation argument, by choosing $\beta(z) = |z|^p$ one obtains, using Gronwall's lemma, the following energy estimate

\[
\|u(t,.)\|_{L^p (\mathbb{R}^d)}\leq  \exp\left(\int_0^T \|\div(b)\|_{L^{\infty}(\mathbb{R}^d)}\right)^{1-\frac{1}{p}}\times\|\bar{u}\|_{L^p (\mathbb{R}^d)}
\]

for all $t\in\left(0,T\right)$ under appropriate conditions on $b$. Thus one gets, using the linearity of the equation, uniqueness in addition to continuous dependence on initial data under appropriate technical conditions on $b$.
Thus the question of well-posedness is reduced to the question of $b$ having the renormalization property, namely whether all distributional solutions are renormalizable. Indeed, under the assumption of $b\in L^1(\left(0,T\right);W^{1,1}_{loc}(\mathbb{R}^d))$, $\div(b)\in L^{\infty}$ and additional growth conditions, DiPerna and Lions showed that this property holds. The authors went on to show that the results developed at the level of (\ref{5CE})/(\ref{5TE}) using the notion of renormalized solutions can be used to define a generalized solution concept for (\ref{5ODE}), namely that of a generalized almost everywhere flow or its equivalent formulation of a regular Lagrangian flow as introduced in the seminal work of \cite{5Ambrosio04}. In this work the author goes on to show the renormalization property for $b$'s with only BV (i.e. bounded variation) regularity. On the other hand, the example in \cite{5DePauw02} suggests that these results are rather sharp and an extension of the DiPerna-Lions approach beyond the BV case is most probably unlikely. This should be contrasted with the following three key observations, at least in relation to the results of the present work. \par

\begin{enumerate}
	\item The Eulerian to Lagrangian route is not the only one possible and provided one starts with an appropriate notion for a generalized solution to (\ref{5ODE}) a theory can in principle be developed at the Lagrangian level and from it one can make conclusions in the other direction i.e. the Eulerian level. This is in fact the approach proposed in the ingenious work of Crippa and De Lellis \cite{5CrippaDeLellis08} were the authors have proved a priori quantitative estimates directly at the level of (\ref{5ODE}) and used these, in combination with certain functional inequalities, to show existence, uniqueness and compactness properties of the (generalized) flow. The viewpoint developed in this work  provided ideas and techniques that were used in many important works that followed e.g. \cite{5BouchutCrippa13}\cite{5Jabin10}\cite{5BOS11}\cite{5Seis17}\cite{5ChampagnatJabin10}  and \cite{5BreschJabin18} to name a few.\par
	Thus the question of whether to adopt the Lagrangian or Eulerian point of view is problem specific and sometimes a combination of ideas from both is necessary to achieve a deeper understanding in certain situations (see \cite{5BreschJabin18}).
	
	\item For $b$ enjoying some sort of structure, e.g. physical or geometric, one usually is able to say much more with sometimes looser requirements see e.g. \cite{5Hauray03},\cite{5BouchutDesvillettes01},\cite{5CCR06} and \cite{5Crippa-et.al.17}. Take for example \cite{5Hauray03}, where it is shown that existence and uniqueness hold for (\ref{5TE}) directly and for $b$ merely in $L^2 _{loc}$ provided a local property of $b$ holds on a large subset of points $x$. This is achieved using the Hamiltonian structure of $b$ in dimension $2$ in order to reduce the problem to a $1$ dimensional problem. Although such results may give the hope that with the right analysis one can attack problems with truly singular $b$'s i.e. no (weak) differentiability whatsoever the conventional wisdom suggests that a certain degree of differentiability is unavoidable.
	
	\item In contrast to the deterministic setting, the picture in the stochastic one is much more positive. It has been known since the work of \cite{5Zvonkin74} and its multidimensional extension in \cite{5Veretennikov79} that the addition of additive Brownian noise restores well-posedness to (\ref{5ODE}) when $b$ is merely bounded and measurable. The ideas in \cite{5Veretennikov79} were used in combination with the ones in \cite{5Ambrosio04} in order to provide a similar result for a transformed version of (\ref{5TE}) in the work of \cite{5FGP10}. In parallel, and using a different set of tools, a similar result was obtained in \cite{5MNP14}. There is however a caveat in these works, namely that the solution concept is different and the change, of the equation one is solving, does not only affect the dynamics, which may have a physical motivation, but it changes the very meaning of the word solution as well. These ideas are well explained in \cite{5Flandoli11} and thus we will not delve into them further, but it is sufficient to say that in almost all instances of restoration of well-posedness at the ODE or PDE levels (what is also called regularization by noise phenomena) the solution concept is tempered with in a rather fundamental way and thus the question of usefulness of such an endeavor is a legitimate one. This last argument has one counterpoint argument and it starts with the deep result of A. M. Davie \cite{5Davie07}.\newline  This work answers the following question posed by N. V. Krylov:\par
	\textbf{Question:} Fix $\epsilon>0$ and take a Brownian path $(B_t)_{t\in [0,1]}$.\newline Does the following ODE
	\begin{equation}\label{5RODE}
	\begin{cases}
	\partial_t X(t,x,s) = b\left(t,X(t,x,s) + \epsilon B_t\right)\\
	X(s,x,s) = x
	\end{cases}
	\end{equation}
	have a unique solution in the space of continuous functions for $b$ bounded and measurable?
	
	Davie's result gives an affirmative answer to this question for a.e. Brownian path for a fixed initial value $x\in\mathbb{R}^d$. This result was followed by the important work of \cite{5Shaposhnikov16} where the argument was simplified, improved to a uniform result in $x\in\mathbb{R}^d$ and to a large extent "streamlined", thus contributing to extensions of Davie's result to settings other than the Brownian one, namely Levy processes \cite{5Priola18}, Stochastic PDE setting \cite{5ButkovskyMytnik19} and Hilbert space setting \cite{5Wresch17}. In \cite{5CatellierGubinelli16} and \cite{5Beck-et.al.14} similar results, at least in spirit, are achieved at the ODE and PDE levels, respectively, albeit through different methods than those of \cite{5Davie07} and \cite{5Shaposhnikov16}. We stress the fact that the methods in \cite{5CatellierGubinelli16} and \cite{5Beck-et.al.14}, in contrast
	to the method employed in the present paper, cannot be used for the construction of path-by-path solutions to (\ref{5RODE}), in the case of fractional Brownian motion perturbation, when the drift vector
	field belongs to $L^{\infty }(\mathbb{R}^{d})$.
	
\end{enumerate}
The above 3 points in the previous discussion lead us to pose the following natural question:\newline
\textbf{Question:} Given a singular vector field $b$ is there a way of creating a smooth dynamical system, from the original one, by changing the "geometry" of time$\times$space? 

In this work we ask the following version of the above question:\newline
\textbf{Question':} Given a class of vector fields $\mathfrak{B}\ni b$ can one construct a "reasonable" and non-trivial transformation $\mathbb{T}:\mathfrak{B}\longrightarrow\mathfrak{B}$, which produces a "small" deviation from $b$, such that (\ref{5ODE}) with $\tilde{b}:=\mathbb{T}[b]$ is well-posed?\newline

\vspace{0.5cm}
Indeed we show that if  

\begin{equation*}
\mathfrak{B}:= L_{\infty ,\infty }^{1,\infty }:=L^{1}(\mathbb{R}^{d};L^{\infty }([0,T];%
\mathbb{R}^{d}))\cap L^{\infty }(\mathbb{R}^{d};L^{\infty }([0,T];\mathbb{R}%
^{d}))
\end{equation*}

then there exists a $($time$\times$space$)$-transformation such that for every $b\in\mathfrak{B}$ there exist a measurable set  $\Omega_b$ of full mass in the Wiener space $\left(\mathcal{C}_0\left([0,T]\right),\mathcal{B}\left(\mathcal{C}_0([0,T])\right),\mu\right)$ such that for all $\omega\in\Omega_b$  the map $\Psi_{\omega}(t,x):= (t,x+B^H_t(\omega))$ is such that:

\begin{equation}\label{5PsiODE}
\begin{cases}
\partial_t X(t,x,s) = b\circ\Psi_{\omega}\left(t,X(t,x,s)\right)\\
X(s,x,s) = x,
\end{cases}
\end{equation}

where $B^H_\cdot$ is a fractional Brownian motion with Hurst parameter $H\in\left(0,\frac{1}{2}\right)$ constructed on the aforementioned Wiener space (see Appendix), has a classical flow in the space of continuous paths.
Moreover we show that $\left(\Omega_b,\Psi_{\omega}\right)$ not only restores well-posedness, in a path-by-path way, but can be chosen so as to increase the classical differentiability of the resulting flows to any arbitrary order $k\geq 0$, with the right choice of the parameter of regularization $H$. Using this result we prove a similar effect at the PDE level.\newline
This is, to the best of our knowledge, the first instance of higher order regularization by noise phenomenon in the sense of improved differentiability of the characteristics even when the vector field in question is far from possessing any degree of regularity. This regularization effect is at the \textbf{pathwise level} and thus our result can be seen as a systematic way of producing regular differentiable, in the classical sense, characteristics for (\ref{5ODE}) for strongly singular vector fields. As for a related work, but in the sense of pathwise uniqueness, which is a weaker notion than path-by-path, we mention the paper \cite{5BNP19}.   

\subsection{Aim of this work}
Our aim in this work is to exhibit a regularization by noise phenomenon for the transport equation driven by a singular vector field that is ill-posed within the classical theory, all this while still retaining the deterministic character of the equation. More precisely, our main contribution is the following path-by-path uniqueness result

\begin{theorem}
	Let $b\in L_{\infty ,\infty }^{1,\infty }$ and\footnote{Here and in what follows, $C(\cdot)$ stands for a constant that depends only on the given arguments.} $H<
	C(d)<\frac{1}{2}$. Then there exists a measurable set $\Omega^*$ of full mass such that for all $\omega\in\Omega^*$
	
	\begin{equation}
	X_t^x = x + \int_0^t b(s,X_s^x)ds + B^H_t(\omega)
	\end{equation}
	has \textit{one and only one} solution \textit{uniformly in} $x\in\mathbb{R}^d$.
\end{theorem}

As an application of our main result, we get the following 
\begin{theorem}\label{5mainResultTE}
	Let $b\in L_{\infty ,\infty }^{1,\infty} $, $H< C(k,d)<\frac{1}{2}$, $u_0\in\mathcal{C}_b^k(\mathbb{R}^d)$, $p>d$, and $k\geq 2$.\newline
	 Define $b^*$ in 
	 \begin{equation}
	 \begin{cases}
	 \frac{\partial}{\partial t} u(t,x) + b^*(t,x)\cdot\nabla_x u(t,x) = 0\text{	,	}(t,x)\in\left[0,T\right]\times\mathbb{R}^d \\
	 u(0,x) = u_0(x) \text{	on	}\mathbb{R}^d
	 \end{cases}
	 \end{equation}
	 
	 as \[b^*(t,x)  = b(t,x+B_t^H(\omega)).
	\]
	
	Then there exists a measurable set $\tilde{\Omega}$ with $\mu(\tilde{\Omega}) =1$ such that for all $\omega\in\tilde{\Omega}$ there exists a unique weak solution $u = u_{\omega}$ in the class  $W^{1,p}_{loc}\left([0,T]\times\mathbb{R}^d\right)$. Moreover \[u(t,\cdot)\in  \bigcap_{p\geq 1} W^{k,p}_{loc}\left(\mathbb{R}^d\right)\] for all $t$.
\end{theorem}

\subsection{The method}

\begin{enumerate}
	\item \textbf{Compactness criterion:} We use a compactness criterion
	developed in \cite{5DMN92} to show that the sequence of solutions to
	(\ref{5SDE}) associated to a sequence of smooth mollification  approximating the driving
	vector field is relatively compact in $L^2$. This is the approach that was
	adopted in e.g. \cite{5MMNPZ13},\cite{5MNP14} or \cite{5ABP18}. The novelty
	in the present work is related to our extensive use of this tool not only
	at the level of the flows but also at the level of their derivatives.
	This culminates in Theorem \ref{5Identification} which says, roughly, that the unique
	flow solving (\ref{5SDE}) admits a spatial derivative and both live in
	$\mathcal{C}([s,T]\times K;L^{2}(\Omega ;\mathbb{R}^{d}))$.
	 It is important to mention that the derivative of the flow plays an important role in many of the computations that we carry out at both the ODE and PDE levels. The higher regularizing power of the fBm, when $H$ is appropriately small, permits non-trivial control on certain non-linear expressions involving limiting and differentiation operations without any differentiability assumptions on $b$. In other words, the higher regularization effect at the level of the flow overrides any need for a commutator-lemma argument, as in the classical proof of \cite{5DiPernaLions89} in the classical setting or \cite{5Beck-et.al.14} and \cite{5FGP10} in the stochastic one, and by implication the need for any differentiability assumptions on the vector field.
	
	\item \textbf{Davie's uniqueness:}
	Our main result at the level of ODE is known as path-by-path uniqueness or as
	uniqueness in the sense of Davie. This is because such a result was proven for
	the first time in Davie's paper \cite{5Davie07} in the setting of Brownian
	motion. This is the strongest type of
	uniqueness that can be proven for (\ref{5SDE}) and should be contrasted with
	the other notion namely that of pathwise uniqueness. Briefly, while Davie's
	uniqueness says something about uniqueness in the space of continuous paths,
	pathwise uniqueness is concerned with an equivalence class at the level of
	stochastic processes defined on some filtered probability space. The recent
	work \cite{5ShaposhnikovWresch20} shows that there is a gap between the two
	notions by constructing several examples. As was mentioned in the
	introduction, and in relation to deterministic problems, a path-by-path result
	is more desirable as it permits a pathwise analysis as we show in Section \ref{5ApplicationsToPDE}.
	
	\item \textbf{The Van Kampen way:} In the original work of Davie \cite{5Davie07},
	the result was obtained, in a direct way, through some highly non-trivial estimates without relying on the existence of strong solutions. Thus it was reasonable to hope that a simpler method which uses the existence of a unique strong solution could be developed. This turned out to be
	the case and in \cite{5Shaposhnikov16} the existence of a stochastic flow of
	homeomorphisms with ``almost'' Lipschitz regularity was used in
	combination with the regularizing effect of the so-called averaging operator
	to recover Davie's result. It turns out that this principle,
	which first appeared in Van Kampen's work \cite{5VanKampen37} and first used in
	the setting of SDE's in the work of \cite{5Shaposhnikov16}, is a very general
	and robust one. This can be witnessed by its applicability to not only SDE's
	driven by non-Markovian dynamics, as in this present work, but was used in the context
	of stochastic PDE's with singular drift coefficients as in \cite{5ButkovskyMytnik19}. In this context, we shall also mention the work \cite{5CatellierGubinelli16}, where the authors study path-by-path solutions of SDE's driven by distributional drift vector fields in Besov spaces and fractional Brownian motion. Here, as was pointed out before, the framework of \cite{5CatellierGubinelli16} cannot be used for vector fields in $L^\infty$.
	
	\item \textbf{Supremum estimate with respect to the derivative of the flow}
	The last ingredient in our approach is based on a uniform control, in
	time, of the spatial regularity of the stochastic flow which substitutes the lack of certain techniques, available     only
         in the Markovian setting. Markovian techniques are
	unavailable in our setting, but by controlling the first derivative
	of the flow we can achieve the desired
	result.

\end{enumerate}

\subsection{Plan of the paper}
The paper is divided into two parts, the first deals with the question of
proving uniqueness, for the SDE interpretation of an ODE perturbed by fBm, in
the sense of Davie.
The proof is based on
\cite{5Shaposhnikov16}\cite{5Shaposhnikov17} and the following two diagrams show the
correspondence between this work and ours

	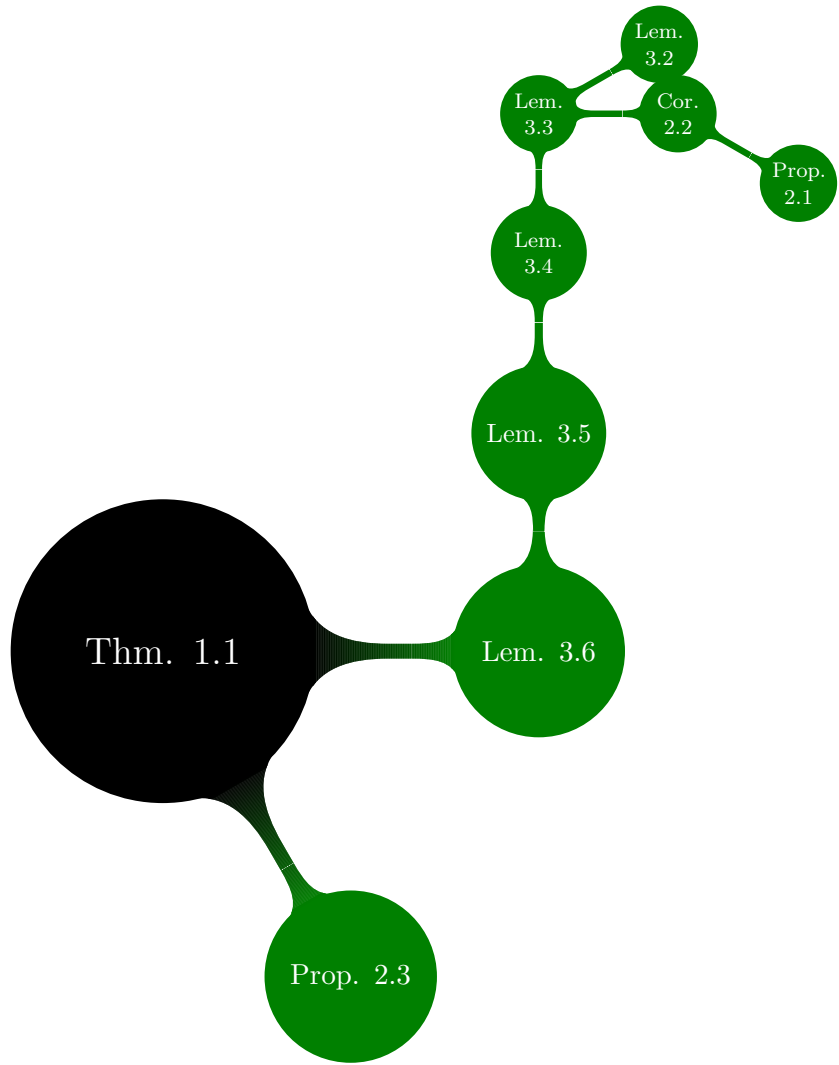
\begin{figure}
		\centering
		\begin{tikzpicture}
		\path[mindmap,concept color=black,text=white]
		node[concept] {Thm. 1.1}
		[clockwise from=0]
		child[concept color=green!50!black] {
		    node[concept] {Lem. 3.6}
			[clockwise from=90]
			child {
    			node[concept] {Lem. 3.5}
    			[clockwise from=90]
    			child { node[concept] {Lem. 3.4} 
    				[clockwise from=90]
    				child { node[concept] {Lem. 3.3} [clockwise from=30]
    					child { node[concept] {Lem. 3.2}[clockwise from=0] }
    					child { node[concept] {Cor. 2.2}[clockwise from=-30]
    						child { 
    							node[concept] {Prop. 2.1}
    						} 
    					}
    				}
    			}
    		}	
		}  
		child[concept color=green!50!black] {
			node[concept] {Prop. 2.3}};
		\end{tikzpicture}
		\caption{Components of the proof in \cite{5Shaposhnikov16}.}
		\label{5figure1}
	\end{figure}

	\begin{figure}
		\centering
		\begin{tikzpicture}
		\path[mindmap,concept color=black,text=white]
		node[concept] {Thm. 1.1}
		[clockwise from=0]
		child[concept color=green!50!black] {
		    node[concept] {(Prop. 2.16)}
			[clockwise from=90]
			child {
    			node[concept] {Lem. 3.5}
    			[clockwise from=90]
    			child { node[concept] {Lem. 3.4} 
    				[clockwise from=90]
    				child { node[concept] {Lem. 3.3} [clockwise from=30]
    					child { node[concept] {Lem. 3.2}[clockwise from=0] }
    					child { node[concept] {(Prop. 2.13)}[clockwise from=-30]
    						child[concept color=blue] { 
    							node[concept] {(Lem. 2.14)}
    						} 
    					}
    				}
    			}
    		}	
		}  
		child[concept color=green!50!black] {
			node[concept] {Prop. 2.3}
			child[concept color=blue] { 
    							node[concept] {Thm. 2.4}
    						} 
    	};
		\end{tikzpicture}
		\caption{Components of the proof in our work. The green components are used as they are, or with a minor modification, from \cite{5Shaposhnikov16}, the blue part is our contribution. The parentheses denote the results as they appear in this work.}
		\label{5figure2}
	\end{figure}
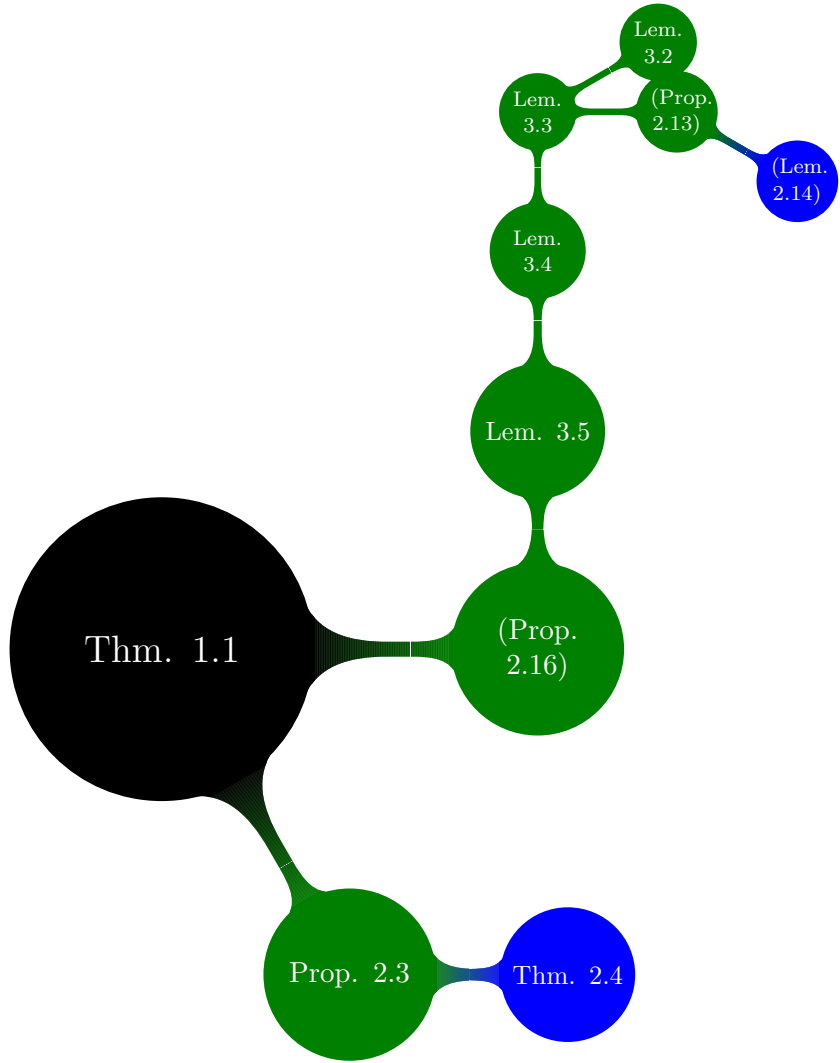

Section 2.1 is concerned with the first component in the Van Kampen argument, namely the regularity of the stochastic flow, and the main theorem of this section is Theorem 2.4. Section 2.2 shows the almost Lipschitz regularity of the averaging operator, which plays the role of the continuous vector field in the original proof of Van Kampen,  and our contribution is Proposition 2.1. The green portion is left as is from \cite{5Shaposhnikov16}.\par
The second part deals with an application of our result to the transport and
continuity equations. This is the content of Section \ref{5ApplicationsToPDE}.

\section{From SDE's to random ODE's}

Consider the following differential equation:

\begin{equation}
X_{t}^{s,x}= x + \int_s^t b(u,X_{u}^{s,x})du + B_{t}^{H} - B_{s}^{H},\quad X_{s}^{s,x}=x,0\leq s \leq t\leq T
\label{5SDE}
\end{equation}

where $B_{t}^{H},0\leq t\leq T$ is a $d-$dimensional fractional Brownian
motion (fBm) with Hurst parameter $H\in (0,\frac{1}{2})$ having the following integral representation \begin{equation}
B_{t}^{H}=\int_{0}^{t}K_{H}(t,s)I_{d\times d}dB_{s}  \label{5RepFractional}
\end{equation}

with respect to a standard Brownian motion (Bm).

Then, using techniques from Malliavin calculus and arguments of a local time
variational calculus kind, as recently developed in the series of works \cite{5BNP19}, \cite{5BLPP18}, \cite{5ABP20}, the following class of vector
fields were investigated
\begin{equation*}
b\in L_{\infty ,\infty }^{1,\infty }:=L^{1}(\mathbb{R}^{d};L^{\infty }([0,T];%
\mathbb{R}^{d}))\cap L^{\infty }(\mathbb{R}^{d};L^{\infty }([0,T];\mathbb{R}%
^{d})).
\end{equation*}%

The main result of \cite{5BNP19} is:

\begin{theorem}
\label{5StrongSolution}Let $s\in\left[0,T\right]$, $b\in L_{\infty ,\infty }^{1,\infty }$ and $k\geq 2$. Then if $H<%
\frac{1}{2(d-1+2k)}$ there exists a unique (global) strong solution\footnote{We recall that strong existence means that, on any given filtered probability space equipped with a given fractional Brownian motion $B^H$, one can find a solution defined on it. Strong uniqueness means that any such two solutions must coincide.} $X_{\cdot
}^{s,x}$ of the SDE (\ref{5SDE}). Moreover, for every $x\in \mathbb{R}^{d},t\in
\lbrack s,T]$ $X_{t}^{s,x}$ is Malliavin differentiable in the direction of
the Brownian motion $B$ in (\ref{5RepFractional}) and $X_{t}^{s,\cdot }$ is
locally Sobolev differentiable $\mu -a.e.$ \par
That is, more precisely,

\[
X_{t}^{s,\cdot }\in \bigcap\limits_{p\geq 2}L^{2}(\Omega ;W^{k,p}(U)) \] 

for bounded and open sets $U\subset \mathbb{R}^{d}$.

\end{theorem}

The aim of this section is to go from (\ref{5StrongSolution}) to a result about uniqueness in the sense of \cite{5Davie07} following the ideas in \cite{5Shaposhnikov16}.

More precisely, the main theorem of this section is the following: 

\begin{theorem}\label{5PathByPath}
	Let $b\in L_{\infty ,\infty }^{1,\infty }$ and $H<
	\frac{1}{2(d+2)}$. Then there exists a measurable set $\Omega^*$ of full mass such that for all $\omega\in\Omega^*$
	
	\begin{equation}
	X_t^x = x + \int_0^t b(s,X_s^x)ds + B^H_t(\omega)
	\label{5IntSDE}
	\end{equation}
	has \textit{one and only one} solution \textit{uniformly in} $x\in\mathbb{R}^d$.
\end{theorem}

\subsection{H\"{o}lder flow}

We start with the first and most important part of our proof, namely the flow. The main result in this section is the strong control in point 4 of the following proposition. This is essentially the same as in (\cite{5Shaposhnikov17}):
\begin{proposition}\label{5Holder flow}
Let $b\in L_{\infty ,\infty }^{1,\infty }$, $H<
\frac{1}{2(d+2)}$ and $(\Omega,\mathcal{F},\mathbf{P})$ be a probability space. Then there exists a H\"{o}lder flow of solutions $\phi_{s,t}(x)$ to equation (\ref{5SDE}).\par
More precisely, for any filtered probability space $(\Omega,\mathcal{F},\{\mathcal{F}_t\},\mathbf{P})$ and fBm $B^H$ generating $\{\mathcal{F}_t\}$, there exist a mapping
\[
(s,t,x,\omega)\mapsto \phi_{s,t}(x,\omega)\quad\textit{with } (s,t)\in\Delta_{0,T}^2,x\in\mathbb{R}^d,\omega\in\Omega
,\]
where \[\Delta_{0,T}^2 := \left\{(s,t): 0\leq s\leq t\leq T \right\} ,\]
such that the following properties hold:
\begin{enumerate}
\item $\forall x\in\mathbb{R}^d$, $\phi_{s,t}(x)$ is a continuous $ \mathcal{F}_{s,t}$ adapted solution to (\ref{5IntSDE}) with $\phi_{s,s}(x) = Id$. $ \mathcal{F}_{s,t}$ stands for the filtration generated by the increments of the driving process.

\item $\mathbf{P}$-almost surely the mapping $x\mapsto\phi_{s,t}(x)$ is a homeomorphism.

\item $\mathbf{P}$-almost surely $\forall x\in\mathbb{R}^d$ and $0\leq s\leq u \leq t \leq T$ we have $\phi_{s,t}(x) = \phi_{u,t}(\phi_{s,u}(x))$.

\item For any $\alpha\in\left(0,1\right)$, $\eta>0$, $N>0$ and a given increasing sequence $S:=\left\{S_n\right\}_{n\geq0}$ of finite sets with $Card(S_n)\leq 2^{n\eta}$, there exists a set $\Omega^{'}$ with mass $1$ such that for any $s\in S_n$ $x,y\in\mathbb{R}^d$ with $|x|,|y|<N,|x-y|\leq 2^{-n}$ and all $t\in[s,T]$
\[\|\phi_{s,t}(x)-\phi_{s,t}(y)\|\leq C(\alpha,N,T,S,\omega)\|x-y\|^\alpha\]

\end{enumerate}
\end{proposition}

Points $(1),(2),(3)$ follow directly from the technical Proposition \ref{5FlowProperty}, Remark \ref{5remark:uniformFlow} and Proposition \ref{5prop:homeomorphism} in the Appendix. Point $(4)$ in contrast is new and technically demanding, its proof is based on the following key result:

\begin{theorem}
\label{5MainEstimate}
\bigskip Let $H<\frac{1}{2(d+2)}$ and let $X_{\cdot }^{s,x}$ be the unique
strong solution to the SDE%
\begin{equation}
dX_{t}^{s,x}=b(t,X_{t}^{s,x})dt+dB_{t}^{H},X_{s}^{x}=x,0\leq t\leq T
\label{11SDE}
\end{equation}%
for $b\in L_{\infty ,\infty }^{1,\infty }$. Let $K$ be a compact cube in $%
\mathbb{R}^{d}$ and $r\in \mathbb{N}.$ Then for all $s\in \left[ 0,T\right) $
and $x,y\in K$:%
\begin{equation*}
E\left[ \sup_{t\in \left[ s,T\right] }\left\Vert
X_{t}^{s,x}-X_{t}^{s,y}\right\Vert ^{2^{r}}\right] \leq
C_{r,d,H,T}(K)\left\Vert x-y\right\Vert ^{2^{r}}.
\end{equation*}

\end{theorem}

\begin{proof}[Proof of Proposition \ref{5Holder flow} (4)] (\cite{5Shaposhnikov17})
	Let $K:= [-N, N]^d$ and define the following random mapping $J$ from $\Omega \times [0, T] \times [-N, N]^d$
	to the Banach space $\mathcal{C}([0, T], \mathbb{R}^d)$ equipped with the supremum norm as follows:
	\[
	\mathtt{J}(\omega, s, x)(t) := X_{\min(s + t, T)}^{s,x}(\omega).
	\]
	The joint continuity of $X_{s, t}^x$ with respect to $t, x$ immediately implies the mapping $\mathtt{J}$
	is continuous.
	Next, the estimate 
	\[
	E[\sup_{t\in \lbrack s,T]}\left\Vert X_{t}^{s,x}-X_{t}^{s,y}\right\Vert
	^{2^{r}}]\leq C_{d,H,T}(K)\left\Vert x-y\right\Vert ^{2^{r}}\text{.}
	\]
	can be written as
	\[
	\sup_{s \in [0, T]} \mathbb{E} \|\mathtt{J}(\omega, s, x) - \mathtt{J}(\omega, s, y)\|^{2^{r}} \leq C_{d,H,T}(K)\left\Vert x-y\right\Vert ^{2^{r}}
	\]
	For any $\alpha \in (0, 1)$ and $\eta > 0$ one can find
	sufficiently large $r > 0$ such that
	\[
	\alpha < \frac{2^{r} - 1 - d}{2^{r}}, \  \eta < 2^{r} - 1 - d - \alpha 2^{r}
	.\]
	So now it is easy to complete the proof applying Lemma \ref{5ExtKolmogorov}.
\end{proof}

In order to prove the previous theorem we need a series of preparatory lemmas and theorems. To fulfill the requirements of this result, we make use of the flow's derivative and for the method we use to achieve that we need to push the Hurst parameter lower than just $\frac{1}{2(d+2)}$. 

To achieve the desired control over the derivative we employ, again, the compactness criterion $L^2$ from \cite{5DMN92}. Lemma 2.6 tells us that the sequence of the derivatives of solutions to the SDE when the drift is smooth satisfy the conditions of the compactness criterion when $H<\frac{1}{2(d+3)}$. This the content of Section \ref{5appFlowDeriv}. Section \ref{5relativeCompactnessDerivative} shows the relative compactness of the mentioned sequence. Section \ref{5identifyGoodLimits} yields "good" versions of the limit. Section \ref{5mainSupEstimate} is about the proof of the main result of this section namely Theorem \ref{5MainEstimate}.

\subsubsection{Controlling the derivative of the sequence of approximating flows}
\label{5appFlowDeriv}
In this part we provide control over the derivative of the solution to (\ref{52.2}) using a compactness criterion of \cite{5DMN92} i.e. Corollary \ref{5VI_compactcrit}.

 Before starting with our program, we need the following technical lemma

 \begin{lemma}\label{5JointHolderContinuity}
	Let $b\in L_{\infty ,\infty }^{1,\infty }$ and $H<%
	\frac{1}{2(d+2)}$. Then there exists, for all $0\leq s< T$ and $x\in\mathbb{R}^d$, a unique strong solution $X_{\cdot
	}^{s,x}$ of the SDE (\ref{5SDE}) such that for all $0\leq s< T$ the mapping 
\[((t,x)\mapsto X_{t
}^{s,x})\]  on $[s,T]\times\mathbb{R}^d$ is locally H\"{o}lder continuous a.e.

\end{lemma}

\begin{proof}
	
	Let $%
	\{b_{n}\}_{n\geq 1}\subset C_{c}^{\infty }([0,T]\times \mathbb{R}^{d})$,
	such that 
	\begin{equation*}
	b_{n}(t,x)\longrightarrow b(t,x)
	\end{equation*}%
	as $n\rightarrow \infty $ for a.e. $(t,x)\in \lbrack 0,T]\times \mathbb{R}%
	^{d}$ with $\sup_{n\geq 0}\Vert b_{n}\Vert _{L_{\infty }^{1}}<\infty $ and
	such that $|b_{n}(t,x)|\leq M<\infty $, $n\geq 1$ a.e. for some constant $M$%
	. Assume that $X_{\cdot }^{s,x,n}$ is the unique strong solution to the SDE 
	\begin{equation*}
	dX_{t}^{s,x,n}=b_{n}(t,X_{t}^{s,x,n})du+dB_{t}^{H},\,\,0\leq t\leq
	T,\,\,\,X_{s}^{s,x,n}=x\in \mathbb{R}^{d}\,.
	\end{equation*}%
	 Then for $t_1, t_2\in [s,T]$ and $x,y\in\mathbb{R}^d$, such that $t_1\leq t_2$, we have that
	\[X_{t_{1}}^{s,x,n}-X_{t_{2}}^{s,y,n} = X_{t_{1}}^{s,x,n}-X_{t_{2}}^{s,x,n} + X_{t_{2}}^{s,x,n}-X_{t_{2}}^{s,y,n} = J_1 + J_2
	\]
	where 
	\[J_1 := X_{t_{1}}^{s,x,n}-X_{t_{2}}^{s,x,n}\]
	
	and 
	
	\[J_2 := X_{t_{2}}^{s,x,n}-X_{t_{2}}^{s,y,n}\]

	We see that for $p>1$ that
	
	\[E[\Vert J_1\Vert^p]\leq 2^p\left( M^p\vert t_1-t_2\vert^p + E[\Vert B^H_{t_2-t_1}\Vert^p]\right) = C(s,T,p,H,M)\vert t_1-t_2\vert^{Hp}, \]
	
	where $\Vert\cdot\Vert$ is the Euclidean norm and $C(\cdot)$ is a constant that depends on its argument.
	On the other hand, by using the fundamental theorem of calculus with respect to the
	one dimensional arguments of the entries of $
	X_{t}^{s,x,n}=\left(X_{t}^{s,x,n,j}\right)
	_{1\leq j\leq d}$, we find that 
	\begin{equation*}
	\begin{split}
	E[\left\Vert J_2\right\Vert ^{p}] = &E[\left\Vert X_{t_{2}}^{s,x,n}-X_{t_{2}}^{s,x,n}\right\Vert ^{p}] \\
	=&(\sum_{j=1}^{d}E[(%
	X_{t_{2}}^{s,x_{1},...,x_{d},n,j}-%
	X_{t_{2}}^{s,y_{1},...,y_{d},n,j})^{2}])^{\frac{p}{2}} \\
	\leq (&C(d)\sum_{j=1}^{d}\{E[(%
	X_{t_{2}}^{s,x_{1},...,x_{d},n,j}-%
	X_{t_{2}}^{s,y_{1},x_{2}...,x_{d},n,j})^{2}] \\
	&+E[(%
	X_{t_{2}}^{s,y_{1},x_{2},x_{3}...,x_{d},n,j}-%
	X_{t_{2}}^{s,y_{1},y_{2},x_{3}...,x_{d},n,j})^{2}] \\
	&+...+E[(
	X_{t_{2}}^{s,y_{1},y_{2},...,y_{d-1},x_{d},n,j}-X_{t_{2}}^{s,y_{1},y_{2},y_{3}...,y_{d},n,j})^{2}]\} )^{\frac{p}{2}}\\
	\leq &(C(d,K)\sum_{j=1}^{d}\{\left\vert x_{1}-y_{1}\right\vert
	^{2}\sup_{x\in K}E[(\frac{\partial}{\partial x_{1}}%
	X_{t_{2}}^{s,x_{1},...,x_{d},n,j})^{2}] \\
	&+\left\vert x_{2}-y_{2}\right\vert ^{2}\sup_{x\in K}E[(\frac{\partial
	}{\partial x_{2}}X_{t_{2}}^{s,x_{1},...,x_{d},n,j})^{2}] \\
	&+...+\left\vert x_{d}-y_{d}\right\vert ^{2}\sup_{x\in K}E[(\frac{\partial
		}{\partial x_{d}}X_{t_{2}}^{s,x_{1},...,x_{d},n,j})^{2}]\})^{\frac{p}{2}} \\
	\leq &C(d,K)\left\Vert x-y\right\Vert ^{p}(\sup_{n\in \mathbb{N},t\in
		\lbrack s,T],x\in K}E[\left\Vert \frac{\partial }{\partial x}%
	X_{t}^{s,x,n}\right\Vert _{\mathbb{R}^{d\times d}}^{2}])^{\frac{p}{2}}.
	\end{split}
	\end{equation*}
	
	where $K$ is a compact cube in $\mathbb{R}^d.$
	
	Using Lemma 5.10 in the Appendix we have
	
	\[\sup_{n\in \mathbb{N},t\in
		\lbrack s,T],x\in \mathbb{R}^d}E\Big[\left\Vert \frac{\partial }{\partial x}%
	X_{t}^{s,x,n}\right\Vert _{\mathbb{R}^{d\times d}}^{2}\Big] < \infty\]
	
	Hence we get
	\[E\Big[\left\Vert X_{t_{1}}^{s,x,n}-X_{t_{2}}^{s,y,n}\right\Vert ^{p}\Big] \leq C(d,s,T,H,K,p,\Vert b\Vert) \left( \left\Vert x-y\right\Vert ^{Hp} + \left\Vert t_1-t_2\right\Vert ^{Hp}\right)\] for all $x,y\in K$, $t_1,t_2\in [s,T]$ and $n\geq 1$.
	
	Further, using Corollary 4.8 in \cite{5BNP19} we have that \[ X_t^{s,x,n}\rightarrow X_t^{s,x}\text{ in } L^2(\Omega)\]
	as $n\rightarrow\infty$ for all $s,t,x$.
	
	Hence using Fatou's lemma we get
	\begin{equation*}
	    \begin{split}
	        E\Big[\left\Vert X_{t_{1}}^{s,x}-X_{t_{2}}^{s,y}\right\Vert ^{p}\Big] &\leq
		\lim_{j\rightarrow\infty } E\Big[\left\Vert X_{t_{1}}^{s,x,n(j)}-X_{t_{2}}^{s,y,n(j)}\right\Vert ^{p}\Big] \\
		&\leq C(d,s,T,H,K,p,\Vert b\Vert)\left( \left\Vert x-y\right\Vert ^{Hp} + \left\Vert t_1-t_2\right\Vert ^{Hp}\right)
	    \end{split}
	\end{equation*}

	for all $x,y\in K$ and $t_1,t_2\in [s,T].$\par
	To conclude, we apply Kolmogorov's continuity theorem (for $Hp\geq d+\beta$ and $\beta >0$) to the stochastic field $(t,x)\mapsto X_t^{s,x}$, which gives a locally H\"{o}lder continuous version and which itself is a unique strong solution to the SDE (\ref{5SDE}).

\end{proof}

\begin{lemma}
\label{5RelativeCompactness}\bigskip Let $H<\frac{1}{2(d+3)}$ and $b\in
C_{c}([0,T]\times \mathbb{R}^{d};\mathbb{R}^{d})$. Assume $X_{\cdot }^{x}$
is the unique strong solution to the SDE%
\begin{equation*}
dX_{t}^{s,x}=b(t,X_{t}^{s,x})dt+dB_{t}^{H},X_{s}^{s,x}=x,0\leq s\leq  t\leq T
\end{equation*}%
and let $\frac{\partial }{\partial x}X_{t}^{s,x}$ be the Fr\'{e}chet
derivative of $X_{t}^{s,x}$. Then there exists a $\beta \in (0,1/2)$ such
that for all $x\in \mathbb{R}^{d}$, $0\leq s<t\leq T$%
\begin{equation}
\int_{0}^{t}\int_{0}^{t}\frac{\left\Vert D_{\theta }\frac{\partial }{%
\partial x}X_{t}^{s,x}-D_{\theta ^{\prime }}\frac{\partial }{\partial x}%
X_{t}^{s,x}\right\Vert _{L^{2}(\Omega ;\mathbb{R}^{d}\otimes \mathbb{R}%
^{d}\otimes \mathbb{R}^{d})}^{2}}{\left\vert \theta -\theta ^{\prime
}\right\vert ^{1+2\beta }}d\theta ^{\prime }d\theta \leq
C_{H,d,T}(\left\Vert b\right\Vert _{L_{\infty }^{\infty }},\left\Vert
b\right\Vert _{L_{\infty }^{1}})  \label{5E1}
\end{equation}%
and%
\begin{equation}
\left\Vert D_{\cdot }\frac{\partial }{\partial x}X_{t}^{s,x}\right\Vert
_{L^{2}([0,t]\times \Omega ;\mathbb{R}^{d}\otimes \mathbb{R}^{d}\otimes 
\mathbb{R}^{d})}\leq C_{H,d,T}(\left\Vert b\right\Vert _{L_{\infty }^{\infty
}},\left\Vert b\right\Vert _{L_{\infty }^{1}}),  \label{5E2}
\end{equation}%
where $C_{H,d,T}:[0,\infty )\times \lbrack 0,\infty )\longrightarrow \lbrack
0,\infty )$ is a continuous function, which only depends on $H,d$ and $T$
(and not on $x$, $s$ or $b$).
\end{lemma}

\begin{proof}
 Assume without loss of generality that $s=0$. Since $b$ is a
smooth vector field, we know that%
\begin{equation}
\begin{split}
\frac{\partial }{\partial x}X_{t}^{x} =& I_{d\times
d}+\int_{0}^{t}Db(u,X_{u}^{x})\frac{\partial }{\partial x}X_{u}^{x}du
\label{5DerivativeFlow} \\
=&\frac{\partial }{\partial x}X_{\theta }^{x}+\int_{\theta
}^{t}Db(u,X_{u}^{x})\frac{\partial }{\partial x}X_{u}^{x}du  \notag
\end{split}
\end{equation}

for all $0<\theta <t$, where $Db:\mathbb{R}^{d}\longrightarrow L(\mathbb{R}%
^{d},\mathbb{R}^{d})$ is the derivative of $b$ with respect to the space
variable.

On the other hand, using the chain rule for Malliavin derivatives combined
with Lemma 1.2.3 in \cite{5Nualart10}, we find that%

\begin{equation}
\begin{split}
D_{s}\frac{\partial }{\partial x}X_{t}^{x} = \int_{s}^{t}D^{2}b(u,X_{u}^{x})D_{s}X_{u}^{x}\frac{\partial }{\partial x}
X_{u}^{x}+Db(u,X_{u}^{x})D_{s}\frac{\partial }{\partial x}X_{u}^{x}du
\end{split}
\end{equation}

in $L^{2}([0,t]\times \Omega )$ for all $t$. Fix $t$. Then for almost all $%
\theta ^{\prime }<\theta $ with $0<\theta ^{\prime }<\theta <t$ we have%

\begin{equation}
\begin{split}
D_{\theta ^{\prime }}\frac{\partial }{\partial x}X_{t}^{x}-D_{\theta }
\frac{\partial }{\partial x}X_{t}^{x} =&\int_{\theta^{\prime }}^{\theta }D^{2}b(u,X_{u}^{x})D_{s}X_{u}^{x}\frac{\partial}{\partial x} X_{u}^{x} + Db(u,X_{u}^{x})D_{s}\frac{\partial }{\partial x} X_{u}^{x}du  \notag \\
&+\int_{\theta }^{t}D^{2}b(u,X_{u}^{x})(D_{\theta ^{\prime
}}X_{u}^{x}-D_{\theta }X_{u}^{x})\frac{\partial }{\partial x}X_{u}^{x} 
\notag \\
&+ Db(u,X_{u}^{x})(D_{\theta ^{\prime }}\frac{\partial }{\partial x}
X_{u}^{x}-D_{\theta }\frac{\partial }{\partial x}X_{u}^{x})du  \notag \\
=&	D_{\theta ^{\prime }}\frac{\partial }{\partial x}X_{\theta}^{x}+\int_{\theta }^{t}D^{2}b(u,X_{u}^{x})(D_{\theta ^{\prime}} X_{u}^{x} - D_{\theta }X_{u}^{x})\frac{\partial }{\partial x} X_{u}^{x}
\label{5Picard2} \\
&+ Db(u,X_{u}^{x})(D_{\theta ^{\prime }}\frac{\partial }{\partial x}
X_{u}^{x} - D_{\theta }\frac{\partial }{\partial x}X_{u}^{x})du  \notag
\end{split}
\end{equation}

Using Picard iteration in connection with (\ref{5DerivativeFlow}), we see that%

\begin{equation}
\begin{split}
\frac{\partial }{\partial x}X_{t}^{x} =&\frac{\partial }{\partial x}%
X_{\theta }^{x}+\sum_{m\geq 1}\int_{\Delta _{\theta
,t}^{m}}Db(u_{1},X_{u_{1}}^{x})...Db(u_{m},X_{u_{m}}^{x})\frac{\partial }{%
\partial x}X_{\theta }^{x}du_{m}...du_{1}  \label{5DX} \\
=&(I_{d\times d}+\sum_{m\geq 1}\int_{\Delta _{\theta
,t}^{m}}Db(u_{1},X_{u_{1}}^{x})...Db(u_{m},X_{u_{m}}^{x})du_{m}...du_{1})%
\frac{\partial }{\partial x}X_{\theta }^{x}  \notag
\end{split}
\end{equation}

in $L^{2}(\Omega ;\mathcal{C}([\theta ,T];\mathbb{R}^{d\times d})$ for all $x$ and $%
\theta $, where%

\begin{equation*}
\Delta _{\theta ,t}^{m}=\{(u_{m},...u_{1})\in \lbrack 0,T]^{m}:\theta
<u_{m}<...<u_{1}<t\}.
\end{equation*}%
Further, using once more the chain rule for the Malliavin derivative, see 
\cite[Proposition 1.2.3]{5Nualart10}, we find that for almost all $\theta \in
\lbrack 0,T]$ 

\begin{equation*}
D_{\theta }X_{t}^{x}=K_{H}(t,\theta )I_{d\times d}+\int_{\theta
}^{t}Db(u,X_{u}^{x})D_{\theta }X_{u}^{n}du,
\end{equation*}%
in $L^{2}([\theta ,T]\times \Omega )$ holds. Thus we have for almost all $
0<\theta ^{\prime }<\theta $ that

\begin{equation*}
\begin{split}
D_{\theta ^{\prime }}X_{t}^{x}- D_{\theta }X_{t}^{x} =& K_{H}(t,\theta
^{\prime })I_{d\times d}-K_{H}(t,\theta )I_{d\times d} \\ &+ \int_{\theta ^{\prime }}^{t}Db(u,X_{u}^{x})D_{\theta ^{\prime
}}X_{u}^{x}du-\int_{\theta }^{t}Db(u,X_{u}^{x})D_{\theta }X_{u}^{x}du \\
=& K_{H}(t,\theta ^{\prime })I_{d\times d}-K_{H}(t,\theta )I_{d\times d} \\
& +\int_{\theta ^{\prime }}^{\theta }Db(u,X_{u}^{x})D_{\theta ^{\prime
}}X_{u}^{x}du+\int_{\theta }^{t}Db(u,X_{u}^{x})(D_{\theta ^{\prime
}}X_{u}^{x}-D_{\theta }X_{u}^{x})du \\
=& K_{H}(t,\theta ^{\prime })I_{d\times d}-K_{H}(t,\theta )I_{d\times
d}+D_{\theta ^{\prime }}X_{\theta }^{x}-K_{H}(\theta ,\theta ^{\prime
})I_{d\times d} \\
& +\int_{\theta }^{t}Db(u,X_{u}^{x})(D_{\theta ^{\prime
}}X_{u}^{x}-D_{\theta }X_{u}^{x})du
\end{split}
\end{equation*}

in $L^{2}([\theta ,T]\times \Omega )$.\par

By applying again Picard iteration to
the above equation, we may write

\begin{equation*}
\begin{split}
D_{\theta ^{\prime }}X_{t}^{x}-& D_{\theta }X_{t}^{x}=K_{H}(t,\theta
^{\prime })I_{d\times d}-K_{H}(t,\theta )I_{d\times d} \\
& +\sum_{m=1}^{\infty }\int_{\Delta _{\theta
,t}^{m}}\prod_{j=1}^{m}Db(s_{j},X_{s_{j}}^{x})\left( K_{H}(s_{m},\theta
^{\prime })I_{d\times d}-K_{H}(s_{m},\theta )I_{d\times d}\right)
ds_{m}\cdots ds_{1} \\
& +\left( I_{d\times d}+\sum_{m=1}^{\infty }\int_{\Delta _{\theta
,t}^{m}}\prod_{j=1}^{m}Db(s_{j},X_{s_{j}}^{x})ds_{m}\cdots ds_{1}\right)
\left( D_{\theta ^{\prime }}X_{\theta }^{x}-K_{H}(\theta ,\theta ^{\prime
})I_{d\times d}\right) .
\end{split}
\end{equation*}

We also observe that one may write

\begin{equation*}
D_{\theta ^{\prime }}X_{\theta }^{x}-K_{H}(\theta ,\theta ^{\prime
})I_{d\times d}=\sum_{m\geq 1}\int_{\Delta _{\theta ^{\prime },\theta
}^{m}}\prod_{j=1}^{m}Db(s_{j},X_{s_{j}}^{x})(K_{H}(s_{m},\theta ^{\prime
})I_{d\times d})\,ds_{m}\cdots ds_{1}.
\end{equation*}%

In summary, we can write 

\begin{equation*}
D_{\theta ^{\prime }}X_{t}^{x}-D_{\theta }X_{t}^{x}=I_{1}(\theta ^{\prime
},\theta )+I_{2}(\theta ^{\prime },\theta )+I_{3}(\theta ^{\prime },\theta ),
\end{equation*}%

where 

\begin{equation*}
\begin{split}
I_{1}(\theta ^{\prime },\theta ):=& K_{H}(t,\theta ^{\prime })I_{d\times
d}-K_{H}(t,\theta )I_{d\times d} \\
I_{2}(\theta ^{\prime },\theta ):=& \sum_{m\geq 1}\int_{\Delta _{\theta
,t}^{m}}\prod_{j=1}^{m}Db(s_{j},X_{s_{j}}^{x})\left( K_{H}(s_{m},\theta
^{\prime })I_{d\times d}-K_{H}(s_{m},\theta )I_{d\times d}\right)
ds_{m}\cdots ds_{1} \\
I_{3}(\theta ^{\prime },\theta ):=& \left( I_{d\times d}+\sum_{m\geq
1}\int_{\Delta _{\theta
,t}^{m}}\prod_{j=1}^{m}Db(s_{j},X_{s_{j}}^{x})ds_{m}\cdots ds_{1}\right) \\
& \times \left( \sum_{m\geq 1}\int_{\Delta _{\theta ^{\prime },\theta
}^{m}}\prod_{j=1}^{m}Db(s_{j},X_{s_{j}}^{n})(K_{H}(s_{m},\theta ^{\prime
})I_{d\times d})ds_{m}\cdots ds_{1}.\right) .
\end{split}
\end{equation*}

Once more we can use Picard iteration applied to (\ref{5Picard2}) and obtain
for almost all $\theta ^{\prime },\theta $ with $\theta ^{\prime }<\theta $
that%
\begin{equation}
\begin{split}
&D_{\theta ^{\prime }}\frac{\partial }{\partial x}X_{t}^{x}-D_{\theta }%
\frac{\partial }{\partial x}X_{t}^{x}  \notag \\
=&D_{\theta ^{\prime }}\frac{\partial }{\partial x}X_{\theta
}^{x}+\int_{\theta }^{t}D^{2}b(u,X_{u}^{x})(D_{\theta ^{\prime
}}X_{u}^{x}-D_{\theta }X_{u}^{x})\frac{\partial }{\partial x}X_{u}^{x}du 
\notag \\
&+\sum_{m\geq 1}\int_{\Delta _{\theta
,t}^{m}}(\prod_{j=1}^{m}Db(s_{j},X_{s_{j}}^{x}))  \notag \\
&\times (D_{\theta ^{\prime }}\frac{\partial }{\partial x}X_{\theta
}^{x}+\int_{\theta }^{s_{m}}D^{2}b(u,X_{u}^{x})(D_{\theta ^{\prime
}}X_{u}^{x}-D_{\theta }X_{u}^{x})\frac{\partial }{\partial x}%
X_{u}^{x}du)ds_{m}\cdots ds_{1}  \notag \\
=&A_{1}+...+A_{8}  \label{5Difference}
\end{split}
\end{equation}
in $L^{2}([\theta ,T]\times \Omega )$, where%

\begin{equation}
\begin{split}
A_{1} :&=D_{\theta ^{\prime }}\frac{\partial }{\partial x}X_{\theta }^{x},
\\
A_{2} :&=\int_{\theta }^{t}D^{2}b(u,X_{u}^{x})I_{1}(\theta ^{\prime
},\theta )\frac{\partial }{\partial x}X_{u}^{x}du, \\
A_{3} :&=\int_{\theta }^{t}D^{2}b(u,X_{u}^{x})I_{2}(\theta ^{\prime
},\theta )\frac{\partial }{\partial x}X_{u}^{x}du, \\
A_{4} :&=\int_{\theta }^{t}D^{2}b(u,X_{u}^{x})I_{3}(\theta ^{\prime
},\theta )\frac{\partial }{\partial x}X_{u}^{x}du, \\
A_{5} :&=\sum_{m\geq 1}\int_{\Delta _{\theta
,t}^{m}}\prod_{j=1}^{m}Db(s_{j},X_{s_{j}}^{x})ds_{m}\cdots ds_{1}A_{1}, \\
A_{6} :&=\sum_{m\geq 1}\int_{\Delta _{\theta
,t}^{m}}(\prod_{j=1}^{m}Db(s_{j},X_{s_{j}}^{x}))A_{2}(s_{m})ds_{m}\cdots
ds_{1} \\
A_{7} :&=\sum_{m\geq 1}\int_{\Delta _{\theta
,t}^{m}}(\prod_{j=1}^{m}Db(s_{j},X_{s_{j}}^{x}))A_{3}(s_{m})ds_{m}\cdots
ds_{1}, \\
A_{8} :&=\sum_{m\geq 1}\int_{\Delta _{\theta
,t}^{m}}(\prod_{j=1}^{m}Db(s_{j},X_{s_{j}}^{x}))A_{4}(s_{m})ds_{m}\cdots
ds_{1}.
\end{split}
\end{equation}

Using Fubini's theorem and the fact that the terms $A_{1},...,A_{8}$ are
continuous in time $t>\theta $ in $L^{2}(\Omega )$, we also see that for all 
$t\in (0,T]$ the above relation holds in $L^{2}(\Omega )$ for almost all $%
\theta ^{\prime },\theta $ with $\theta ^{\prime }<\theta <t$.

Let us first consider the most difficult term $A_{7}$. By applying dominated
convergence in connection with the term $A_{7}$ and (\ref{5DX}) we get that%
\begin{equation}
\begin{split}
A_{7} =&\sum_{m\geq 1}\int_{\Delta _{\theta
,t}^{m}}(\prod_{j=1}^{m}Db(s_{j},X_{s_{j}}^{x}))\int_{\theta
}^{s_{m}}D^{2}b(u,X_{u}^{x})I_{2}(\theta ^{\prime },\theta )\frac{\partial }{%
\partial x}X_{u}^{x}duds_{m}\cdots ds_{1} \\
=&\sum_{m\geq 1}\int_{\Delta _{\theta
,t}^{m}}(\prod_{j=1}^{m}Db(s_{j},X_{s_{j}}^{x}))\int_{\theta
}^{s_{m}}D^{2}b(u,X_{u}^{x})I_{2}(\theta ^{\prime },\theta )duds_{m}\cdots
ds_{1}\frac{\partial }{\partial x}X_{\theta }^{x} \\
&+\sum_{m\geq 1}\sum_{m_{1}\geq 1}\int_{\Delta _{\theta
,t}^{m}}(\prod_{j=1}^{m}Db(s_{j},X_{s_{j}}^{x}))\int_{\theta
}^{s_{m}}D^{2}b(u,X_{u}^{x})I_{2}(\theta ^{\prime },\theta ) \\
&\times \int_{\Delta _{\theta
,u}^{m_{1}}}Db(u_{1},X_{u_{1}}^{x})...Db(u_{m_{1}},X_{u_{m_{1}}}^{x})du_{m_{1}}...du_{1}duds_{m}\cdots ds_{1}%
\frac{\partial }{\partial x}X_{\theta }^{x} \\
=&\sum_{m\geq 1}\sum_{m_{2}\geq 1}\int_{\Delta _{\theta
,t}^{m}}(\prod_{j=1}^{m}Db(s_{j},X_{s_{j}}^{x}))\int_{\theta
}^{s_{m}}D^{2}b(u,X_{u}^{x}) \\
&\times \int_{\Delta _{\theta
,u}^{m_{2}}}\prod_{j=1}^{m_{2}}Db(z_{j},X_{z_{j}}^{x})\left(
K_{H}(z_{m_{2}},\theta ^{\prime })I_{d\times d}-K_{H}(z_{m_{2}},\theta
)I_{d\times d}\right) dz_{m_{2}}\cdots dz_{1} \\
&duds_{m}\cdots ds_{1}\frac{\partial }{\partial x}X_{\theta }^{x} \\
&+\sum_{m\geq 1}\sum_{m_{1}\geq 1}\sum_{m_{2}\geq 1}\int_{\Delta _{\theta
,t}^{m}}(\prod_{j=1}^{m}Db(s_{j},X_{s_{j}}^{x}))\int_{\theta
}^{s_{m}}D^{2}b(u,X_{u}^{x}) \\
&\times \int_{\Delta _{\theta
,u}^{m_{2}}}\prod_{j=1}^{m_{2}}Db(z_{j},X_{z_{j}}^{x})\left(
K_{H}(z_{m_{2}},\theta ^{\prime })I_{d\times d}-K_{H}(z_{m_{2}},\theta
)I_{d\times d}\right) dz_{m_{2}}\cdots dz_{1} \\
&\times \int_{\Delta _{\theta
,u}^{m_{1}}}Db(u_{1},X_{u_{1}}^{x})...Db(u_{m_{1}},X_{u_{m_{1}}}^{x})du_{m_{1}}...du_{1}duds_{m}\cdots ds_{1}%
\frac{\partial }{\partial x}X_{\theta }^{x} \\
=&A_{9}+A_{10}
\end{split}
\end{equation}
in $L^{2}(\Omega )$ for almost all $\theta ^{\prime },\theta $ with $\theta
^{\prime }<\theta <t$, where $A_{9}$ is the first and $A_{10}$ the
second term in the last equality. Let $A_{11}$ be the first factor in $%
A_{10} $ and%
\begin{equation}
\begin{split}
&L_{m_{1},m_{2}} \\
=&\int_{\theta }^{s_{m}}D^{2}b(u,X_{u}^{x})\int_{\Delta _{\theta
,u}^{m_{2}}}\prod_{j=1}^{m_{2}}Db(z_{j},X_{z_{j}}^{x})\left(
K_{H}(z_{m_{2}},\theta ^{\prime })I_{d\times d}-K_{H}(z_{m_{2}},\theta
)I_{d\times d}\right) dz_{m_{2}}\cdots dz_{1} \\
&\times \int_{\Delta _{\theta
,u}^{m_{1}}}Db(u_{1},X_{u_{1}}^{x})...Db(u_{m_{1}},X_{u_{m_{1}}}^{x})du_{m_{1}}...du_{1}du
\end{split}
\end{equation}
We can now apply the shuffling relation (\ref{5shuffleIntegral}) to the term $L_{m_{1},m_{2}}$ and obtain that%
\begin{equation}\label{5I2}
L_{m_{1},m_{2}}=\int_{\Delta _{\theta ,s_{m}}^{m_{1}+m_{2}+1}}\mathcal{H}%
_{m_{1},m_{2}}^{X}(u)du_{m_{1}+m_{2}+1}...du_{1}  
\end{equation}%
for $u=(u_{1},...,u_{m_{1}+m_{2}+1}),$ where the integrand $\mathcal{H}%
_{m_{1},m_{2}}^{X}(u)\in \mathbb{R}^{d}\otimes \mathbb{R}^{d}\otimes \mathbb{%
R}^{d}$ has entries given by sums of at most $C(d)^{m_{1}+m_{2}+1}$
summands, which are products of length $m_{1}+m_{2}+1$ of functions
belonging to the class%
\begin{equation*}
\left\{ 
\begin{array}{c}
\frac{\partial ^{j}}{\partial x_{l_{1}}\partial x_{l_{j}}}%
b^{(i)}(u,X_{u}^{x})\left( K_{H}(u,\theta ^{\prime })-K_{H}(u,\theta
)\right) ^{\varepsilon }, \\ 
j=1,2,l_{1},l_{2},i=1,...,d,\varepsilon \in \{0,1\}%
\end{array}%
\right\} .
\end{equation*}%
Here it is important to mention that second order derivatives of functions
as well as the factor $\left( K_{H}(u,\theta ^{\prime })-K_{H}(u,\theta
)\right) $ in those products of functions on $\Delta _{\theta
,s_{m}}^{m_{1}+m_{2}+1}$ in (\ref{5I2}) only appear once. Thus the absolute
value of the multi-index $\alpha $ with respect to the total order of
derivatives of those products of functions in connection with Lemma %
\ref{5OrderDerivatives} in the Appendix is given by%
\begin{equation}
\left\vert \alpha \right\vert =m_{1}+m_{2}+2.
\end{equation}%
By definition we have that%
\begin{equation}
\begin{split}
A_{11} =&\sum_{m\geq 1}\sum_{m_{1}\geq 1}\sum_{m_{2}\geq 1}\int_{\Delta
_{\theta ,t}^{m}}(\prod_{j=1}^{m}Db(s_{j},X_{s_{j}}^{x}))\int_{\Delta
_{\theta ,s_{m}}^{m_{1}+m_{2}+1}}\mathcal{H}_{m_{1},m_{2}}^{X}(u) \\
&du_{m_{1}+m_{2}+1}...du_{1}ds_{m}\cdots ds_{1}.
\end{split}
\end{equation}
Using now Lemma \ref{5shuffleIntegral} to the intergrals in $A_{10}$, we find
that%
\begin{equation}
A_{11}=\sum_{m\geq 1}\sum_{m_{1}\geq 1}\sum_{m_{2}\geq 1}\int_{\Delta
_{\theta ,t}^{m+m_{1}+m_{2}+1}}\mathcal{H}%
_{m,m_{1},m_{2}}^{X}(u)du_{m+m_{1}+m_{2}+1}...du_{1}  \label{5A11}
\end{equation}%
for $u=(u_{1},...,u_{m+m_{1}+m_{2}+1}),$ where the integrand $\mathcal{H}%
_{m,m_{1},m_{2}}^{X}(u)\in \mathbb{R}^{d}\otimes \mathbb{R}^{d}\otimes 
\mathbb{R}^{d}$ possesses entries given by sums of at most $%
C(d)^{m+m_{1}+m_{2}+1}$ summands, which are products of length $%
m+m_{1}+m_{2}+1$ of functions belonging to the class%
\begin{equation*}
\left\{ 
\begin{array}{c}
\frac{\partial ^{j}}{\partial x_{l_{1}}\partial x_{l_{j}}}%
b^{(i)}(u,X_{u}^{x})\left( K_{H}(u,\theta ^{\prime })-K_{H}(u,\theta
)\right) ^{\varepsilon }, \\ 
j=1,2,l_{1},l_{2},i=1,...,d,\varepsilon \in \{0,1\}%
\end{array}%
\right\} .
\end{equation*}%
Also here we observe that second order derivatives of functions as well as
the factor $\left( K_{H}(u,\theta ^{\prime })-K_{H}(u,\theta )\right) $ in
those products of functions on $\Delta _{\theta ,t}^{m+m_{1}+m_{2}+1}$ in (%
\ref{5I2}) only appear once. Therefore the absolute value of the multi-index 
$\alpha $ with respect to the total order of derivatives of those products
of functions in connection with Lemma \ref{5OrderDerivatives} in the
Appendix is given by%
\begin{equation}
\left\vert \alpha \right\vert =m+m_{1}+m_{2}+2.
\end{equation}

Further, we see that%

\begin{equation}
\begin{split}
E[\left\Vert A_{10}\right\Vert _{\mathbb{R}^{d\times d\times d}}^{2}]
&\leq C(d)E[\left\Vert A_{11}\right\Vert _{\mathbb{R}^{d\times d\times
d}}^{2}\left\Vert \frac{\partial }{\partial x}X_{\theta }^{x}\right\Vert _{%
\mathbb{R}^{d\times d\times d}}^{2}]  \notag \\
&\leq C(d)E[\left\Vert A_{11}\right\Vert _{\mathbb{R}^{d\times d\times
d}}^{4}]^{1/2}E[\left\Vert \frac{\partial }{\partial x}X_{\theta
}^{x}\right\Vert _{\mathbb{R}^{d\times d}}^{4}]^{1/2}.  \label{5EA10}
\end{split}
\end{equation}

It follows from Lemma \ref{5BoundedDerivatives} that%
\begin{equation*}
\sup_{x\in \mathbb{R}^{d}}E[\left\Vert \frac{\partial }{\partial x}X_{\theta
}^{x}\right\Vert _{\mathbb{R}^{d\times d}}^{4}]^{1/2}\leq
C_{H,d,T}(\left\Vert b\right\Vert _{L_{\infty }^{\infty }},\left\Vert
b\right\Vert _{L_{\infty }^{1}})
\end{equation*}%
for some continuous function $C_{H,d,T}:[0,\infty )^{2}\longrightarrow
\lbrack 0,\infty )$, which doesn't depend on $\theta $.

On the other hand, we can employ H\"{o}lder's inequality and Girsanov's
theorem (Theorem \ref{5girsanov}) in combination with Lemma \ref{5Novikov} in
the Appendix and get that%
\begin{equation}
\begin{split}
&E[\left\Vert A_{11}\right\Vert _{\mathbb{R}^{d\times d\times d}}^{4}]^{1/2}\leq   \notag \\
&C(\left\Vert b\right\Vert _{L_{\infty }^{\infty }})\left( \sum_{m\geq
1}\sum_{m_{1}\geq 1}\sum_{m_{2}\geq 1}\sum_{i\in I}\left\Vert \int_{\Delta
_{\theta ,t}^{m+m_{1}+m_{2}+1}}\mathcal{H}%
_{i,m,m_{1},m_{2}}^{B^{H}}(u)du_{m+m_{1}+m_{2}+1}...du_{1}\right\Vert
_{L^{8}(\Omega ;\mathbb{R})}\right) ^{2}  \label{5LP}
\end{split}
\end{equation}

where $C:[0,\infty )\longrightarrow \lbrack 0,\infty )$ is a continuous
function. Here $\#I\leq K^{m+m_{1}+m_{2}+1}$ for a constant $K=K(d)$ and the
integrands $\mathcal{H}_{i,m,m_{1},m_{2}}^{B^{H}}(u)$ are of the form 
\begin{equation*}
\mathcal{H}_{i,m,m_{1},m_{2}}^{B^{H}}(u)=\prod\limits_{l=1}^{m+m_{1}+m_{2}+1} h_{l}(u_{l}),h_{l}\in \Lambda
,l=1,...,m+m_{1}+m_{2}+1
\end{equation*}
where 
\begin{equation*}
\Lambda :=\left\{ 
\begin{array}{c}
\frac{\partial ^{j}}{\partial x_{l_{1}}\partial x_{l_{j}}}%
b^{(i)}(u,x+B_{u}^{H})\left( K_{H}(u,\theta ^{\prime })-K_{H}(u,\theta
)\right) ^{\varepsilon },
j=1,2,l_{1},l_{2},i=1,...,d,\varepsilon \in \{0,1\}%
\end{array}%
\right\} .
\end{equation*}%
Again, in this case functions with second order derivatives as well as the
factor $\left( K_{H}(u,\theta ^{\prime })-K_{H}(u,\theta )\right) $ only
appear once in those products.

Define 
\begin{equation*}
J=\left( \int_{\Delta _{\theta ,t}^{m+m_{1}+m_{2}+1}}\mathcal{H}%
_{i,m,m_{1},m_{2}}^{B^{H}}(u)du_{m+m_{1}+m_{2}+1}...du_{1}\right) ^{8}.
\end{equation*}%
Using once more the shuffle relation (\ref{5shuffleIntegral}) in the Appendix,
successively, we obtain that $J$ can be written as a sum of, at most of
length $K^{m+m_{1}+m_{2}+1}$ with summands of the form%
\begin{equation}
\int_{\Delta _{\theta
,t}^{8(m+m_{1}+m_{2}+1)}}\prod%
\limits_{l=1}^{8(m+m_{1}+m_{2}+1)}f_{l}(u_{l})du_{8(m+m_{1}+m_{2}+1)}...du_{1},
\label{5f}
\end{equation}%
where $f_{l}\in \Lambda $ for all $l$.

Here the number of factors $f_{l}$ in the above product, which have a second
order derivative, is exactly $8.$ The same applies to the factor $\left(
K_{H}(u,\theta ^{\prime })-K_{H}(u,\theta )\right) $. Therefore the total
order of the derivatives involved in (\ref{5f}) in connection with
Proposition \ref{5OrderDerivatives} is given by%
\begin{equation}
\left\vert \alpha \right\vert =8(m+m_{1}+m_{2}+2).  \label{5alpha2}
\end{equation}

We can now apply Theorem \ref{5mainestimate2} for $m$ given by $%
8(m+m_{1}+m_{2}+1)$ and $\sum_{j=1}^{8(m+m_{1}+m_{2}+1)}\varepsilon _{j}=8$
and find that%
\tiny
\begin{equation*}
\begin{split}
&\left\vert E\left[ \int_{\Delta _{\theta
,t}^{8(m+m_{1}+m_{2}+1)}}\prod%
\limits_{l=1}^{8(m+m_{1}+m_{2}+1)}f_{l}(u_{l})du_{8(m+m_{1}+m_{2}+1)}...du_{1}%
\right] \right\vert 
\leq \\
&C^{m+m_{1}+m_{2}}(\left\Vert b\right\Vert _{L^{1}(\mathbb{R}^{d};L^{\infty
}([0,T];\mathbb{R}^{d}))})^{8(m+m_{1}+m_{2}+1)}\left( \frac{\theta -\theta
^{\prime }}{\theta \theta ^{\prime }}\right) ^{8\gamma }\theta ^{8\left( H-%
\frac{1}{2}-\gamma \right) }\times \\
&\frac{((2(8(m+m_{1}+m_{2}+2))!)^{1/4}(t-\theta
)^{-H(8(m+m_{1}+m_{2}+1)d+2\cdot 8(m+m_{1}+m_{2}+2))-(H-\frac{1}{2}-\gamma
)8+8(m+m_{1}+m_{2}+1)}}{\Gamma (-H(2d8(m+m_{1}+m_{2}+1)+4\cdot
8(m+m_{1}+m_{2}+2))+2(H-\frac{1}{2}-\gamma )8+2\cdot
8(m+m_{1}+m_{2}+1))^{1/2}}
\end{split}
\end{equation*}
\normalsize
for a constant $C$ depending only on $H,T,d$.

So the latter shows that%
\tiny
\begin{equation*}
\begin{split}
&E[\left\Vert A_{11}\right\Vert _{\mathbb{R}^{d\times d\times d}}^{4}]^{1/2}
\leq \\
&C(\left\Vert b\right\Vert _{L_{\infty }^{\infty }})\Big(\sum_{m\geq
1}\sum_{m_{1}\geq 1}\sum_{m_{2}\geq 1}K^{m+m_{1}+m_{2}}\Big((\left\Vert
b\right\Vert _{L^{1}(\mathbb{R}^{d};L^{\infty }([0,T];\mathbb{R}%
^{d}))})^{m+m_{1}+m_{2}+1}\left( \frac{\theta -\theta ^{\prime }}{\theta
\theta ^{\prime }}\right) ^{\gamma }\theta ^{H-\frac{1}{2}-\gamma }\times \\
&\frac{((2(8(m+m_{1}+m_{2}+2))!)^{1/32}(t-\theta
)^{-H((m+m_{1}+m_{2}+1)d+2(m+m_{1}+m_{2}+2))+(H-\frac{1}{2}-\gamma
)+(m+m_{1}+m_{2}+1)}}{\Gamma (-H(2d8(m+m_{1}+m_{2}+1)+4\cdot
8(m+m_{1}+m_{2}+2))+2(H-\frac{1}{2}-\gamma )8+2\cdot
8(m+m_{1}+m_{2}+1))^{1/16}}\Big)\Big)^{2}
\end{split}
\end{equation*}
\normalsize
for a constant $K$ depending on $H,T,d$.

Since for small $\varepsilon =\gamma \in (0,H)$ we have that%
\begin{equation*}
\frac{1-2\varepsilon }{2(d+3)}\leq \frac{\frac{1}{2}-\varepsilon }{d+2+\frac{%
1}{4}}\leq \frac{(\frac{1}{2}-\varepsilon )+\frac{\varepsilon -\gamma }{%
m+m_{1}+m_{2}+1}}{d+2+\frac{1}{m+m_{1}+m_{2}+1}}
\end{equation*}
for $m,m_{1},m_{2}\geq 1$, the above sum converges, when $H<\frac{1}{2(d+3)}$%
.

So using the latter combined with (\ref{5EA10}), we find that%
\begin{equation*}
E[\left\Vert A_{10}\right\Vert _{\mathbb{R}^{d\times d\times d}}^{2}]\leq
C_{H,d,T}(\Vert b\Vert _{L_{\infty }^{\infty }},\Vert b\Vert _{L_{\infty
}^{1}})\left( \frac{\theta -\theta ^{\prime }}{\theta \theta ^{\prime }}%
\right) ^{2\gamma }\theta ^{2\left( H-\frac{1}{2}-\gamma \right) }
\end{equation*}%
for a continuous function $C_{H,d,T}:[0,\infty )\times \lbrack 0,\infty
)\longrightarrow \lbrack 0,\infty )$.

Let us now choose a small $\gamma \in (0,H)$ and a suitably small $\beta \in
(0,1/2)$, $0<\beta <\gamma <H<1/2$ such that 
\begin{equation*}
\int_{0}^{t}\int_{0}^{t}\left\vert \frac{\theta -\theta ^{\prime }}{\theta
\theta ^{\prime }}\right\vert ^{2\gamma }|\theta |^{2\left( H-\frac{1}{2}%
-\gamma \right) }|\theta -\theta ^{\prime }|^{-1-2\beta }d\theta ^{\prime
}d\theta <\infty ,
\end{equation*}%

for every $t\in (0,T]$. See the proof of Lemma A.4 in \cite{5BNP19}.
Hence%
\begin{equation}
\int_{0}^{t}\int_{0}^{t}\frac{E[\left\Vert A_{10}\right\Vert _{\mathbb{R}%
^{d\times d\times d}}^{2}]}{|\theta -\theta ^{\prime }|^{1+2\beta }}d\theta
^{\prime }d\theta \leq C_{H,d,T}(\Vert b\Vert _{L_{\infty }^{\infty }},\Vert
b\Vert _{L_{\infty }^{1}})  \label{5EstimateA10}
\end{equation}%
for all $t\in (0,T]$, where$C_{H,d,T}:[0,\infty )\times \lbrack 0,\infty
)\longrightarrow \lbrack 0,\infty )$ is a continuous function.

In the same way (and mostly even easier), we can treat the other terms $%
A_{9},A_{1},...,A_{6}$ and $A_{8}$ and obtain estimates of the type (\ref%
{5EstimateA10}).

Thus we get from (\ref{5Difference}) the first estimate (\ref{5E1}).

In a similar way (but simpler), we also derive the estimate (\ref{5E2}).
\end{proof}

\subsubsection{Finding $\left((t,x)\mapsto X_{t}^{s,x},(t,x)\mapsto \frac{\partial }{\partial x}X_{t}^{s,x}\right)$}
\label{5relativeCompactnessDerivative}
\begin{theorem}
\label{5StrongConvergenceDerivative} Suppose that $H<\frac{1}{2(d+3)}$. Let $%
\{b_{n}\}_{n\geq 1}\subset C_{c}^{\infty }([0,T]\times \mathbb{R}^{d})$,
such that 
\begin{equation*}
b_{n}(t,x)\longrightarrow b(t,x)
\end{equation*}%
as $n\rightarrow \infty $ for a.e. $(t,x)\in \lbrack 0,T]\times \mathbb{R}%
^{d}$ with $\sup_{n\geq 0}\Vert b_{n}\Vert _{L_{\infty }^{1}}<\infty $ and
such that $|b_{n}(t,x)|\leq M<\infty $, $n\geq 0$ a.e. for some constant $M$%
. Assume that $X_{\cdot }^{s,x,n}$ is the unique strong solution to the SDE 
\begin{equation*}
dX_{t}^{s,x,n}=b_{n}(t,X_{t}^{s,x,n})du+dB_{t}^{H},\,\,0\leq t\leq
T,\,\,\,X_{s}^{s,x,n}=x\in \mathbb{R}^{d}\,.
\end{equation*}%
Let $K$ be a compact cube in $\mathbb{R}^{d}$. Then for all $0\leq s<T$ the
sequences 
\begin{equation*}
((t,x)\mapsto X_{t}^{s,x,n})_{n\geq 1}
\end{equation*}%
and%
\begin{equation*}
((t,x)\mapsto \frac{\partial }{\partial x}X_{t}^{s,x,n})_{n\geq 1}
\end{equation*}
are relatively compact in $\mathcal{C}([s,T]\times K;L^{2}(\Omega ;\mathbb{R}^{d}))$
and $\mathcal{C}([s,T]\times K;L^{2}(\Omega ;\mathbb{R}^{d\times d}))$, respectively.
\end{theorem}

\begin{remark}
We mention here that the second statement in Theorem \ref%
{5StrongConvergenceDerivative} on the relative compactness of the derivative
of the flows is not needed for proving the estimate in Theorem \ref%
{5MainEstimate}.  We believe that this result is of independent
interest and that it will pay off dividends in the study of other problems
in SDE and SPDE theory.
\end{remark}

\begin{proof}
Without loss of generality, assume that $s=0$. Let $x,y\in K$ and $%
t_{1}<t_{2}$. Then%
\begin{equation*}
\begin{split}
&E\left[\left\Vert \frac{\partial }{\partial x}X_{t_{1}}^{x,n}-\frac{\partial }{%
\partial x}X_{t_{2}}^{y,n}\right\Vert ^{2}\right] \leq 2\left(E\left[\left\Vert \frac{\partial }{\partial x}X_{t_{1}}^{x,n}-\frac{%
\partial }{\partial x}X_{t_{2}}^{x,n}\right\Vert ^{2}\right]+E\left[\left\Vert \frac{%
\partial }{\partial x}X_{t_{2}}^{x,n}-\frac{\partial }{\partial x}%
X_{t_{2}}^{y,n}\right\Vert ^{2}\right]\right).
\end{split}
\end{equation*}

By using the fundamental theorem of calculus with respect to the
one dimensional arguments of the entries of $\frac{\partial }{\partial x}%
X_{t}^{x,n}=\left( \frac{\partial }{\partial x_{i}}X_{t}^{x,n,j}\right)
_{1\leq i,j\leq d}$ and by using the Euclidean norm $\left\Vert \cdot
\right\Vert $, we find that 

\begin{equation*}
\begin{split}
&E[\left\Vert \frac{\partial }{\partial x}X_{t_{2}}^{x,n}-\frac{\partial }{%
\partial x}X_{t_{2}}^{y,n}\right\Vert ^{2}] \\
=&\sum_{i,j=1}^{d}E[(\frac{\partial }{\partial x_{i}}%
X_{t_{2}}^{x_{1},...,x_{d},n,j}-\frac{\partial }{\partial x_{i}}%
X_{t_{2}}^{y_{1},...,y_{d},n,j})^{2}] \\
\leq &C_{d}\sum_{i,j=1}^{d}\{E[(\frac{\partial }{\partial x_{i}}%
X_{t_{2}}^{x_{1},...,x_{d},n,j}-\frac{\partial }{\partial x_{i}}%
X_{t_{2}}^{y_{1},x_{2}...,x_{d},n,j})^{2}] \\
&+E[(\frac{\partial }{\partial x_{i}}%
X_{t_{2}}^{y_{1},x_{2},x_{3}...,x_{d},n,j}-\frac{\partial }{\partial x_{i}}%
X_{t_{2}}^{y_{1},y_{2},x_{3}...,x_{d},n,j})^{2}] \\
&+...+E[(\frac{\partial }{\partial x_{i}}%
X_{t_{2}}^{y_{1},y_{2},...,y_{d-1},x_{d},n,j}-\frac{\partial }{\partial x_{i}%
}X_{t_{2}}^{y_{1},y_{2},y_{3}...,y_{d},n,j})^{2}]\} \\
\leq &C_{d}(K)\sum_{i,j=1}^{d}\{\left\vert x_{1}-y_{1}\right\vert
^{2}\sup_{x\in K}E[(\frac{\partial ^{2}}{\partial x_{1}\partial x_{i}}%
X_{t_{2}}^{x_{1},...,x_{d},n,j})^{2}] \\
&+\left\vert x_{2}-y_{2}\right\vert ^{2}\sup_{x\in K}E[(\frac{\partial ^{2}%
}{\partial x_{1}\partial x_{i}}X_{t_{2}}^{x_{1},...,x_{d},n,j})^{2}] \\
&+...+\left\vert x_{d}-y_{d}\right\vert ^{2}\sup_{x\in K}E[(\frac{\partial
^{2}}{\partial x_{1}\partial x_{i}}X_{t_{2}}^{x_{1},...,x_{d},n,j})^{2}]\} \\
\leq &C_{d}(K)\left\Vert x-y\right\Vert ^{2}\sup_{n\in \mathbb{N},t\in
\lbrack 0,T],x\in K}E[\left\Vert \frac{\partial ^{2}}{\partial x^{2}}%
X_{t}^{x,n}\right\Vert _{\mathbb{R}^{d\times d\times d}}^{2}].
\end{split}
\end{equation*}
We know by our assumptions and by Lemma \ref{5BoundedDerivatives} in the Appendix that%
\begin{equation*}
\sup_{n\in \mathbb{N},t\in \lbrack 0,T],x\in K}E[\left\Vert \frac{\partial
^{2}}{\partial x^{2}}X_{t}^{x,n}\right\Vert _{\mathbb{R}^{d\times d\times
d}}^{2}]<\infty \text{.}
\end{equation*}%
Further, by using exactly the same arguments as in the proof of Lemma \ref%
{5RelativeCompactness}, we also obtain that%
\begin{equation*}
E[\left\Vert \frac{\partial }{\partial x}X_{t_{1}}^{x,n}-\frac{\partial }{%
\partial x}X_{t_{2}}^{x,n}\right\Vert ^{2}]\leq C_{d,H,T}\left\vert
t_{1}-t_{2}\right\vert ^{\epsilon },
\end{equation*}%
where $\epsilon >0$ and $C_{d,H,T}$ is a constant only depending on $d,H,T$.

Altogether, we see that%
\begin{equation*}
E[\left\Vert \frac{\partial }{\partial x}X_{t_{1}}^{x,n}-\frac{\partial }{%
\partial x}X_{t_{2}}^{y,n}\right\Vert ^{2}]\leq C_{d,H,T}(K)(\left\Vert
x-y\right\Vert ^{2}+\left\vert t_{1}-t_{2}\right\vert ^{\epsilon })
\end{equation*}%
for a constant $C_{d,H,T}(K)$, which only depends on $d,H,T$ and the compact
cube $K$.

On the other hand, we conclude from Lemma \ref{5RelativeCompactness} combined
with Corollary \ref{5VI_compactcrit} that the family $\left\{ \frac{\partial 
}{\partial x}X_{t}^{x,n}\right\} _{n\in \mathbb{N},t\in \lbrack 0,T],x\in K}$
is relatively compact in $L^{2}(\Omega ;\mathbb{R}^{d\times d})$. Thus it
follows from the Theorem of Arzela-Ascoli, with respect to the infinite dimensional state space, that%
\begin{equation*}
((t,x)\mapsto \frac{\partial }{\partial x}X_{t}^{s,x,n})_{n\geq 1}
\end{equation*}%
is relatively compact in $\mathcal{C}([0,T]\times K;L^{2}(\Omega ;\mathbb{R}^{d\times
d}))$.

The first statement of the Theorem can be treated, similarly.
\end{proof}

\bigskip

\subsubsection{Identifying the limits and proving compatibility of the two limits}
\label{5identifyGoodLimits}

\begin{lemma}
\label{5ContinuousVersion}Retain the conditions of Theorem \ref%
{5StrongConvergenceDerivative}. Let $Y$ and $Y^{\prime }$ be the limits of 
\begin{equation*}
((t,x)\mapsto X_{t}^{s,x,n_{k}})_{k\geq 1}
\end{equation*}%
and%
\begin{equation*}
((t,x)\mapsto \frac{\partial }{\partial x}X_{t}^{s,x,n_{k}})_{k\geq 1}
\end{equation*}%
in $\mathcal{C}([s,T]\times K;L^{2}(\Omega ;\mathbb{R}^{d}))$ and $\mathcal{C}([s,T]\times
K;L^{2}(\Omega ;\mathbb{R}^{d\times d}))$, respectively with respect to a
subsequence $(n_{k})_{k\geq 1}$, which only depends on the compact cube $K$, 
$s$ and $T$. Then, $Y$ and $Y^{\prime }$ have continuous versions $Z$ and $%
Z^{\prime }$, respectively.
\end{lemma}

\begin{proof}
By using very similar arguments as in the proof of Theorem \ref%
{5StrongConvergenceDerivative} in combination with the techniques in Lemma %
\ref{5RelativeCompactness}, we can find for integers $p\geq 2$ the estimate%
\begin{equation*}
E[\left\Vert \frac{\partial }{\partial x}X_{t_{1}}^{x,n_{k}}-\frac{\partial 
}{\partial x}X_{t_{2}}^{y,n_{k}}\right\Vert ^{p}]\leq (\left\Vert
x-y\right\Vert ^{p-1}+\left\vert t_{1}-t_{2}\right\vert ^{\epsilon (p-1)}),
\end{equation*}%
where $\epsilon >0$ (small) and $C_{d,H,T,p}(K)$ is a constant depending only
on $d,H,T,p$ and $K$. So if $p$ is large enough, we obtain that%
\begin{equation*}
E[\left\Vert \frac{\partial }{\partial x}X_{t_{1}}^{x,n}-\frac{\partial }{%
\partial x}X_{t_{2}}^{y,n}\right\Vert ^{p}]\leq C_{d,H,T,p}(K)(\left\Vert
x-y\right\Vert +\left\vert t_{1}-t_{2}\right\vert )^{d+1+\beta }
\end{equation*}%
for some $\beta >0$. By applying Fatou's Lemma in connection with a
subsequence $(m_{k})_{k\geq 1}$ of $(n_{k})_{k\geq 1}$, we find that%
\begin{equation}
\begin{split}
&E[\left\Vert Y^{\prime }(s,t_{1},x)-Y^{\prime }(s,t_{2},y)\right\Vert ^{p}]
\\
\leq &\underline{\lim }_{k\longrightarrow \infty }C_{d,H,T,p}(K)(\left\Vert
x-y\right\Vert +\left\vert t_{1}-t_{2}\right\vert )^{d+1+\beta } \\
=&C_{d,H,T,p}(K)(\left\Vert x-y\right\Vert +\left\vert
t_{1}-t_{2}\right\vert )^{d+1+\beta }\text{.}
\end{split}
\end{equation}

The proof follows by using Kolmogorov's continuity theorem.

The existence of a continuous version of $Y$ can be shown in the same way.
\end{proof}

\begin{theorem}
\label{5Identification}Assume the conditions of Lemma \ref{5ContinuousVersion}%
. Then%
\begin{equation*}
Z(s,t,x)=X_{t}^{s,x}\text{ and }Z^{\prime }(s,t,x)=\frac{\partial }{\partial
x}X_{t}^{s,x},
\end{equation*}%
a.e. in $t,x,\omega $, where $\frac{\partial }{\partial x}X_{t}^{s,x}$ is a
representative of the spatial Sobolev derivative of $X_{t}^{s,x}$ in $%
W^{1,2}(U)$ for an open and bounded set $U\subset K$.
\end{theorem}

\begin{proof}
The proof follows from an integration by part argument in connection with
Theorem \ref{5StrongConvergenceDerivative} and the uniform estimate in Lemma %
\ref{5BoundedDerivatives} in the Appendix.
\end{proof}

\subsubsection{Controlling $\sup_{t\in \lbrack s,T]}\left\Vert \frac{\partial }{%
		\partial x}X_{t}^{s,x}\right\Vert ^{2^{n}}$  for $n\geq 1$}
	
	\label{5mainSupEstimate}
	We are now coming to the proof Theorem \ref{5MainEstimate}, which also provides a uniform estimate, in $t$, of the moments of the derivative of our solution.

\begin{proof}[Proof of Theorem \ref{5MainEstimate}]
Let us assume without loss of generality that $s=0$ and $T=1$. In proving
this result we want to apply the inequality of Garsia-Rodemich-Rumsey (see
Lemma \ref{GarciaRodemichRumsey} in the Appendix) to the case, when $%
d(t,s)=\left\vert t-s\right\vert ^{\frac{\varepsilon }{1+\varepsilon }}$, $%
0<\varepsilon <1$, $\Psi (x)=x^{^{\frac{4(1+\varepsilon )}{\varepsilon }%
}},x\geq 0$, $\Lambda =\left[ 0,1\right] $, $f(t)=\left\Vert
X_{t}^{x}-X_{t}^{y}\right\Vert ,x,y\in K$, where $K$ is a compact cube in $%
\mathbb{R}^{d}$. Then, $\sigma (r)\geq r^{\frac{1+\varepsilon }{\varepsilon }%
}$ and we get that%
\begin{equation*}
\left\vert f(t)-f(s)\right\vert \leq 18\int_{0}^{d(t,s)/2}\Psi ^{-1}\left( 
\frac{U}{(\sigma (r))^{2}}\right) dr,
\end{equation*}%
where%
\begin{equation*}
U=\int_{0}^{1}\int_{0}^{1}\Psi \left( \frac{\left\vert
f(t_{2})-f(t_{1})\right\vert }{d(t_{2},t_{1})}\right) dt_{2}dt_{1}.
\end{equation*}%
Setting $s=0$, we see that%
\begin{eqnarray*}
\left\vert f(t)\right\vert  &\leq &18\int_{0}^{d(t,0)/2}\Psi ^{-1}\left( 
\frac{U}{(\sigma (r))^{2}}\right) dr+\left\vert f(0)\right\vert  \\
&=&18\int_{0}^{d(t,0)/2}\Psi ^{-1}\left( \frac{U}{(\sigma (r))^{2}}\right)
dr+\left\Vert x-y\right\Vert .
\end{eqnarray*}%
Let $p=2^{r}>1$ with $r\in \mathbb{N}$ such that $p\frac{\varepsilon }{%
4(1+\varepsilon )}>1$. Then%
\begin{equation*}
\left\vert f(t)\right\vert ^{p}\leq C_{p}((\int_{0}^{d(t,0)/2}\Psi
^{-1}\left( \frac{U}{(\sigma (r))^{2}}\right) dr)^{p}+\left\Vert
x-y\right\Vert ^{p}).
\end{equation*}%
Hence,%
\begin{eqnarray*}
\sup_{0\leq t\leq 1}\left\vert f(t)\right\vert ^{p} &\leq
&C_{p}((\int_{0}^{1}\Psi ^{-1}\left( \frac{U}{(\sigma (r))^{2}}\right)
dr)^{p}+\left\Vert x-y\right\Vert ^{p}) \\
&=&C_{p}((\int_{0}^{1}\left( \frac{U}{(\sigma (r))^{2}}\right) ^{\frac{%
\varepsilon }{4(1+\varepsilon )}}dr)^{p}+\left\Vert x-y\right\Vert ^{p}) \\
&\leq &C_{p}((\int_{0}^{1}\left( \frac{1}{r^{^{\frac{2(1+\varepsilon )}{%
\varepsilon }}}}\right) ^{\frac{\varepsilon }{4(1+\varepsilon )}}dr)^{p}U^{p%
\frac{\varepsilon }{4(1+\varepsilon )}}+\left\Vert x-y\right\Vert ^{p}) \\
&=&C_{p}((\int_{0}^{1}\frac{1}{r^{^{\frac{1}{2}}}}dr)^{p}U^{p\frac{%
\varepsilon }{4(1+\varepsilon )}}+\left\Vert x-y\right\Vert ^{p}) \\
&\leq &C_{p}(U^{p\frac{\varepsilon }{4(1+\varepsilon )}}+\left\Vert
x-y\right\Vert ^{p}).
\end{eqnarray*}%
On the other hand,%
\begin{eqnarray*}
U^{p\frac{\varepsilon }{4(1+\varepsilon )}} &\leq
&\int_{0}^{1}\int_{0}^{1}\left( \Psi \left( \frac{\left\vert
f(t_{2})-f(t_{1})\right\vert }{d(t_{2},t_{1})}\right) \right) ^{p\frac{%
\varepsilon }{4(1+\varepsilon )}}dt_{2}dt_{1} \\
&=&\int_{0}^{1}\int_{0}^{1}\left( \left( \frac{\left\vert
f(t_{2})-f(t_{1})\right\vert }{d(t_{2},t_{1})}\right) ^{\frac{%
4(1+\varepsilon )}{\varepsilon }}\right) ^{p\frac{\varepsilon }{%
4(1+\varepsilon )}}dt_{2}dt_{1} \\
&=&\int_{0}^{1}\int_{0}^{1}\left( \frac{\left\vert
f(t_{2})-f(t_{1})\right\vert }{d(t_{2},t_{1})}\right) ^{p}dt_{2}dt_{1}.
\end{eqnarray*}%
Therefore, we get that%
\begin{eqnarray*}
E\left[ \sup_{0\leq t\leq 1}\left\vert f(t)\right\vert ^{p}\right]  &\leq
&C_{p}(E\left[ U^{p\frac{\varepsilon }{2(1+\varepsilon )}}\right]
+\left\Vert x-y\right\Vert ^{p}) \\
&\leq &C_{p}(\int_{0}^{1}\int_{0}^{1}E\left[ \left( \frac{\left\vert
f(t_{2})-f(t_{1})\right\vert }{d(t_{2},t_{1})}\right) ^{p}\right]
dt_{2}dt_{1}+\left\Vert x-y\right\Vert ^{p}) \\
&\leq &C_{p}(\int_{0}^{1}\int_{0}^{1}E\left[ \left( \frac{\left\Vert
X_{t_{2}}^{x}-X_{t_{2}}^{y}-(X_{t_{1}}^{x}-X_{t_{1}}^{y})\right\Vert }{%
d(t_{2},t_{1})}\right) ^{p}\right] dt_{2}dt_{1}+\left\Vert x-y\right\Vert
^{p}).
\end{eqnarray*}%
Further, by using Theorem \ref{5StrongConvergenceDerivative}, the uniform
integrability of $\left\Vert X_{t}^{x,n}\right\Vert ^{p},n\geq 1,0\leq t\leq
1$ and Fatou `s Lemma, we obtain that%
\begin{eqnarray*}
&&E\left[ \sup_{0\leq t\leq 1}\left\vert f(t)\right\vert ^{p}\right]  \\
&\leq &C_{p}(\underline{\lim }_{n\longrightarrow \infty
}\int_{0}^{1}\int_{0}^{1}E\left[ \left( \frac{\left\Vert
X_{t_{2}}^{x,n}-X_{t_{2}}^{y,n}-(X_{t_{1}}^{x,n}-X_{t_{1}}^{y,n})\right\Vert 
}{d(t_{2},t_{1})}\right) ^{p}\right] dt_{2}dt_{1}+\left\Vert x-y\right\Vert
^{p})\text{,}
\end{eqnarray*}%
where $X_{t}^{x,n},0\leq t\leq 1$ is the strong solution to \ref{11SDE}
associated with the approximating sequence of smooth vector fields $%
b_{n},n\geq 1$ in Theorem \ref{5StrongConvergenceDerivative}.

Let $x=(x_{1},...,x_{d})^{\ast },y=(y_{1},...,y_{d})^{\ast }$ and and assume
without loss of generality that $x_{i}>y_{i},i=1,...,d$. Then, by using the
fundamental theorem of calculus, we can write%
\begin{eqnarray*}
&&X_{t}^{x,n}-X_{t}^{n,y} \\
&=&(X_{t}^{x_{1},...,x_{d},n}-X_{t}^{y_{1},x_{2},...,x_{d},n})+(X_{t}^{y_{1},x_{2,}x_{3},...,x_{d},n}-X_{t}^{y_{1},y_{2},x_{3},...,x_{d},n})
\\
&&+...+(X_{t}^{y_{1},y_{2,}...,y_{d-1},x_{d},n}-X_{t}^{y_{1},y_{2},...,y_{d},n})
\\
&=&\int_{y_{1}}^{x_{1}}\frac{\partial }{\partial z_{1}}%
X_{t}^{z_{1},x_{2}...,x_{d},n}dz_{1}+\int_{y_{2}}^{x_{2}}\frac{\partial }{%
\partial z_{2}}X_{t}^{y_{1},z_{2},x_{3},...,x_{d},n}dz_{2} \\
&&+...+\int_{y_{d}}^{x_{d}}\frac{\partial }{\partial z_{d}}%
X_{t}^{y_{1},y_{2},...,y_{d-1},z_{d},n}dz_{d}.
\end{eqnarray*}%
So%
\begin{eqnarray*}
&&X_{t_{2}}^{x,n}-X_{t_{2}}^{y,n}-(X_{t_{1}}^{x,n}-X_{t_{1}}^{y,n}) \\
&=&\int_{y_{1}}^{x_{1}}(\frac{\partial }{\partial z_{1}}%
X_{t_{2}}^{z_{1},x_{2}...,x_{d},n}-\frac{\partial }{\partial z_{1}}%
X_{t_{1}}^{z_{1},x_{2}...,x_{d},n})dz_{1}+\int_{y_{2}}^{x_{2}}(\frac{%
\partial }{\partial z_{2}}X_{t_{2}}^{y_{1},z_{2},x_{3},...,x_{d},n}-\frac{%
\partial }{\partial z_{2}}X_{t_{1}}^{y_{1},z_{2},x_{3},...,x_{d},n})dz_{2} \\
&&+...+\int_{y_{d}}^{x_{d}}(\frac{\partial }{\partial z_{d}}%
X_{t_{2}}^{y_{1},y_{2},...,y_{d-1},z_{d},n}-\frac{\partial }{\partial z_{d}}%
X_{t_{1}}^{y_{1},y_{2},...,y_{d-1},z_{d},n})dz_{d}.
\end{eqnarray*}%
Hence,%
\begin{eqnarray*}
&&E\left[ \left( \frac{\left\Vert
X_{t_{2}}^{x,n}-X_{t_{2}}^{y,n}-(X_{t_{1}}^{x,n}-X_{t_{1}}^{y,n})\right\Vert 
}{d(t_{2},t_{1})}\right) ^{p}\right] \\
&\leq &C_{d}(E\left[ (\int_{y_{1}}^{x_{1}}\left\Vert \frac{\partial }{%
\partial z_{1}}X_{t_{2}}^{z_{1},x_{2}...,x_{d},n}-\frac{\partial }{\partial
z_{1}}X_{t_{1}}^{z_{1},x_{2}...,x_{d},n}\right\Vert
/d(t_{2},t_{1})dz_{1})^{p}\right] \\
&&+E\left[ (\int_{y_{2}}^{x_{2}}\left\Vert \frac{\partial }{\partial z_{2}}%
X_{t_{2}}^{y_{1},z_{2},x_{3},...,x_{d},n}-\frac{\partial }{\partial z_{2}}%
X_{t_{1}}^{y_{1},z_{2},x_{3},...,x_{d},n}\right\Vert
/d(t_{2},t_{1})dz_{2})^{p}\right] \\
&&+...+E\left[ (\int_{y_{d}}^{x_{d}}\left\Vert \frac{\partial }{\partial
z_{d}}X_{t_{2}}^{y_{1},y_{2},...,y_{d-1},z_{d},n}-\frac{\partial }{\partial
z_{d}}X_{t_{1}}^{y_{1},y_{2},...,y_{d-1},z_{d},n}\right\Vert
/d(t_{2},t_{1})dz_{d})^{p})\right] \\
&\leq &C_{d}(\left\vert x_{1}-y_{1}\right\vert ^{p-1}\int_{y_{1}}^{x_{1}}E 
\left[ (\left\Vert \frac{\partial }{\partial z_{1}}%
X_{t_{2}}^{z_{1},x_{2}...,x_{d},n}-\frac{\partial }{\partial z_{1}}%
X_{t_{1}}^{z_{1},x_{2}...,x_{d},n}\right\Vert /d(t_{2},t_{1}))^{p}\right]
dz_{1} \\
&&+\left\vert x_{2}-y_{2}\right\vert ^{p-1}\int_{y_{2}}^{x_{2}}E\left[
(\left\Vert \frac{\partial }{\partial z_{2}}%
X_{t_{2}}^{y_{1},z_{2},x_{3},...,x_{d},n}-\frac{\partial }{\partial z_{2}}%
X_{t_{1}}^{y_{1},z_{2},x_{3},...,x_{d},n}\right\Vert /d(t_{2},t_{1}))^{p}%
\right] dz_{2} \\
&&+...+ \\
&&\left\vert x_{d}-y_{d}\right\vert ^{p-1}\int_{y_{d}}^{x_{d}}E\left[
(\left\Vert \frac{\partial }{\partial z_{d}}%
X_{t_{2}}^{y_{1},y_{2},...,y_{d-1},z_{d},n}-\frac{\partial }{\partial z_{d}}%
X_{t_{1}}^{y_{1},y_{2},...,y_{d-1},z_{d},n}\right\Vert /d(t_{2},t_{1}))^{p})%
\right] dz_{d}) \\
&\leq &C_{d}(\left\vert x_{1}-y_{1}\right\vert ^{p}+...+\left\vert
x_{d}-y_{d}\right\vert ^{p})\times \\
&&\times \sup_{x\in K}E\left[ (\left\Vert \frac{\partial }{\partial x}%
X_{t_{2}}^{x,n}-\frac{\partial }{\partial x}X_{t_{1}}^{x,n}\right\Vert
/d(t_{2},t_{1}))^{p})\right] \\
&\leq &C_{d}\left\Vert x-y\right\Vert ^{p}\sup_{x\in K}E\left[ (\left\Vert 
\frac{\partial }{\partial x}X_{t_{2}}^{x,n}-\frac{\partial }{\partial x}%
X_{t_{1}}^{x,n}\right\Vert /d(t_{2},t_{1}))^{p})\right] .
\end{eqnarray*}%
Thus%
\begin{eqnarray}
&&E\left[ \sup_{0\leq t\leq 1}\left\vert f(t)\right\vert ^{p}\right]  \notag
\\
&\leq &C_{p,d}(\underline{\lim }_{n\longrightarrow \infty
}\int_{0}^{1}\int_{0}^{1}\sup_{x\in K}E\left[ (\left\Vert \frac{\partial }{%
\partial x}X_{t_{2}}^{x,n}-\frac{\partial }{\partial x}X_{t_{1}}^{x,n}\right%
\Vert /d(t_{2},t_{1}))^{p})\right] dt_{2}dt_{1}\left\Vert x-y\right\Vert ^{p}
\notag \\
&&+\left\Vert x-y\right\Vert ^{p})\text{.}  \label{Garcia}
\end{eqnarray}

Fix $0\leq t_{2},t_{1}\leq 1$ and suppose that $t_{2}>t_{1}$. Since \ the
stochastic flow associated with the smooth vector field $b_{n}$ is smooth,
too (compare to e.g. \cite{5Kunita}), we obtain that%
\begin{equation}
\frac{\partial }{\partial x}X_{t}^{x}=I_{d\times
d}+\int_{0}^{t}Db(u,X_{u}^{x})\frac{\partial }{\partial x}X_{u}^{x}du,
\label{5InitialDerivative}
\end{equation}%
where $Db:\mathbb{R}^{d}\longrightarrow L(\mathbb{R}^{d},\mathbb{R}^{d})$ is
the derivative of $b$ with respect to the space variable.

Using Picard iteration, we see that%
\begin{equation}
\frac{\partial }{\partial x}X_{t}^{x}=I_{d\times d}+\sum_{m\geq
1}\int_{\Delta
_{0,t}^{m}}Db(u_{1},X_{u_{1}}^{x})...Db(u_{m},X_{u_{m}}^{x})du_{m}...du_{1},
\label{5FirstOrderExpansion}
\end{equation}%
in $L^{2}(\left[ 0,1\right] \times \Omega )$, where%
\begin{equation*}
\Delta _{s,t}^{m}=\{(u_{m},...u_{1})\in \lbrack 0,T]^{m}:\theta
<u_{m}<...<u_{1}<t\}.
\end{equation*}%
So we can see that 
\begin{eqnarray*}
&&\frac{\partial }{\partial x}X_{t_{2}}^{x,n}-\frac{\partial }{\partial x}%
X_{t_{1}}^{x,n} \\
&=&\sum_{m\geq 1}\int_{\Delta _{0,t_{2}}^{m}}(\chi _{\left[ t_{1},t_{2}%
\right]
}(u_{1})Db_{n}(u_{1},X_{u_{1}}^{x,n}))Db_{n}(u_{2},X_{u_{2}}^{x,n})...Db_{n}(u_{1},X_{u_{m}}^{x,n})du_{m}...du_{1}
\end{eqnarray*}%
in $L^{2}(\left[ 0,1\right] \times \Omega )$.

Therefore we obtain from the Cauchy-Schwarz inequality and Girsanov's
theorem (Theorem \ref{5girsanov}) in combination with Lemma \ref{5Novikov}
in the Appendix that%
\begin{equation}
E\left[ \left\Vert \frac{\partial }{\partial x}X_{t_{2}}^{x,n}-\frac{%
\partial }{\partial x}X_{t_{1}}^{x,n}\right\Vert ^{p}\right] \leq
C(\left\Vert b_{n}\right\Vert _{L_{\infty }^{\infty }})\left( \sum_{m\geq
1}\sum_{i\in I}\left\Vert \int_{\Delta _{0,t_{2}}^{m}}\mathcal{H}%
_{i}^{B^{H}}(u)du_{m}...du_{1}\right\Vert _{L^{2p}(\Omega ;\mathbb{R}%
)}\right) ^{p}  \label{Start}
\end{equation}%
where $C:[0,\infty )\longrightarrow \lbrack 0,\infty )$ is a continuous
function which doesn't depend on $p$. Here $\#I\leq K^{m}$ for a constant $%
K=K(d)$ and the integrands $\mathcal{H}_{i}^{B^{H}}(u)$ are of the form 
\begin{equation*}
\mathcal{H}_{i}^{B^{H}}(u)=\prod\limits_{l=1}^{m}h_{l}(u_{l}),h_{l}\in
\Lambda _{1}\cup \Lambda _{2},l=1,...,m
\end{equation*}%
where 
\begin{equation*}
\Lambda _{1}:=\left\{ \frac{\partial }{\partial x_{l}}%
b_{n}^{(i)}(u,x+B_{u}^{H}),l,i=1,...,d\right\}
\end{equation*}%
and 
\begin{equation*}
\Lambda _{2}:=\left\{ \chi _{\left[ t_{1},t_{2}\right] }(u)\frac{\partial }{%
\partial x_{l}}b_{n}^{(i)}(u,x+B_{u}^{H}),l,i=1,...,d\right\}
\end{equation*}

Define 
\begin{equation*}
J=\left( \int_{\Delta _{0,t_{2}}^{m}}\mathcal{H}%
_{i}^{B^{H}}(u)du_{m}...du_{1}\right) ^{2p}.
\end{equation*}%
Using once more Lemma \ref{5partialshuffle} in the Appendix, successively,
we obtain that $J$ can be written as a sum of, at most of length $K^{pm}$
with summands of the form%
\begin{equation}
\int_{\Delta
_{0,t_{2}}^{2pm}}\prod\limits_{l=1}^{2pm}f_{l}(u_{l})du_{2pm}...du_{1},
\label{Product}
\end{equation}%
where $f_{l}\in \Lambda _{1}\cup \Lambda _{2}$ for all $l$.

Thus the total order of the derivatives involved in (\ref{5f}) in connection
with Lemma \ref{5OrderDerivatives} is given by%
\begin{equation}
\left\vert \alpha \right\vert =2pm.
\end{equation}%
Here we also mention that there are exactly $2p$ factors $f_{l}$ in $\Lambda
_{2}$ in the product of \ref{Product}.

Further, we know from \ref{IntPart} in the Appendix that%
\begin{equation*}
\int_{\Delta
_{0,t_{2}}^{2pm}}\prod\limits_{l=1}^{2pm}f_{l}(u_{l})du_{2pm}...du_{1}=%
\int_{(\mathbb{R}^{d})^{2pm}}\Lambda _{\alpha }^{g}(0,t_{2},z)dz\text{,}
\end{equation*}%
where%
\begin{equation}
\Lambda _{\alpha }^{g}(\theta ,t,z)=(2\pi )^{-d2pm}\int_{(\mathbb{R}%
^{d})^{2pm}}\int_{\Delta _{\theta
,t}^{2pm}}\prod\limits_{j=1}^{2pm}g_{j}(s_{j},z_{j})(-iu_{j})^{\alpha
_{j}}\exp (-i\left\langle u_{j},B_{s_{j}}-z_{j}\right\rangle dsdu\text{.}
\label{LocalTime}
\end{equation}%
Here $g_{j}\in \Lambda _{1}^{\ast }\cup \Lambda _{2}^{\ast },j=1,...,2pm$
for 
\begin{equation*}
\Lambda _{1}^{\ast }:=\left\{ \frac{\partial }{\partial x_{l}}%
b_{n}^{(i)}(u,x+B_{u}^{H}),l,i=1,...,d\right\}
\end{equation*}%
and 
\begin{equation*}
\Lambda _{2}^{\ast }:=\left\{ \chi _{\left[ t_{1},t_{2}\right] }(u)\frac{%
\partial }{\partial x_{l}}b_{n}^{(i)}(u,x+B_{u}^{H}),l,i=1,...,d\right\} .
\end{equation*}%
Also here we note that there are exactly $2p$ factors $g_{j}$ in $\Lambda
_{2}^{\ast }$ in the product of (\ref{LocalTime}).

So we can apply Lemma \ref{LocalTimeEstimate} in the Appendix for $m$
replaced by $2pm$ get that

\begin{eqnarray*}
&&\left\vert E\left[ \int_{(\mathbb{R}^{d})^{2pm}}\Lambda _{\alpha
}^{g}(0,t_{2},z)dz\right] \right\vert \\
&\leq &C^{pm+\left\vert \alpha \right\vert /2}\int_{(\mathbb{R}%
^{d})^{2pm}}(\Psi _{\alpha }^{g}(0,t_{2},z))^{1/2}dz \\
&=&C^{2pm}\int_{(\mathbb{R}^{d})^{2pm}}(\Psi _{\alpha
}^{g}(0,t_{2},z))^{1/2}dz,
\end{eqnarray*}%
where%
\begin{eqnarray*}
&&\Psi _{\alpha }^{g}(0,t_{2},z) \\
&=&\dprod\limits_{l=1}^{d}\sqrt{(2\left\vert \alpha ^{(l)}\right\vert )!}%
\sum_{\sigma \in S(2pm,2pm)}\int_{\Delta _{0,t_{2}}^{4pm}}\left\vert
g_{\sigma }(s,z)\right\vert \dprod\limits_{j=1}^{4pm}\frac{1}{\left\vert
s_{j}-s_{j-1}\right\vert ^{H(d+2\sum_{i=1}^{d}\alpha _{\left[ \sigma (j)%
\right] }^{(i)})}}ds_{1}...ds_{4pm},
\end{eqnarray*}%
\begin{equation*}
\left\vert g_{\sigma }(s,z)\right\vert :=\dprod\limits_{j=1}^{4pm}g_{\left[
\sigma (j)\right] }(s_{j},z_{\left[ \sigma (j)\right] }),
\end{equation*}%
\begin{equation*}
\left[ j\right] :=\left\{ 
\begin{array}{cc}
j & ,\text{ }1\leq j\leq 2pm \\ 
j-2pm & ,\text{ }m+1\leq j\leq 4pm%
\end{array}%
\right.
\end{equation*}%
and $\sigma \in S(2pm,2pm)$ are shuffle permutations.

Here, 
\begin{equation*}
\sum_{i=1}^{d}\alpha _{\left[ \sigma (j)\right] }^{(i)}=1
\end{equation*}%
for all $j$. Using Stirling 
\'{}%
s formula we also see that%
\begin{equation*}
\dprod\limits_{l=1}^{d}(2\left\vert \alpha ^{(l)}\right\vert )!\leq
(2\left\vert \alpha \right\vert )!C^{\left\vert \alpha \right\vert }.
\end{equation*}%
So, it follows from H\"{o}lder `s inequality that 
\begin{eqnarray*}
&&\Psi _{\alpha }^{g}(0,t_{2},z) \\
&\leq &C^{\left\vert \alpha \right\vert }\sqrt{(2\left\vert \alpha
\right\vert )!}\sum_{\sigma \in S(2pm,2pm)}\int_{\Delta
_{0,t_{2}}^{4pm}}\left\vert g_{\sigma }(s,z)\right\vert
\dprod\limits_{j=1}^{4pm}\frac{1}{\left\vert s_{j}-s_{j-1}\right\vert
^{H(d+2)}}ds_{1}...ds_{4pm} \\
&=&C^{2pm}\sqrt{(4pm)!}\sum_{\sigma \in S(2pm,2pm)}\int_{\Delta
_{0,t_{2}}^{4pm}}\left\vert g_{\sigma }(s,z)\right\vert
\dprod\limits_{j=1}^{4pm}\frac{1}{\left\vert s_{j}-s_{j-1}\right\vert
^{H(d+2)}}ds_{1}...ds_{4pm} \\
&\leq &C^{2pm}\sqrt{(4pm)!}\sum_{\sigma \in S(2pm,2pm)}(\int_{\Delta
_{0,t_{2}}^{4pm}}(\left\vert g_{\sigma }^{\ast }(s,z)\right\vert
)^{1+\varepsilon }\dprod\limits_{j=1}^{4pm}\frac{1}{\left\vert
s_{j}-s_{j-1}\right\vert ^{(1+\varepsilon )H(d+2)}}ds_{1}...ds_{4pm})^{\frac{%
1}{1+\varepsilon }}\times \\
&&\times (\int_{\Delta _{0,t_{2}}^{4pm}}\dprod\limits_{j=1}^{4pm}\kappa
_{j}(s_{j})ds_{1}...ds_{4pm})^{\frac{\varepsilon }{1+\varepsilon }}\text{,}
\end{eqnarray*}%
where $\kappa _{j}$ is either $\chi _{\left[ 0,t_{2}\right] }$ or $\chi _{%
\left[ 0,t_{1}\right] }$. Note that $\chi _{\left[ 0,t_{1}\right] }$ appears
exactly $4p$ times in the product $\dprod\limits_{j=1}^{4pm}\kappa _{j}$. On
the other hand, $g_{\sigma }^{\ast }$ is a product of functions given by the
components of the vector field $b_{n}$. Each of those components occur $4p$
times in that product.

So we find that%
\begin{eqnarray*}
&&(\int_{\Delta _{0,t_{2}}^{4pm}}\dprod\limits_{j=1}^{4pm}\kappa
_{j}(s_{j})ds_{1}...ds_{4pm})^{\frac{\varepsilon }{1+\varepsilon }} \\
&\leq &(\int_{\left[ 0,t_{2}\right] ^{4pm}}\dprod\limits_{j=1}^{4pm}\kappa
_{j}(s_{j})ds_{1}...ds_{4pm})^{\frac{\varepsilon }{1+\varepsilon }} \\
&=&((\int_{\left[ 0,t_{2}\right] }\chi _{\left[ 0,t_{1}\right]
}(s)ds)^{4p}(\int_{\left[ 0,t_{2}\right] }\chi _{\left[ 0,t_{2}\right]
}(s)ds)^{4pm-4p})^{\frac{\varepsilon }{1+\varepsilon }} \\
&=&(\left\vert t_{2}-t_{1}\right\vert ^{\frac{\varepsilon }{1+\varepsilon }%
})^{4p}(t_{2})^{(4pm-4p)\frac{\varepsilon }{1+\varepsilon }}\leq (\left\vert
t_{2}-t_{1}\right\vert ^{\frac{\varepsilon }{1+\varepsilon }})^{4p}\text{.}
\end{eqnarray*}%
Further, we also have that%
\begin{eqnarray*}
&&C^{2pm}\sqrt{(4pm)!}\sum_{\sigma \in S(2pm,2pm)}(\int_{\Delta
_{0,t_{2}}^{4pm}}(\left\vert g_{\sigma }^{\ast }(s,z)\right\vert
)^{1+\varepsilon }\dprod\limits_{j=1}^{4pm}\frac{1}{\left\vert
s_{j}-s_{j-1}\right\vert ^{(1+\varepsilon )H(d+2)}}ds_{1}...ds_{4pm})^{\frac{%
1}{1+\varepsilon }} \\
&\leq &C^{2pm}\sqrt{(4pm)!}\sum_{\sigma \in S(2pm,2pm)}\left\Vert g_{\sigma
}^{\ast }(\cdot ,z)\right\Vert _{L^{\infty }(\left[ 0,1\right]
^{4pm})}\times  \\
&&\times (\int_{\Delta _{0,t_{2}}^{4pm}}\dprod\limits_{j=1}^{4pm}\frac{1}{%
\left\vert s_{j}-s_{j-1}\right\vert ^{(1+\varepsilon )H(d+2)}}%
ds_{1}...ds_{4pm})^{\frac{1}{1+\varepsilon }}.
\end{eqnarray*}%
Hence,%
\begin{eqnarray*}
&&\left\vert E\left[ \int_{(\mathbb{R}^{d})^{2pm}}\Lambda _{\alpha
}^{g}(0,t_{2},z)dz\right] \right\vert  \\
&\leq &C^{pm+\left\vert \alpha \right\vert /2}\int_{(\mathbb{R}%
^{d})^{2pm}}(\Psi _{\alpha }^{g}(0,t_{2},z))^{1/2}dz \\
&=&C^{2pm}\int_{(\mathbb{R}^{d})^{2pm}}(\Psi _{\alpha
}^{g}(0,t_{2},z))^{1/2}dz \\
&\leq &K^{2pm}(\int_{(\mathbb{R}^{d})^{2pm}}(C^{2pm}\sqrt{(4pm)!}%
\sum_{\sigma \in S(2pm,2pm)}\left\Vert g_{\sigma }^{\ast }(\cdot
,z)\right\Vert _{L^{\infty }(\left[ 0,1\right] ^{4pm})}\times  \\
&&\times (\int_{\Delta _{0,t_{2}}^{4pm}}\dprod\limits_{j=1}^{4pm}\frac{1}{%
\left\vert s_{j}-s_{j-1}\right\vert ^{(1+\varepsilon )H(d+2)}}%
ds_{1}...ds_{4pm})^{\frac{1}{1+\varepsilon }})^{\frac{1}{2}}(\left\vert
t_{2}-t_{1}\right\vert ^{\frac{\varepsilon }{1+\varepsilon }})^{2p}dz) \\
&\leq &K^{2pm}C^{pm}\sqrt[4]{(4pm)!}\sup_{\sigma \in S(2pm,2pm)}\int_{(%
\mathbb{R}^{d})^{2pm}}(\left\Vert g_{\sigma }^{\ast }(\cdot ,z)\right\Vert
_{L^{\infty }(\left[ 0,1\right] ^{4pm})})^{1/2}dz\times  \\
&&\times \sum_{\sigma \in S(2pm,2pm)}(\int_{\Delta
_{0,t_{2}}^{4pm}}\dprod\limits_{j=1}^{4pm}\frac{1}{\left\vert
s_{j}-s_{j-1}\right\vert ^{(1+\varepsilon )H(d+2)}}ds_{1}...ds_{4pm})^{\frac{%
1}{1+\varepsilon }})^{\frac{1}{2}} \\
&&\times (\left\vert t_{2}-t_{1}\right\vert ^{\frac{\varepsilon }{%
1+\varepsilon }})^{2p}
\end{eqnarray*}%
As in \cite{5BNP19} we observe that%
\begin{eqnarray*}
&&\int_{(\mathbb{R}^{d})^{2pm}}(\left\Vert g_{\sigma }^{\ast }(\cdot
,z)\right\Vert _{L^{\infty }(\left[ 0,1\right] ^{4pm})})^{1/2}dz \\
&=&\int_{(\mathbb{R}^{d})^{2pm}}\prod\limits_{j=1}^{2pm}\left\Vert
g_{j}(\cdot ,z_{j})\right\Vert _{L^{\infty }(\left[ 0,1\right] )}dz \\
&=&\prod\limits_{j=1}^{2pm}\left\Vert g_{j}\right\Vert _{L^{1}(\mathbb{R}%
^{d};L^{\infty }(\left[ 0,1\right] ))}
\end{eqnarray*}%
for all $\sigma $. Hence,%
\begin{eqnarray*}
&&\left\vert E\left[ \int_{(\mathbb{R}^{d})^{2pm}}\Lambda _{\alpha
}^{g}(0,t_{2},z)dz\right] \right\vert  \\
&\leq &(\left\vert t_{2}-t_{1}\right\vert ^{\frac{\varepsilon }{%
1+\varepsilon }})^{2p}C_{p}^{m}\prod\limits_{j=1}^{2pm}\left\Vert
g_{j}\right\Vert _{L^{1}(\mathbb{R}^{d};L^{\infty }(\left[ 0,1\right]
))}\times  \\
&&\times \sqrt[4]{(4pm)!}\sum_{\sigma \in S(2pm,2pm)}(\int_{\Delta
_{0,t_{2}}^{4pm}}\dprod\limits_{j=1}^{4pm}\frac{1}{\left\vert
s_{j}-s_{j-1}\right\vert ^{(1+\varepsilon )H(d+2)}}ds_{1}...ds_{4pm})^{\frac{%
1}{1+\varepsilon }})^{\frac{1}{2}}. \\
&\leq &(\left\vert t_{2}-t_{1}\right\vert ^{\frac{\varepsilon }{%
1+\varepsilon }})^{2p}C_{p}^{m}(\left\Vert b_{n}\right\Vert _{L^{1}(\mathbb{R%
}^{d};L^{\infty }(\left[ 0,1\right] ;\mathbb{R}^{d}))})^{2pm}\times  \\
&&\times \sqrt[4]{(4pm)!}\sum_{\sigma \in S(2pm,2pm)}(\int_{\Delta
_{0,t_{2}}^{4pm}}\dprod\limits_{j=1}^{4pm}\frac{1}{\left\vert
s_{j}-s_{j-1}\right\vert ^{(1+\varepsilon )H(d+2)}}ds_{1}...ds_{4pm})^{\frac{%
1}{1+\varepsilon }})^{\frac{1}{2}}.
\end{eqnarray*}%
On the other hand, Lemma A.7 in \cite{5BNP19} implies for $\varepsilon _{j}=0
$ and $w_{j}=-(1+\varepsilon )H(d+2)$ that%
\begin{eqnarray*}
&&\int_{\Delta _{0,t_{2}}^{4pm}}\dprod\limits_{j=1}^{4pm}\frac{1}{\left\vert
s_{j}-s_{j-1}\right\vert ^{(1+\varepsilon )H(d+2)}}ds_{1}...ds_{4pm} \\
&\leq &C^{4pm}\frac{t_{2}^{-H4pm(1+\varepsilon )(d+2)+4pm}}{\Gamma
(-H4pm(1+\varepsilon )(d+2)+4pm)}.
\end{eqnarray*}%
So altogether, we obtain from (\ref{Start}) that%
\begin{eqnarray*}
&&E\left[ \left\Vert \frac{\partial }{\partial x}X_{t_{2}}^{x,n}-\frac{%
\partial }{\partial x}X_{t_{1}}^{x,n}\right\Vert ^{p}\right] \leq
C(\left\Vert b\right\Vert _{L_{\infty }^{\infty }})\left( \sum_{m\geq
1}\sum_{i\in I}\left\Vert \int_{\Delta _{0,t_{2}}^{m}}\mathcal{H}%
_{i}^{B^{H}}(u)du_{m}...du_{1}\right\Vert _{L^{2p}(\Omega ;\mathbb{R}%
)}\right) ^{p} \\
&\leq &C(\left\Vert b_{n}\right\Vert _{L_{\infty }^{\infty }})\left(
\sum_{m\geq 1}K_{d,p}^{m}\left\vert E\left[ \int_{(\mathbb{R}%
^{d})^{2pm}}\Lambda _{\alpha }^{g}(0,t_{2},z)dz\right] \right\vert
^{1/2p}\right) ^{p} \\
&\leq &C(\left\Vert b_{n}\right\Vert _{L_{\infty }^{\infty }})\left(
\sum_{m\geq 1}K_{d,p}^{m}(\left\vert t_{2}-t_{1}\right\vert ^{\frac{%
\varepsilon }{1+\varepsilon }})C_{p}^{m/2p}(\left\Vert b_{n}\right\Vert
_{L^{1}(\mathbb{R}^{d};L^{\infty }(\left[ 0,1\right] ;\mathbb{R}%
^{d}))})^{m}\times \right.  \\
&&\left. \times \left( \sqrt[4]{(4pm)!}L_{d,p}^{m}\left( C^{4pm}\frac{%
t_{2}^{-H4pm(1+\varepsilon )(d+2)+4pm}}{\Gamma (-H4pm(1+\varepsilon
)(d+2)+4pm)}\right) ^{\frac{1}{2(1+\varepsilon )}}\right) ^{\frac{1}{2p}%
}\right) ^{p} \\
&=&C(\left\Vert b\right\Vert _{L_{\infty }^{\infty }})(\left\vert
t_{2}-t_{1}\right\vert ^{\frac{\varepsilon }{1+\varepsilon }})^{p}\left(
\sum_{m\geq 1}K_{d,p}^{m}C_{p}^{m/2p}(\left\Vert b_{n}\right\Vert _{L^{1}(%
\mathbb{R}^{d};L^{\infty }(\left[ 0,1\right] ;\mathbb{R}^{d}))})^{m}\times
\right.  \\
&&\left. \times \left( \sqrt[4]{(4pm)!}L_{d,p}^{m}\left( C^{4pm}\frac{%
t_{2}^{-H4pm(1+\varepsilon )(d+2)+4pm}}{\Gamma (-H4pm(1+\varepsilon
)(d+2)+4pm)}\right) ^{\frac{1}{2(1+\varepsilon )}}\right) ^{\frac{1}{2p}%
}\right) ^{p}.
\end{eqnarray*}%
\ \ 

By assumption we have that%
\begin{equation*}
H<\frac{1}{2(d+2)}.
\end{equation*}%
Hence, there exists a $\varepsilon \in (0,1)$ such that%
\begin{equation*}
H<\frac{(1-\varepsilon )}{(1+\varepsilon )}\frac{1}{2(d+2)},
\end{equation*}%
which is equivalent to%
\begin{equation*}
\frac{1}{1+\varepsilon }-\frac{1}{2}>H(d+2).
\end{equation*}%
The latter inequality is equivalent to%
\begin{equation*}
-H(d+2)+\frac{1}{1+\varepsilon }>\frac{1}{2}.
\end{equation*}%
Therefore, we get that%
\begin{equation*}
-H4pm(1+\varepsilon )(d+2)+4pm>(1+\varepsilon )2pm.
\end{equation*}%
So, by using Stirling `s formula and the fact that $H<\frac{(1-\varepsilon )%
}{(1+\varepsilon )}\frac{1}{2(d+2)}<\frac{1}{(1+\varepsilon )(d+2)}$, we
find that%
\begin{eqnarray*}
&&\sum_{m\geq 1}(K_{d,p}^{m}C_{p}^{m/2p}(\left\Vert b_{n}\right\Vert _{L^{1}(%
\mathbb{R}^{d};L^{\infty }(\left[ 0,1\right] ;\mathbb{R}^{d}))})^{m}\times \\
&&\times \left( \sqrt[4]{(4pm)!}L_{d,p}^{m}\left( C^{4pm}\frac{%
t_{2}^{-H4pm(1+\varepsilon )(d+2)+4pm}}{\Gamma (-H4pm(1+\varepsilon
)(d+2)+4pm)}\right) ^{\frac{1}{2(1+\varepsilon )}}\right) ^{\frac{1}{2p}}) \\
&\leq &\sum_{m\geq 1}(K_{d,p}^{m}C_{p}^{m/2p}(\left\Vert b_{n}\right\Vert
_{L^{1}(\mathbb{R}^{d};L^{\infty }(\left[ 0,1\right] ;\mathbb{R}%
^{d}))})^{m}\times \\
&&\times T^{-Hm(d+2)+\frac{m}{(1+\varepsilon )}}\left( \sqrt[4]{(4pm)!}%
L_{d,p}^{m}\left( C^{4pm}\frac{1}{\Gamma (-H4pm(1+\varepsilon )(d+2)+4pm)}%
\right) ^{\frac{1}{2(1+\varepsilon )}}\right) ^{\frac{1}{2p}}) \\
&<&\infty
\end{eqnarray*}%
for $T=1$.

Altogether, we see that there exists a continuous function $C:\left[
0,\infty \right) \times \left[ 0,\infty \right) \longrightarrow \left[
0,\infty \right) $ depending on $d,p,H$ and the compact cube $K\subset 
\mathbb{R}^{d}$ such that%
\begin{equation*}
E\left[ \left\Vert \frac{\partial }{\partial x}X_{t_{2}}^{x,n}-\frac{%
\partial }{\partial x}X_{t_{1}}^{x,n}\right\Vert ^{p}\right] (\left\vert
t_{2}-t_{1}\right\vert ^{\frac{\varepsilon }{1+\varepsilon }%
})^{p}\sup_{n\geq 1}C(\left\Vert b_{n}\right\Vert _{L_{\infty }^{\infty
}},\left\Vert b_{n}\right\Vert _{L^{1}(\mathbb{R}^{d};L^{\infty }(\left[ 0,1%
\right] ;\mathbb{R}^{d}))})<\infty
\end{equation*}%
for all $x\in K,$ $n\geq 1_{.}$

Therefore, we can conclude from (\ref{Garcia}) that for all $H<\frac{1}{%
2(d+2)}$ and all compact cubes $K\subset \mathbb{R}^{d}$ there exists a $%
p_{0}=p_{0}(H)>1$ such that for all $p=2^{r}>p_{0}$ there exists a constant $%
C_{d,H,K}$ (which also depends on $\left\Vert b\right\Vert _{L_{\infty
}^{\infty }},\left\Vert b\right\Vert _{L^{1}(\mathbb{R}^{d};L^{\infty }(%
\left[ 0,1\right] ;\mathbb{R}^{d}))}$) such that for all $x,y\in K$: 
\begin{equation*}
E\left[ \left( \sup_{0\leq t\leq 1}\left\Vert X_{t}^{x}-X_{t}^{y}\right\Vert
\right) ^{p}\right] \leq C_{d,H,p,K}\left\Vert x-y\right\Vert ^{p}.
\end{equation*}%
Let $a>1$ and $aq=2^{r}>p_{0}(H)$ for $q>1$. Then there exists a $%
C_{d,H,a,K} $ \ such that for all $x,y\in K$ 
\begin{eqnarray*}
E\left[ \left( \sup_{0\leq t\leq 1}\left\Vert X_{t}^{x}-X_{t}^{y}\right\Vert
\right) ^{a}\right] &\leq &E\left[ \left( \sup_{0\leq t\leq 1}\left\Vert
X_{t}^{x}-X_{t}^{y}\right\Vert \right) ^{aq}\right] ^{\frac{1}{q}} \\
&\leq &C_{d,H,a,K}\left\Vert x-y\right\Vert ^{a}.
\end{eqnarray*}

\end{proof}

\subsection{Pathwise regularization by noise}

We start with a probabilistic result regarding the regularization effect of the averaging operator. From this and with practically the same argument as in \cite{5Shaposhnikov16}, this result is improved to one stated at the pathwise level. 
 
\begin{proposition}\label{5Almost lip.}
Let $H\in\left(0,\frac{1}{2}\right)$ and $\zeta\in\left(0,\frac{1}{2}\right]$ such that $H < \frac{\zeta}{d+2}$. Then, there exist constants $C,\alpha\geq 0$ such that, for any Borel measurable map $b\in L^\infty(\left[r,u\right]\times\mathbb{R}^d,\mathbb{R}^d)$ with $\|b\|_\infty\leq 1$, any Borel measurable functions $h_1,h_2\in L^\infty(\left[r,u\right],\mathbb{R}^d)$ and any $\lambda\geq 0$, the following equality holds:
\begin{equation*}
    \begin{split}
        \mathbb{P}\Bigg[\Big\vert\int_r^u b(s,B_s^H + h_1(s))-b(s,B_s^H + h_2(s)) ds \Big\vert\geq &\lambda \left(u-r\right)^{(1-\zeta)}\|h_1-h_2\|_\infty\Bigg]\\ &\leq C\exp(-\alpha\lambda^2)
    \end{split}
\end{equation*}

\end{proposition}

This is based on the following estimate (compare \cite{5Davie07} and \cite{5Shaposhnikov16}):

\begin{lemma}
\bigskip Let $b\in C_{b}^{1}(\left[ 0,T\right] \times \mathbb{R}^{d};\mathbb{%
R}^{d})$ and $\left\Vert b\right\Vert _{\infty }\leq 1$. Suppose that $H<%
\frac{1}{6}$. Then there exists a $C<\infty $ and a sufficiently small $%
\alpha >0$ (which depend on $H,d,T$, but not on $b$) such that for $0\leq
s<t\leq T$:%
\begin{equation*}
E\left[ \exp (\frac{\alpha }{\left\vert t-s\right\vert ^{2(1-3H)}}\left\Vert
\int_{s}^{t}D_{x}b(u,B_{u}^{H})du\right\Vert ^{2})\right] <C,
\end{equation*}%
where $D_{x}$ is the Fr\'{e}chet derivative of $b$ with respect to the
spatial variable $x$.

\end{lemma}

\begin{proof}
IConsider without loss of generality the case, when $b\in C_{b}^{1}(\left[ 0,T%
\right] \times \mathbb{R}^{d};\mathbb{R})$. Assume that $b$ is smooth and
compactly supported. Then the Clark-Ocone formula combined with the chain
rule for Malliavin derivatives $D^{W}$ (see \cite{5Nualart10}) in the
direction of the Wiener process $W_{t},0\leq t\leq T$ in the representation%
\begin{equation*}
B_{t}^{H}=\int_{0}^{t}K_{H}(t,u)I_{d\times d}dW_{u},0\leq t\leq T
\end{equation*}%
yields for $r>s$%
\begin{eqnarray*}
\nabla _{x}b(r,B_{r}^{H}) &=&E\left[ \nabla _{x}b(r,B_{r}^{H})\right.
\left\vert \mathcal{F}_{s}\right] +\int_{s}^{r}(E\left[ D_{u}^{W}(\nabla
_{x}b(r,B_{r}^{H}))\right. \left\vert \mathcal{F}_{u}\right] )^{\ast }dW_{u}
\\
&=&E\left[ \nabla _{x}b(r,B_{r}^{H})\right. \left\vert \mathcal{F}_{s}\right]
+\int_{s}^{r}K_{H}(r,u)(E\left[ \nabla _{x}^{2}b(r,B_{r}^{H}))\right.
\left\vert \mathcal{F}_{u}\right] )^{\ast }dW_{u}.
\end{eqnarray*}%
So using the stochastic Fubini theorem we get that%
\begin{eqnarray*}
\int_{s}^{t}\nabla _{x}b(r,B_{r}^{H})dr &=&\int_{s}^{t}E\left[ \nabla
_{x}b(r,B_{r}^{H})\right. \left\vert \mathcal{F}_{s}\right]
dr+\int_{s}^{t}\int_{s}^{r}K_{H}(r,u)(E\left[ \nabla
_{x}^{2}b(r,B_{r}^{H}))\right. \left\vert \mathcal{F}_{u}\right] )^{\ast
}dW_{u}dr \\
&=&\int_{s}^{t}E\left[ \nabla _{x}b(r,B_{r}^{H})\right. \left\vert \mathcal{F%
}_{s}\right] dr+\int_{s}^{t}\int_{u}^{t}K_{H}(r,u)(E\left[ \nabla
_{x}^{2}b(r,B_{r}^{H}))\right. \left\vert \mathcal{F}_{u}\right] )^{\ast
}drdW_{u}.
\end{eqnarray*}%
We can also write%
\begin{eqnarray*}
B_{r}^{H} &=&E\left[ B_{r}^{H}\right. \left\vert \mathcal{F}_{u}\right]
+B_{r}^{H}-E\left[ B_{r}^{H}\right. \left\vert \mathcal{F}_{u}\right]  \\
&=&\int_{0}^{u}K_{H}(r,l)I_{d\times
d}dW_{l}+\int_{u}^{r}K_{H}(r,l)I_{d\times d}dW_{l} \\
&=&:Z_{u}^{(1)}+Z_{r,u}^{(2)},r>u\text{.}
\end{eqnarray*}%
Since $Z_{u}^{(1)}$, $Z_{r,u}^{(2)}$ are independent with%
\begin{equation*}
Cov\left[ Z_{r,u}^{(2)}\right] =\sigma _{H}^{2}(r,u)I_{d\times d}
\end{equation*}%
for $\sigma _{H}^{2}(r,u):=\int_{u}^{r}K_{H}^{2}(r,l)dl$, we find that%
\begin{equation*}
E\left[ \nabla _{x}b(r,B_{r}^{H})\right. \left\vert \mathcal{F}_{s}\right]
=((P_{\sigma _{H}^{2}(r,s)}(\frac{\partial }{\partial x_{j}}b(r,\cdot
))(Z_{s}^{(1)}))_{1\leq j\leq d})^{\ast }
\end{equation*}%
as well as 
\begin{equation*}
E\left[ \nabla _{x}^{2}b(r,B_{r}^{H}))\right. \left\vert \mathcal{F}_{u}%
\right] =(P_{\sigma _{H}^{2}(r,u)}(\frac{\partial ^{2}}{\partial
x_{i}\partial x_{j}}b(r,\cdot ))(Z_{u}^{(1)}))_{1\leq i,j\leq d},
\end{equation*}%
where $P_{t},t\geq 0$ is the semigroup associated with the Wiener process.
See for similar arguments in \cite{5NO02} or \cite{5MMNPZ13}.

Further, using standard integration by parts in connection with the heat
kernel we observe that%
\begin{eqnarray*}
P_{\sigma _{H}^{2}(r,s)}(\frac{\partial }{\partial x_{j}}b(r,\cdot ))(y) &=&%
\frac{1}{(2\pi )^{d/2}\sigma _{H}(r,s)}\int_{\mathbb{R}^{d}}\frac{\partial }{%
\partial x_{j}}b(r,x+y)\exp (-\frac{1}{2\sigma _{H}^{2}(r,s)}\left\Vert
x\right\Vert ^{2})dx \\
&=&-\frac{1}{\sigma _{H}^{2}(r,s)}E\left[ b(r,W_{\sigma
_{H}^{2}(r,s)}+y)W_{\sigma _{H}^{2}(r,s)}^{(j)}\right] 
\end{eqnarray*}%
as well as similarly%
\begin{eqnarray*}
&&(P_{\sigma _{H}^{2}(r,u)}(\frac{\partial ^{2}}{\partial x_{i}\partial x_{j}%
}b(r,\cdot ))(y))_{1\leq i,j\leq d} \\
&=&\frac{1}{\sigma _{H}^{4}(r,u)}E\left[ b(r,W_{\sigma
_{H}^{2}(r,u)}+y)(W_{\sigma _{H}^{2}(r,u)}\otimes W_{\sigma
_{H}^{2}(r,u)}-\sigma _{H}^{2}(r,u)I_{d\times d})\right] \text{.}
\end{eqnarray*}%
So we obtain that 
\begin{eqnarray*}
&&\int_{s}^{t}\nabla _{x}b(r,B_{r}^{H})dr \\
&=&\int_{s}^{t}-\frac{1}{\sigma _{H}^{2}(r,s)}(\left. E\left[ b(r,W_{\sigma
_{H}^{2}(r,s)}+y)W_{\sigma _{H}^{2}(r,s)}\right] \right\vert
_{y=Z_{s}^{(1)}})^{\ast }dr \\
&&+\int_{s}^{t}\int_{u}^{t}K_{H}(r,u)\frac{1}{\sigma _{H}^{4}(r,u)}\times  \\
&&\times (\left. E\left[ b(r,W_{\sigma _{H}^{2}(r,u)}+y)(W_{\sigma
_{H}^{2}(r,u)}\otimes W_{\sigma _{H}^{2}(r,u)}-\sigma
_{H}^{2}(r,u)I_{d\times d})\right] \right\vert _{y=Z_{u}^{(1)}})^{\ast
}drdW_{u}.
\end{eqnarray*}%
We also know that $\sigma _{H}(r,s)\geq c_{H}\left\vert r-s\right\vert ^{H}$%
. Thus we get that 
\begin{eqnarray*}
&&\left\Vert \int_{s}^{t}-\frac{1}{\sigma _{H}^{2}(r,s)}(\left. E\left[
b(r,W_{\sigma _{H}^{2}(r,s)}+y)W_{\sigma _{H}^{2}(r,s)}\right] \right\vert
_{y=Z_{s}^{(1)}})^{\ast }dr\right\Vert  \\
&\leq &\left\Vert b\right\Vert _{\infty }\int_{s}^{t}\frac{1}{\sigma
_{H}^{2}(r,s)}E\left[ \left\Vert W_{\sigma _{H}^{2}(r,s)}\right\Vert \right]
dr\leq C(H,d)\left\Vert b\right\Vert _{\infty }\left\vert t-s\right\vert
^{1-H}.
\end{eqnarray*}%
In addition, we find that%
\begin{eqnarray*}
&&\left\Vert \int_{u}^{t}K_{H}(r,u)\frac{1}{\sigma _{H}^{4}(r,u)}(\left. E%
\left[ b(r,W_{\sigma _{H}^{2}(r,u)}+y)(W_{\sigma _{H}^{2}(r,u)}\otimes
W_{\sigma _{H}^{2}(r,u)}-\sigma _{H}^{2}(r,u)I_{d\times d}\right]
\right\vert _{y=Z_{u}^{(1)}})^{\ast }dr\right\Vert  \\
&\leq &\left\Vert b\right\Vert _{\infty }\int_{u}^{t}K_{H}(r,u)\frac{1}{%
\sigma _{H}^{4}(r,u)}(E\left[ (\left\Vert W_{\sigma _{H}^{2}(r,u)}\otimes
W_{\sigma _{H}^{2}(r,u)}\right\Vert \right] +\sigma _{H}^{2}(r,u))dr \\
&\leq &C(H,d)\left\Vert b\right\Vert _{\infty }(\frac{1}{\frac{1}{2}-3H}%
\left\vert t-u\right\vert ^{\frac{1}{2}-3H}+\frac{1}{\frac{1}{2}-H}%
\left\vert t-u\right\vert ^{\frac{1}{2}-H}).
\end{eqnarray*}%
Define%
\begin{eqnarray*}
A(u) &=&(a_{ij}(u))_{1\leq i,j\leq d} \\
&=&\int_{u}^{t}K_{H}(r,u)\frac{1}{\sigma _{H}^{4}(r,u)}(\left. E\left[
b(r,W_{\sigma _{H}^{2}(r,u)}+y)(W_{\sigma _{H}^{2}(r,u)}\otimes W_{\sigma
_{H}^{2}(r,u)}-\sigma _{H}^{2}(r,u)I_{d\times d})\right] \right\vert
_{y=Z_{u}^{(1)}})^{\ast }dr.
\end{eqnarray*}%
Then%
\begin{equation*}
\left\Vert \int_{s}^{t}A(u)dW_{u}\right\Vert ^{2}\leq \frac{1}{2}%
\sum_{i=1}^{d}\sum_{j_{1},j_{2}=1}^{d}((%
\int_{s}^{t}a_{ij_{1}}(u)dW_{u}^{(j_{1})})^{2}+(%
\int_{s}^{t}a_{ij_{2}}(u)dW_{u}^{(j_{2})})^{2}).
\end{equation*}%
So for $\beta >0$%
\begin{eqnarray*}
&&E\left[ \exp (\beta \left\Vert \int_{s}^{t}A(u)dW_{u}\right\Vert ^{2})%
\right]  \\
&\leq &E\left[ \exp (\beta
d\sum_{i=1}^{d}\sum_{j_{1}=1}^{d}(%
\int_{s}^{t}a_{ij_{1}}(u)dW_{u}^{(j_{1})})^{2})\right] ^{1/2}E\left[ \exp
(\beta
d\sum_{i=1}^{d}\sum_{j_{2}=1}^{d}(%
\int_{s}^{t}a_{ij_{2}}(u)dW_{u}^{(j_{2})})^{2})\right] ^{1/2} \\
&\leq &\dprod\limits_{i,j=1}^{d}E\left[ \exp (\beta
d^{3}(\int_{s}^{t}a_{ij}(u)dW_{u}^{(j)})^{2})\right] ^{1/d^{2}}.
\end{eqnarray*}%
The theorem of Dambis-Dubins (see e.g. \cite{KaratzasShreve}) also shows that%
\begin{equation*}
M_{h}^{(i,j)}:=\int_{s}^{h}a_{ij}(u)dW_{u}^{(j)}=B_{\left\langle
M^{(i,j)}\right\rangle _{h}}^{(i,j)},h\leq t
\end{equation*}%
where $B_{u}^{(i,j)},u\geq 0$ is a Brownian motion with respect to a certain
filtration (on a possibly extended probability space) for each $i,j$. Since%
\begin{eqnarray*}
\left\langle M^{(i,j)}\right\rangle _{t}
&=&\int_{s}^{t}(a_{ij}(u))^{2}du\leq \int_{s}^{t}\left\Vert A(u)\right\Vert
^{2}du \\
&\leq &C(H,d)\left\Vert b\right\Vert _{\infty }^{2}\int_{s}^{t}(\frac{1}{%
\frac{1}{2}-3H}\left\vert t-u\right\vert ^{\frac{1}{2}-3H}+\frac{1}{\frac{1}{%
2}-H}\left\vert t-u\right\vert ^{\frac{1}{2}-H})^{2}du \\
&\leq &C(H,d,T)\left\Vert b\right\Vert _{\infty }^{2}\left\vert
t-s\right\vert ^{2-6H}=:\gamma (t,s)
\end{eqnarray*}%
and 
\begin{equation*}
(M_{t}^{(i,j)})^{2}\leq (\sup_{0\leq l\leq \gamma (t,s)}\left\vert
B_{l}^{(i,j)}\right\vert )^{2}\overset{law}{=}\gamma (t,s)(\sup_{0\leq l\leq
1}\left\vert B_{l}\right\vert )^{2}
\end{equation*}%
for a Brownian motion $B_{l},0\leq l\leq 1$. Hence, it follows from Fernique 
\'{}%
s theorem that there exists a sufficiently small $\beta =\beta (H,d,T)>0$
such that%
\begin{eqnarray*}
E\left[ \exp (\beta \left\Vert \int_{s}^{t}A(u)dW_{u}\right\Vert ^{2})\right]
&\leq &\dprod\limits_{i,j=1}^{d}E\left[ \exp (\beta d^{3}\gamma
(t,s)(\sup_{0\leq l\leq 1}\left\vert B_{l}^{(i,j)}\right\vert )^{2})\right]
^{1/d^{2}} \\
&=&E\left[ \exp (\beta d^{3}\gamma (t,s)(\sup_{0\leq l\leq 1}\left\vert
B_{l}\right\vert )^{2})\right] <\infty .
\end{eqnarray*}%
Altogether, we see that there exists a $C(H,d,T)>0$ and a sufficiently small 
$\alpha =\alpha (H,d,T)>0$ such that%
\begin{equation*}
E\left[ \exp (\alpha \left\Vert \int_{s}^{t}D_{x}b(u,B_{u}^{H})du\right\Vert
^{2})\right] \leq E\left[ \exp (\alpha C(H,d,T)\left\Vert b\right\Vert
_{\infty }^{2}\left\vert t-s\right\vert ^{2-6H}(1+\sup_{0\leq l\leq
1}\left\vert B_{l}\right\vert )^{2})\right] <\infty .
\end{equation*}

The general case of vector fields $b\in C_{b}^{1}(\left[ 0,T\right] \times 
\mathbb{R}^{d};\mathbb{R}^{d})$ is obtained by using approximation combined
with Fatou 
\'{}%
s Lemma.

\end{proof}

\begin{proof}[Proof of Proposition \ref{5Almost lip.}]
The proof follows directly from an approximation argument combined with Chebyshev's inequality and the previous lemma. 
\end{proof}


In what follows we will need the following functional space:

\begin{equation*}
\begin{split}
Lip_N([r,u],\mathbb{R}^d) &:=  \\ 
& \{h\in \mathcal{C}([r,u],\mathbb{R}^d) : \|h(t)-h(s)\|\leq |t-s|, (t,s)\in[r,u]^2,\max_{s\in[r,u]}\|h(s)\|\leq N\}
\end{split}
\end{equation*}

endowed with the uniform metric.

Now we can state the main result of this section.

\begin{proposition}
There exists constants $C$ and $\nu>0$, independent of $l:=u-r$, and a set $\Omega^{'}$ such that \[P(\Omega \backslash \Omega^{'})\leq C\exp(-l^{-\nu})\]

and for any $h_1,h_2\in Lip_N([r,u],\mathbb{R}^d)$ with $\|h_1-h_2\|\leq 4l$, $\omega\in\Omega^{'}$ the following estimate holds:

\[\|\int_r^u b(s,B_s^H(\omega) + h_1(s))-b(s,B_s^H(\omega) + h_2(s)) ds \| \leq Cl^{\frac{4}{3}}\]
\end{proposition}

\begin{proof}
The proof follows from Proposition \ref{5Almost lip.}, choosing $\zeta=\frac{1}{2}$, following the same reasoning as in Lemma 3.6 in \cite{5Shaposhnikov16}.
\end{proof}

\subsection{Proof of the main theorem}

Now in order to proof Theorem 2.2, one proceeds as in the proof of Theorem 1.1 in \cite{5Shaposhnikov16} without any changes. For the sake of completeness, we include the proof of this theorem adapted to our setting

\begin{proof}
	Without loss of generality suppose again that $\|b\|_{\infty}\leq 1$. Fix an $N\geq 1$ and let $Lip_N([r,u],\mathbb{R}^d)$ be defined as above. Without loss of generality assume $T=1$. For each $k\geq 1$ divide the interval $[0,1]$ into $M=2^k$ closed subintervals \[\left[0,\frac{1}{M}\right],\cdots,\left[\frac{(M-1)}{M},1\right]\]  
	and for each of the sub-intervals we apply Proposition (2.15) to get sets, say, $\Omega_{k,i}$ where the conclusions of the proposition hold. Let \[\Omega_k := \bigcap_{i = 0}^{M - 1}\Omega_{k, i}. \]
	
	The Borel-Cantelli lemma gives us that the set \[\Omega^{**} := \liminf_{k \to \infty}\Omega_k = \bigcup_{K = 1}^{\infty}\bigcap_{k = K}^{\infty}\Omega_k\] has full measure.
	
	Proposition (2.3) point 4, with $\eta = 1$ and 
	\[S = \{S_n\}_{n = 1}^{\infty} =  \left\{k/{2^n}; k \in \{0, 1, \ldots, 2^n - 1\}\right\}\] 
	
	gives the following estimate
	\[
	|X^{s,x}_t(\omega) - X^{s,y}_t(\omega)| \leq
	C(\alpha, T, N, S, \omega)|x - y|^{\alpha}, \
	|x - y| \leq \frac{1}{2^n}, s \in S_n
	\]
	where \[
	(s,t,x,\omega)\mapsto X^{s,x}_t(\omega)=:X(x,s,t,B^H(\omega))\quad\textit{with } (s,t)\in\Delta_{0,T}^2,x\in\mathbb{R}^d,\omega\in\Omega^*
	\]
	is the flow map from Proposition (2.3) and $\Omega^*$ is a modification of $\Omega^{**}$ of full measure.\par
	
	We now show that for any $\omega\in\Omega^*$ such that \[|x| + T\|b\|_{\infty} + \max_{u\in[0,1]}|B^H_u(\omega)| = |x| + 1 + \max_{u\in[0,1]}|B^H_u(\omega)| \leq N,\] (\ref{5SDE}) has a unique solution.\par
	Take such an $\omega\in\Omega^*$ and let $Y_t$ be any solution to equation
	(\ref{5SDE}) in $\mathcal{C}\left( [0,T];\mathbb{R}^d \right)$, for this $\omega$, starting from
	$x\in\mathbb{R}^d$, we shall show that then $Y_t = X(x,0,t,B^H(\omega))$. Due
	to the construction of $\Omega^*$ there exists $K$, which depends on the path
	$B^H(\omega)$ such that for all $k\geq K$ the path $B^H(\omega)\in\Omega_k$.
	Let 
	\[
	M' = 2^{k'}, \quad r = \frac{i}{M'}, \quad \text{where} \ k' \geq K.
	\]
	Define the following auxiliary function on the interval $[0, r]$:
	\[
	f(t) := X(x, 0, r, B^H(\omega)) - X(Y_t, t, r, B^H(\omega)).
	\]
	
	We observe that for any $s \leq t$, the flow property says that
	\begin{equation*}
	\begin{split}
	f(t) - f(s) =& X\left(Y_s, s, r, B^H(\omega)\right) - X\left(Y_t, t, r, B^H(\omega)\right) \\
	=& X\left(X(Y_s, s, t, B^H(\omega)),t, r, B^H(\omega)\right)-X\left(Y_t, t, r, B^H(\omega)\right).
	\end{split}
	\end{equation*}
	
	The difference $X(Y_s, s, t, B^H(\omega))-Y_t$ can be represented as follows:
	\begin{equation*}
	\begin{split}
	X(Y_s, s, t, B^H(\omega)) - Y_t =
	\int_{s}^{t}b\Big(u, Y_s +& B^H_u(\omega) - B^H_s(\omega) + \int_{s}^{u}b(r, X_r) dr\Big) du - \\  &\int_{s}^{t}b\left(u, Y_s + B^H_u(\omega) - B^H_s(\omega) + \int_{s}^{u}b(r, Y_r) dr\right) du
	\end{split}
	\end{equation*}
	
	Hence 
	\begin{equation*}
	\begin{split}
	X(Y_s, s, t, B^H(\omega)) - Y_t = \int_{s}^{t}b\left(u, B^H_u(\omega) + h_1(u)\right)\,du - \int_{s}^{t}b\left(u, B^H_u(\omega) + h_2(u)\right) du,
	\end{split}
	\end{equation*}
	
	where
	$$
	h_1(u) = Y_s - B^H_s(\omega) + \int_{s}^{u}b(r, X_r) dr, \ \ h_2(u) =Y_s - B^H_s(\omega) + \int_{s}^{u}b(r, Y_r) dr .
	$$
	Let $k \geq k'$ and $M = 2^{k}$ and take $s, t$ to be $\frac{i}{M}$ and $\frac{i + 1}{M}$, respectively. Then by the regularization property of the averaging operator we obtain the following estimate:
	\[
	|X(Y_s, s, t, B^H(\omega))-Y_t| \leq \frac{C}{M^{\frac{4}{3}}}
	\]
	Taking $M$ sufficiently large this yields
	\[
	|X(Y_s, s, t, B^H(\omega))-Y_t| \leq \frac{1}{M}.
	\]
	
	Hence there exists a positive constant $C = C(N, S, B^H(\omega))$ such that
	\[
	|f(t) - f(s)| \leq C|X(Y_s, s, t, B^H(\omega))-Y_t|^{\frac{4}{5}}.
	\]
	
	\[
	\left|f\left(\frac{i + 1}{M}\right) - f\left(\frac{i}{M}\right)\right| \leq \left(\frac{C}{M^{\frac{4}{3}}}\right)^{\frac{4}{5}},
	\]
	and consequently
	\[
	|f(r)| \leq \sum_{i=0}^{2^k}\left|f\left(\frac{i + 1}{M}\right) - f\left(\frac{i}{M}\right)\right|\leq \frac{C\times M}{M^{\frac{16}{15}}}=\frac{C}{M^{\frac{1}{15}}}.
	\]
	Due to the arbitrariness of $k$ we conclude
	\[
	f(r) = X(x, 0, r, B^H(\omega)) - Y_r = 0.
	\]
	Since $r$ was an arbitrary dyadic number in $[0, 1]$  with  a sufficiently large denominator,
	the continuity of $Y_t$ and $X(x, 0, t, B^H(\omega))$ implies the equality $Y_t = X(x, 0, t, B^H(\omega))$ for each $t \in [0, T]$.
	The proof is complete.
\end{proof}

\section{Applications to the Transport equation (TE) and Continuity equation (CE)}
\label{5ApplicationsToPDE}

\subsection{Existence of unique regular solutions to deterministic transport equations for perturbed velocity fields $b\in L_{\infty ,\infty }^{1,\infty }$}

As an application of the results in Section 2, we will exhibit a regularization by noise phenomena by restoring well-posedness of the following transport equation

\begin{equation}\label{5TEoriginal}
\begin{cases}
\frac{\partial}{\partial t} u(t,x) + b(t,x)\cdot\nabla_x u(t,x) = 0\text{	,	}(t,x)\in\left[0,T\right]\times\mathbb{R}^d, \\
u(0,x) = u_0(x) \text{	on	}\mathbb{R}^d,
\end{cases}
\end{equation}

when $b\in L_{\infty ,\infty }^{1,\infty }$.\par

We will show that the following transport equation

\begin{equation}\label{5TEperturbed}
\begin{cases}
\frac{\partial}{\partial t} u(t,x) + b^*(t,x)\cdot\nabla_x u(t,x) = 0\text{	,	}(t,x)\in\left[0,T\right]\times\mathbb{R}^d, \\
u(0,x) = u_0(x) \text{	on	}\mathbb{R}^d,
\end{cases}
\end{equation}

when $b^* = b\circ\Psi$, for a certain $\Psi$, has a unique (weak) solution $u\in W^{1,\infty}_{loc}\left(\left[0,T\right]\times\mathbb{R}^d\right)$ with $u(t,\cdot)\in\bigcap_{p\geq 1}W^{k,p}_{loc}\left(\mathbb{R}^d\right)$.

In proving such a result, which cannot be treated in the framework of renormalized solutions of \cite{5DiPernaLions89}, \cite{5Ambrosio04}, we need some auxiliary results.\par
If $b\in L_{\infty ,\infty }^{1,\infty }$, $H< \left(\frac{1}{2(d+3)}\land\frac{1}{2(d+2k-1)}\right)$ we know from Theorem \ref{5StrongSolution}, Theorem \ref{5PathByPath} and Lemma \ref{5JointHolderContinuity} that there exists a measurable set $\tilde{\Omega}$ with full measure such that for all $\omega\in\tilde{\Omega}$ and $x\in\mathbb{R}^d$ the solution path of the strong solution uniquely solves (2.3) in the space of continuous functions. Moreover, for all $\omega\in\tilde{\Omega}$ we have that, for all $t\in[0,T]$ 
\begin{equation}
 \left(x\mapsto X^{x}_t(\omega)\right) \in \bigcap_{p\geq 1}W^{k,p}_{loc}\left(\mathbb{R}^d;\mathbb{R}^d\right)
 \label{5differentiableforwardflow}
\end{equation}
as well as 
\begin{equation}
\left(\left(t,x\right)\mapsto X^{x}_t(\omega)\right)
\label{5Holderregularityforwardflow}
\end{equation}
is locally H\"{o}lder continuous on $[0,T]\times\mathbb{R}^d$. On the other hand, by using backward SDE's (see e.g. \cite{5MNP14}) and the same arguments as in Section 2, we find that there exists an adapted process $Y^x_{\cdot}$ such that for all $\omega\in\tilde{\Omega}$ $\left(\left(t,x\right)\mapsto Y^{x}_t(\omega)\right)$ is the global inverse of $\left(\left(t,x\right)\mapsto X^{x}_t(\omega)\right)$ and such that (\ref{5differentiableforwardflow}) and (\ref{5Holderregularityforwardflow}) are satisfied.\par

Our approach will be to employ a variation of the classical method of characteristics in connection with construction of solutions to (\ref{5TEperturbed}). To this end, we consider the ODE
\begin{equation}
	\frac{d}{dt} \hat{X}^x_t = b^*(t,\hat{X}^x_t),\quad \hat{X}^x_0 = x\in\mathbb{R}^d\text{ , }t\in\left[0,T\right]
	\label{5transformeddriftODE}
\end{equation}

for the vector field $b^*(t,x) = b(t,x+B^H(\omega))$ for fixed $\omega\in\tilde{\Omega}$ and where $b$ and $H$ are as before. \par
By using a simple transformation we see that $\hat{X}^x_{\cdot}(\omega)$ given by  \[\hat{X}^x_{t}(\omega)  = {X}^x_{t}(\omega) - {B}^H_{t}(\omega) \]
is the unique solution of (\ref{5transformeddriftODE}) for all $\omega\in\tilde{\Omega}$, $x\in\mathbb{R}^d$. Similarly, we can find an inverse solution process $\hat{Y}^x_{\cdot}(\omega)$ of $\hat{X}^x_{\cdot}(\omega)$ (with respect to the corresponding backward ODE)
and observe that both processes have the properties (\ref{5differentiableforwardflow}) and (\ref{5Holderregularityforwardflow}) (for the same $\hat{\Omega}$).

\begin{lemma}
	Let $H<\frac{1}{2(d+3)}$. Then there exists a measurable set $\Omega^{'}$ with full measure such that for all $\omega\in\Omega^{'}$ 
	\[
	\left(\left(t,x\right)\mapsto \hat{Y}^{x}_t(\omega)\right)\in W^{1,\infty}_{loc}\left(\left[0,T\right]\times\mathbb{R}^d;\mathbb{R}^d\right)
	\]
\end{lemma}

\begin{proof}
		One can show in exactly the same way as in Theorem \ref{5Identification} that $\left(\left(t,x\right)\mapsto \frac{\partial}{\partial x}\hat{Y}^{x}_t\right)$ has a continuous version . So the same also applies to $\hat{Y}^{x}_{\cdot}$. \par
		On the other hand, we see that $\hat{Y}^{x}_{\cdot}$ is absolutely continuous with the bounded weak derivative \[-b^*(t,\hat{Y}^{x}_{t}).\]
		So we find a measurable set $\Omega^{'}$ of full measure with the required property.
\end{proof}

In the sequel we also want to apply the following chain rule for Sobolev functions.

\begin{lemma}
	Let $f\in\mathcal{C}_b^1(\mathbb{R}^d;\mathbb{R})$ and $u\in W^{1,1}_{loc}\left(\mathbb{R}^n;\mathbb{R}^d\right)$. Then \[f\circ u\in  W^{1,1}_{loc}\left(\mathbb{R}^n;\mathbb{R}^d\right)\]
	and 
	\[D_x\left(f\circ u\right)= \nabla f\circ D_x u\]
\end{lemma}

\begin{proof}
	The proof follows from the classical chain rule applied to a sequence $(u_i)_{i\geq 1}$ in $\mathcal{C}^{\infty}\left(U\right)$, for $U$ a bounded open set, which approximates $u$ in $W^{1,1}\left(U;\mathbb{R}^d\right)$.
\end{proof}

\begin{remark}
	In a similar way to the previous lemma, one can show that if ${f\in\mathcal{C}_b^k(\mathbb{R}^d;\mathbb{R})}$ and $u\in\bigcap_{p\geq 1} W^{k,p}_{loc}\left(\mathbb{R}^n;\mathbb{R}^d\right)$ then \[f\circ u\in \bigcap_{p\geq 1} W^{k,p}_{loc}\left(\mathbb{R}^n;\mathbb{R}\right).\]
\end{remark}

We also need the following special version of a change of variable formula, which can be found in e.g (\cite{5Hajlasz93}). 

\begin{theorem}\label{5ChangeOfVariableFormula}
	Let $\mathcal{O}\subseteq\mathbb{R}^d$ be an open set and $f\in W^{1,p}_{loc} (\mathcal{O})$, where $p>d$. Denote by $\mathcal{J}_f$ the Jacobian of $f$ and for $E\subseteq\mathcal{O}$ the function $N_f(\cdot , E) : \mathbb{R}^d\mapsto \mathbb{N}\cup\{\infty\}$, the \textit{ Banach indicatrix of f}, by $\textmd{N}_f(y , E) := \textit{card}\Bigl(f^{-1}\left(\{y\}\right)\bigcap E\Bigr)$.\par
	If $u:\mathbb{R}^d \rightarrow \mathbb{R}$ is a measurable function and $E\subseteq\mathcal{O}$ is measurable, then 
	\[\int_{E}\left(u\circ f\right)\vert\mathcal{J}_f\vert dx = \int_{\mathbb{R}^d} u(y) N_f(y , E) dy
	\] provided that at least one of $\left(u\circ f\right)\vert\mathcal{J}_f\vert$ or $u(y) N_f(y , E) $ is integrable.
\end{theorem}

Before stating our main result, we give a definition for a solution to our transport equation:
\begin{definition}
	Let $b^*\in L^{\infty}\left([0,T]\times\mathbb{R}^d\right)$ and $u_0\in \mathcal{C}(\mathbb{R}^d)$ in (\ref{5TEperturbed}). Then we call a function $u\in\mathcal{C}\left([0,T]\times\mathbb{R}^d\right)\cap W^{1,1}_{loc}\left([0,T]\times\mathbb{R}^d\right)$ with $u(0,x)=u_0(x)$, for all $x\in\mathbb{R}^d $, a $\textit{{weak solution to the transport equation (\ref{5TEperturbed})}}$, if for all $\rho\in\mathcal{C}_c^{\infty}\left(\left(0,T\right)\right)$ and $\eta\in\mathcal{C}_c^{\infty}\left(\mathbb{R}^d\right)$
	\begin{equation}\label{5weakformulation}
	 \int_{0}^{T}\int_{\mathbb{R}^d}\{-u(t,x)\rho^{'}(t)\eta(x)  + b^*(t,x)\cdot \nabla u(t,x)\rho(t)\eta(x)\} dx dt = 0,
	\end{equation}
    where $\rho^{'}$ denotes the derivative of $\rho$.
\end{definition}

We can now state our main result in this section:

\begin{theorem}
	Let $b\in L_{\infty ,\infty }^{1,\infty} $, $H< \left(\frac{1}{2(d+3)}\land\frac{1}{2(d+2k-1)}\right)$, $u_0\in\mathcal{C}_b^k(\mathbb{R}^d)$, $p>d$, and $k\geq 2$. Define $b^*$ in (\ref{5weakformulation}) as \[b^*(t,x)  = b(t,x+B_t^H(\omega)).
	\]
	
	Then there exists a measurable set $\tilde{\Omega}$ with $\mu(\tilde{\Omega}) =1$ such that for all $\omega\in\tilde{\Omega}$ there exists a unique weak solution $u = u_{\omega}$ in the class  $W^{1,p}_{loc}\left([0,T]\times\mathbb{R}^d\right)$. Moreover \[u(t,\cdot)\in  \cap_{p\geq 1} W^{k,p}_{loc}\left(\mathbb{R}^d\right)\] for all $t$.
\end{theorem}

\begin{proof}
	\begin{enumerate}
		\item \textbf{Existence:}
		Let $(b_n)_{n\geq 1}\subseteq\mathcal{C}_c^{\infty}\left([0,T]\times\mathbb{R}^d\right)$ be an approximating sequence for $b$ as in Theorem \ref{5StrongConvergenceDerivative}. Then for the process $\hat{Y}^{x,n}_{\cdot}$ associated with the coefficients $b_n$, the process $u_n(\omega,t,x) = u_0(\hat{Y}^{x,n}_{t})$ satisfies (\ref{5weakformulation}) (for an appropriate version of $\hat{Y}^{x,n}_{\cdot}$) a.e., that is for all $\rho\in\mathcal{C}_c^{\infty}\left(\left(0,T\right)\right)$ and $\eta\in\mathcal{C}_c^{\infty}\left(\mathbb{R}^d\right)$
		\begin{equation}\label{5weakformulationApprox}
			\begin{split}
			0 =& \int_{0}^{T}\int_{\mathbb{R}^d}\{-u_n(\omega,t,x)\rho^{'}(t)\eta(x)  + b^*_n(t,x)\cdot \nabla u_n(\omega,t,x)\rho(t)\eta(x)\} dx dt \\
					0 =& \int_{0}^{T}\int_{\mathbb{R}^d}\{-u_0(\hat{Y}^{x,n}_{t}(\omega))\rho^{'}(t)\eta(x)\\  &+ b_n(t,x + B^H_t(\omega))\cdot (\nabla u_0)(\hat{Y}^{x,n}_{t}(\omega))^{T}\frac{\partial}{\partial x}\hat{Y}^{x,n}_{t}(\omega)\rho(t)\eta(x)\} dx dt\textit{,		for a.a. }\omega .
			\end{split}
		\end{equation}
	Using the same proof as in Lemma \ref{5ContinuousVersion}, we find that \[\Big(\left(t,x\right)\mapsto \hat{Y}^{x,n_j}_{t}\Big)_{j\geq 1}\] and \[\Big(\left(t,x\right)\mapsto \frac{\partial}{\partial x}\hat{Y}^{x,n_j}_{t}\Big)_{j\geq 1}\]
	converge in $\mathcal{C}\left([0,T]\times K ; L^2(\Omega;\mathbb{R}^d)\right)$ and $\mathcal{C}\left([0,T]\times K ; L^2(\Omega;\mathbb{R}^{d\times d})\right)$ respectively for a subsequence which only depends on the compact cube $K$. Further we have that \[Y(t,x,\omega) = \hat{Y}^{x}_{t}(\omega)\text{	,	}Y^{'}(t,x,\omega) = \frac{\partial}{\partial x}\hat{Y}^{x}_{t}(\omega)\] $(t,x,\omega)$-a.e. (See Theorem \ref{5Identification}).\par

	Now, since $B^H_t$ has a Gaussian density and $u_0\in \mathcal{C}_b^k(\mathbb{R}^d)$ for $k\geq 2$, we obtain that the right hand side of (\ref{5weakformulationApprox}) converges for $n\rightarrow\infty$ in $L^1(\Omega)$ to the expression%
	{\small
	\[ \int_{0}^{T}\int_{\mathbb{R}^d}\{-u_0(\hat{Y}^{x}_{t}(\omega))\rho^{'}(t)\eta(x)+ b(t,x + B^H_t(\omega))\cdot (\nabla u_0)(\hat{Y}^{x}_{t}(\omega))^{T}\frac{\partial}{\partial x}\hat{Y}^{x}_{t}(\omega)\rho(t)\eta(x)\} dx dt,\]}%
	which must be zero $\omega$-a.e.\par
	
		Because of Lemma 3.1 and Lemma 3.2 we know that, $\omega$-a.e., \[(\nabla u_0)(\hat{Y}^{x}_{t}(\omega))^{T}\frac{\partial}{\partial x}\hat{Y}^{x}_{t}(\omega) = \nabla u(\omega,t,x)\quad (t,x) \text{-a.e.	}\]
		
		Hence we obtain that
		\begin{equation}\label{5moduloseperability}
			\int_{0}^{T}\int_{\mathbb{R}^d}\{-u (\omega,t,x)\rho^{'}(t)\eta(x)  + b^* (t,x)\cdot \nabla u (\omega,t,x)\rho(t)\eta(x)\} dx dt = 0
		\end{equation}
		$\omega\text{-a.e.}$
		 
		 Because of separability of Sobolev spaces, when $p<\infty$, we also see that there exists a $\Omega^+$ with $\mu(\Omega^+) = 1$ such that for all  $\rho\in\mathcal{C}_c^{\infty}\left(\left(0,T\right)\right)$ and $\eta\in\mathcal{C}_c^{\infty}\left(\mathbb{R}^d\right)$ the condition (\ref{5moduloseperability}) holds.\par
		 Altogether, using in addition Remark 3.3 we can find a set $\Omega^*$ with $\mu(\Omega^*) = 1$, which has the required properties with respect to existence and regularity of solutions in the statement of the theorem.

		 \item \textbf{Uniqueness:}  For $\omega\in\Omega^*$, let $u=u_{\omega}$ be a weak solution in $W^{1,p}_{loc}\left([0,T]\times\mathbb{R}^d\right)$. We want to show that for $\rho\in\mathcal{C}_c^{\infty}\left(\left(0,T\right)\right)$ and $\eta\in\mathcal{C}_c^{\infty}\left(\mathbb{R}^d\right)$
		 
		 \begin{equation}\label{5constantcharacteristics}
		 	0 = \int_{0}^{T}\int_{\mathbb{R}^d} \Big(u (t,\hat{X}^x_t(\omega))\rho^{'}(t)\eta(x) \Big) dx dt
		 \end{equation}
		 
		 To this end , let $\mathcal{U}\subseteq\mathbb{R}^d$ be an open bounded set with $$\mathcal{U} \supseteq \Big(\text{supp}(\eta)\cup \hat{X}(\omega)\big([0,T]\times\text{supp}(\eta) \big)\Big)$$ and let $\left(u_n\right)_{n\geq 1}\subset \mathcal{C}^{\infty}\left((0,T)\times\mathcal{U}\right)$ such that $u_n \xrightarrow[n\rightarrow\infty]{} u$ in $W^{1,p}\left((0,T)\times\mathcal{U}\right)$.\par
		 
		Then using the chain rule of Lemma 3.2 (for $f\in\mathcal{C}_b^1(\mathcal{U};\mathbb{R})$), integration by parts and Theorem \ref{5ChangeOfVariableFormula} ($E = \mathcal{O} = \mathcal{U} = \hat{Y}^{-1}_t(\omega)\left(O_t\right)$) we find that

		 \begin{equation}\label{5constantcharacteristicsApprox}
		 \begin{split}
		 	&\int_{0}^{T}\int_{\mathbb{R}^d} \Big(u_n (t,\hat{X}^x_t(\omega)) \rho^{'}(t)\eta(x) \Big)  dx dt  \\ &=-\int_{0}^{T}\int_{\mathbb{R}^d} \left\{\left(\frac{\partial}{\partial t} u_n (t,\hat{X}^x_t(\omega)) + b^*(t,\hat{X}^x_t(\omega))\cdot(\nabla u_n) (t,\hat{X}^x_t(\omega)) )\right)\rho^{ }(t)\eta(x) \right\}  dx dt		 \\
		 	&=-	\int_{0}^{T}\int_{\mathcal{U}} \left\{\left(\frac{\partial}{\partial t} u_n (t,y) + b^*(t,y)\cdot(\nabla u_n) (t,y) )\right)\rho^{ }(t)\eta\left(\hat{Y}^y_t(\omega)\right)\left\vert\det \frac{\partial}{\partial y}\hat{Y}^y_t(\omega)\right\vert \right\} dy dt 
		 \end{split}
		 \end{equation}
		 
		 Since $\frac{\partial}{\partial y}\hat{Y}^y_t(\omega)$, and therefore $\left\vert\det \frac{\partial}{\partial y}\hat{Y}^y_t(\omega)\right\vert$, is essentially bounded on $\mathcal{U}$ for this fixed $\omega\in\Omega^*$, it follows from our assumption that the last expression in (\ref{5constantcharacteristicsApprox}) converges for $n\rightarrow\infty$ to  
		 
		 \[-	\int_{0}^{T}\int_{\mathcal{U}} \left\{\left(\frac{\partial}{\partial t} u (t,y) + b^*(t,y)\cdot(\nabla u) (t,y) )\right)\rho^{ }(t)\eta\left(\hat{Y}^y_t(\omega)\right)\left\vert\det \frac{\partial}{\partial y}\hat{Y}^y_t(\omega)\right\vert \right\} dy dt \]
		 which is equal to zero.\par
		 On the other hand, by using Theorem \ref{5ChangeOfVariableFormula} once more we can argue similarly that 
		 
		 \begin{equation}\label{5constantcharacteristicsApprox2}
		 \begin{split}
		 \int_{0}^{T}\int_{\mathbb{R}^d} \Big(u_n (t,\hat{X}^x_t(\omega)) &\rho^{'}(t)\eta(x) \Big)  dx dt  \\ &= \int_{0}^{T}\int_{\mathcal{U}} \left(u_n (t,y)\rho^{'}(t)\eta\left(\hat{Y}^y_t(\omega)\right)\left\vert\det \frac{\partial}{\partial y}\hat{Y}^y_t(\omega)\right\vert \right) dy dt		 \\
		 \xrightarrow[n\rightarrow\infty]{}&	\int_{0}^{T}\int_{\mathcal{U}} \left(u  (t,y)\rho^{'}(t)\eta\left(\hat{Y}^y_t(\omega)\right)\left\vert\det \frac{\partial}{\partial y}\hat{Y}^y_t(\omega)\right\vert \right) dy dt \\
		 =& \int_{0}^{T}\int_{\mathbb{R}^d} \Big(u(t,\hat{X}^x_t(\omega)) \rho^{'}(t)\eta(x) \Big)  dx dt 
		 \end{split}
		 \end{equation}
		 
		 So (\ref{5constantcharacteristics}) holds for all test functions $\eta$ and $\rho$. Since $\left(\left(t,x\right)\mapsto u\left(t,\hat{X}^x_t(\omega)\right)\right)$ is continuous it follows that for all $x\in\mathbb{R}^d$ \[0 = \int_{0}^{T} u\left(t,\hat{X}^x_t(\omega)\right)\rho^{'}(t) dt\] for all $\rho\in\mathcal{C}_c^{\infty}\left(\left(0,T\right)\right)$.\par
		 
		 Thus we conclude that for all $x\in\mathbb{R}^d$ $\left( t \mapsto u\left(t,\hat{X}^x_t(\omega)\right)\right)$ is absolutely continuous and that \[u(t,\hat{X}^x_t) = u(0,\hat{X}^x_0) = u(0,x) = u_0(x)\] for all $t$ and $x$.\par
		 Hence \[u(t,x) = u_0\left(\hat{Y}^x_t(\omega)\right)\]
		 for all $t,x$.
		 
	\end{enumerate}
\end{proof}

\begin{remark}
    One may expect a similar result to the previous one in the case of a perturbation given by a Wiener process, that is $B^H_\cdot$, when $H=\frac{1}{2}$.\par
    However, the proof of Theorem 3.6 makes use of the property that $\frac{\partial}{\partial x} Y_t^x$ or $\frac{\partial}{\partial x} \hat{Y}_t^x$ belong to $L_{loc}^\infty\left([0,T]\times\mathbb{R}^d\right)$ a.e.\par
    It is not obvious, whether the latter property holds in the Wiener case, too. Therefore, we cannot directly see, how a result as in Theorem 3.6 could be established for perturbations along Wiener paths.
\end{remark}

\subsection{Existence of unique solutions to the Continuity equation (CE)}

In this section we study the following equation
\begin{equation}
\begin{cases}
\partial_t\mu + {\div} \left(b\mu\right) = 0 \\
\mu_0 = \bar{\mu},
\end{cases}
\tag{CE}
\label{5eqn:CE}	  
\end{equation}

and its related ODE 

\begin{equation}\label{5eqn:ODE}\tag{ODE}
\begin{cases}
\frac{d}{dt} X(t,x) = b\left(t,X(t,x)\right)\\
X(0,x) = x,
\end{cases}
\end{equation}

when the vector field $b$ belongs to $L_{\infty ,\infty }^{1,\infty }$ i.e.

\begin{equation*}
L^{1}(\mathbb{R}^{d};L^{\infty }([0,T];%
\mathbb{R}^{d}))\cap L^{\infty }(\mathbb{R}^{d};L^{\infty }([0,T];\mathbb{R}%
^{d})).\footnote{We assume in this section that $b$ is defined everywhere in order to make sense of the product of $b\mu$ when $\mu$ is a measure.}
\end{equation*}%

It is well known  (see \cite{5MS19}) that (\ref{5eqn:CE}) is in general ill-posed even when the vector field $b$ is continuous, and that a minimum degree of differentiability is necessary for uniqueness of solutions to (\ref{5eqn:CE}). Note that the existence of a flow is not sufficient for the uniqueness of solutions to (\ref{5eqn:CE}) for initial measures which are signed and further conditions on $b$ need to be imposed. When $b$ is Lipschitz, uniqueness in the class of signed measures holds via Theorem 8.1.7 in \cite{5AGS08}. Outside of the Lipschitz setting only few results are available in this class (See \cite{5BahouriChemin94}\cite{5AmbrosioBernard08}). In the class of positive measures the situation is much more clear by the use of the superposition principle which says that uniqueness of (\ref{5eqn:CE}) in the class of positive measures is equivalent to the uniqueness at the level of (\ref{5eqn:ODE}). More precisely we have 
\begin{theorem}\label{5SuperpositionPrinciple}
	Let $A\subset\mathbb{R}^d$ be a Borel set. Then the following two properties are equivalent:
	\begin{enumerate}
		\item Solutions of (\ref{5eqn:ODE}) are unique for every initial point $x\in A$;
		\item Positive measure-valued solutions of the (\ref{5eqn:CE}) are unique for every initial data $\bar{\mu}$ which is a positive measure concentrated on $A$, i.e. such that $\bar{\mu}(\mathbb{R}^d\setminus A) =0$. 	
	\end{enumerate}
\end{theorem}

The proof of this highly non trivial result can be found in e.g. \cite{5Crippa07}. Note that the superposition principle holds only for the class of positive measures and, without any assumptions on $b$, the result is false for signed measures by the counter example in \cite{5BonicattoGusev18} which shows that uniqueness at the level of the flows is not sufficient for the uniqueness of (\ref{5eqn:CE}) in the class of signed measures. 

The main result of this section says that if we consider again the transformation

\begin{equation}\label{5Psi}
\begin{cases}
\Psi : \Omega\times [0,T] \times\mathbb{R}^d\longrightarrow [0,T]\times\mathbb{R}^d\\
\Psi_{\omega}(t,x):= (t,x+B^H_t(\omega))
\end{cases}
\end{equation}

for a certain set $\Omega$ of full measure which depends only on $b$, then, by arguments similar to the previous section, we know that the following system 
\begin{equation}
\begin{cases}
\frac{d}{dt} \hat{X}(t,x) = b\circ\Psi_{\omega}\left(t,\hat{X}(t,x)\right)\\
\hat{X}(0,x) = x\in\mathbb{R}^d
\end{cases}
\end{equation}

is well posed for every $\omega\in\Omega$ provided $H$ is chosen appropriately, then we have by the superposition principle the following:

\begin{theorem}\label{5mainResultCE}
	Let $b\in L_{\infty ,\infty }^{1,\infty} $, $H< \frac{1}{2(d+2)}$. Define $b^*$ in (CE) as \[b^*(t,x)  = b(t,x+B_t^H(\omega)).
	\]
	
	Then there exists a measurable set $\tilde{\Omega}$ with $\mu(\tilde{\Omega}) =1$ such that for all $\omega\in\tilde{\Omega}$ and all initial data $\bar{\mu}\in\mathcal{M}^+(\mathbb{R}^d)$ there exists a unique weak solution $\mu_t = \mu_t^{\omega}$ given by
	\[\mu_t = \hat{X}(t,\cdot)_{\#}\bar{\mu}\]
	i.e.
	\begin{equation}
		\int_{\mathbb{R}^d} \varphi d\mu_t = \int_{\mathbb{R}^d} \varphi(\hat{X}(t,x)) d\bar{\mu}_t
	\end{equation}
    where $\hat{X}$ is the unique solution to (\ref{5PsiODE}).
\end{theorem}

\section{Further improvement of the spatial regularity of solutions of the transport equation with singular velocity fields perturbed along fractional Brownian paths with Hurst parameter "$H\downarrow 0$" }

In this section we want to explain how the spatial regularity of the constructed unique path-by-path solutions of the transport equation (\ref{5TEperturbed}) can be improved even further by using a perturbation based on fractional Brownian paths with a Hurst parameter approaching $0$.\par
To be more precise, we want to replace the fractional Brownian motion $B^H_{\cdot}$, for $H\in\left(0,\frac{1}{2}\right)$, in Sections 2 and 3 with an even more regularizing driving process $\mathbb{B}^H_{\cdot}$ for $H:=(H_n)_{n\geq 1}\subset\left(0,\frac{1}{2}\right)$ (see \cite{5ABP18}), defined as 
\begin{equation}
    \mathbb{B}_t^H = \sum_{n\geq 1}\lambda_n B_t^{H_n,n},\text{  }0\leq t\leq T
\end{equation}
where $B_{\cdot}^{H_n,n}$, $n\geq 1$ are independent $d$-dimensional fractional Brownian motions with Hurst parameters $H_n\in\left(0,\frac{1}{2}\right)$, $H_n\downarrow 0$ for $n\rightarrow\infty$.\par
Further, $\left(\lambda_n\right)_{n\geq 1}$ is a sequence of real numbers such that there exists a bijection
\begin{equation}\label{50fbm:cond1}
    \{n\in\mathbb{N}: \lambda_n\neq 0\}\rightarrow\mathbb{N},
\end{equation}
\begin{equation}\label{50fbm:cond2}
    \sum_{n\geq 1}\lvert\lambda_n\rvert\in \left(0,\infty\right),
\end{equation}
and
\begin{equation}\label{50fbm:cond3}
        \sum_{n\geq 1}\lvert\lambda_n\rvert E[\sup_{0\leq s\leq 1}\lvert B_s^{H_n,n}\rvert]<\infty
\end{equation}
Under conditions (\ref{50fbm:cond1}),(\ref{50fbm:cond2}) and (\ref{50fbm:cond3}) it was shown in \cite{5ABP18} that $\mathbb{B}_\cdot^H$ is a stationary Gaussian process, which possesses a continuous, but not H\"{o}lder continuous version.\par

In fact, for the vector fields

\begin{equation}\label{5class0fbm}
    b\in\mathcal{L}_{2,p}^q:=L^q([0,T];L^p(\mathbb{R}^d;\mathbb{R}^d))\cap L^1(\mathbb{R}^d;L^\infty([0,T];\mathbb{R}^d)),\quad p,q\in(2,\infty]
\end{equation}

the authors in \cite{5ABP18} established the following result by using an infinite dimensional version of a compactness criterion for square integrable Wiener functionals in \cite{5DMN92}:

\begin{theorem}\label{5ExistUniqSDE0fbm}
    There exists a sequence  $\left(\lambda_n\right)_{n\geq 1}$ depending only on  $\left(H_n\right)_{n\geq 1}$ and $d$, which satisfies the conditions  (\ref{50fbm:cond1}),(\ref{50fbm:cond2}) and (\ref{50fbm:cond3}) such that for all $b\in \mathcal{L}_{2,p}^q,\quad p,q\in(2,\infty]$ and $s\in[0,T]$, $x\in\mathbb{R}^d$ there exists a unique strong solution $X_\cdot^{s,x}$ to the SDE
    \begin{equation}\label{5SDE0fbm}
        X_t^{s,x} = x +\int_{s}^{t} b(u,X_u^{s,x})du + \mathbb{B}_t^H - \mathbb{B}_s^H,\quad X_s^{s,x} = x\in\mathbb{R}^d,\quad s\leq t\leq T
    \end{equation}
\end{theorem}
Moreover, for all $s\leq t\leq T$
\begin{equation}
\left(x\mapsto X_t^{s,x}\right)\in \bigcap_{k\geq 1}\bigcap_{\alpha>2} L^2(\Omega;W^{k,\alpha}(U))    
\end{equation}
for all bounded open subsets $U\subset\mathbb{R}^d$.\par
Using exactly the same techniques as in Section 2, we mention that one can prove (similarly to Theorem \ref{5PathByPath}) path-by-path uniqueness of solutions to (\ref{5SDE0fbm}) in the sense of Davie \cite{5Davie07} uniformly in the initial conditions.
\begin{theorem}\label{5Davie0fbm}
    Retain the conditions of Theorem \ref{5ExistUniqSDE0fbm} for $p,q=\infty $.\par
    Then for all $0\leq s\leq T$ there exists a measurable set $\Omega^*$ with $\mu(\Omega^*)=1$ such that for all $\omega\in\Omega^*$ the equation 
    \begin{equation}\label{5RDE0fbm}
        X_t^{s,x}(\omega) = x +\int_{s}^{t} b(u,X_u^{s,x}(\omega))du + \mathbb{B}_t^H(\omega) - \mathbb{B}_s^H(\omega),\quad X_s^{s,x} = x\in\mathbb{R}^d,\quad s\leq t\leq T
    \end{equation}
    has a unique solution in the space of continuous functions, uniformly in $x\in\mathbb{R}^d$.
\end{theorem}
By applying Theorem \ref{5Davie0fbm} and \ref{5ExistUniqSDE0fbm} in combination with very similar proofs as in Section 3.1, we can get the following result:

\begin{theorem}\label{5mainResultTE0fbm}
    Assume the conditions of Theorem \ref{5ExistUniqSDE0fbm} and let $b\in\mathcal{L}_{2,p^{'}}^{q^{'}}, p^{'},q^{'}=\infty$, $u_0\in\mathcal{C}_{b}^{\infty }(\mathbb{R}^{d})$, $p>d$. Define $b^*$ in \ref{5TEperturbed} by $b^*(t,x) = b(t,x+\mathbb{B}_t^H(\omega))$.\par
    Then there exists a measurable set $\Omega^*$ of full mass such that for all $\omega\in\Omega^*$ there exists a unique weak solution $u=u^\omega$ in the class $W_{loc}^{1,p}\left([0,T]\times\mathbb{R}^d\right)$ to the transport equation (\ref{5TEperturbed}).\par
    Furthermore, we have that \[u(t,\cdot)\in\bigcap_{k\geq 1}\bigcap_{\alpha\geq 1}W_{loc}^{k,\alpha}\left(\mathbb{R}^d\right)\]
    
    for all $t$.
\end{theorem}

We get as well the following theorem at the level of the continuity equation

\begin{theorem}\label{5mainResultCE0fbm}
	Let $b\in\mathcal{L}_{2,p^{}}^q, p^{},q^{}=\infty$. Define $b^*$ in (\ref{5eqn:CE}) as \[b^*(t,x) = b(t,x+\mathbb{B}_t^H(\omega)).
	\]
	
	Then there exists a measurable set $\tilde{\Omega}$ with full mass such that for all $\omega\in\tilde{\Omega}$ and all initial data $\bar{\mu}\in\mathcal{M}^+(\mathbb{R}^d)$ there exists a unique weak solution $\mu_t = \mu_t^{\omega}$ given by
	\[\mu_t = \hat{X}(t,\cdot)_{\#}\bar{\mu}\]
	i.e.
	\begin{equation}\label{5solutionCE}
		\int_{\mathbb{R}^d} \varphi d\mu_t = \int_{\mathbb{R}^d} \varphi(\hat{X}(t,x)) d\bar{\mu}_t,
	\end{equation}
	where $\hat{X}$ is the unique solution to
    \begin{equation}
    \begin{cases}
    \frac{d}{dt} \hat{X}(t,x) = b^{*}\left(t,\hat{X}(t,x)\right)\\
    \hat{X}(0,x) = x\in\mathbb{R}^d.
    \end{cases}
    \end{equation}

\end{theorem}

\section{Appendix}

	\bigskip We start by stating some basic facts about fractional Brownian motion and fractional calculus (see \cite{5Samko-et.al.93} and \cite%
{5Lizorkin01}). We then recall some results from the Malliavin calculus with respect to fractional Brownian motion. For an in-depth treatment of this material see \cite{5Nualart10}. We end with a collection of some technical lemmas that we make use of in our paper. 
	
	\subsection{Fractional Brownian motion}
\bigskip We want to recall here a version of Girsanov's theorem for the
fractional Brownian motion. For this purpose, let us pass in review some
basic concepts from fractional calculus (see \cite{5Samko-et.al.93} and \cite%
{5Lizorkin01}).

Let $a,$ $b\in \mathbb{R}$ with $a<b$. Let $f\in L^{p}([a,b])$ with $p\geq 1$
and $\alpha >0$. Introduce the \emph{left-} and \emph{right-sided
Riemann-Liouville fractional integrals} as 
\begin{equation*}
I_{a^{+}}^{\alpha }f(x)=\frac{1}{\Gamma (\alpha )}\int_{a}^{x}(x-y)^{\alpha
-1}f(y)dy
\end{equation*}%
and 
\begin{equation*}
I_{b^{-}}^{\alpha }f(x)=\frac{1}{\Gamma (\alpha )}\int_{x}^{b}(y-x)^{\alpha
-1}f(y)dy
\end{equation*}%
for almost all $x\in \lbrack a,b]$, where $\Gamma $ is the Gamma function.

For a given integer $p\geq 1$, let $I_{a^{+}}^{\alpha }(L^{p})$ (resp. $%
I_{b^{-}}^{\alpha }(L^{p})$) be the image of $L^{p}([a,b])$ of the operator $%
I_{a^{+}}^{\alpha }$ (resp. $I_{b^{-}}^{\alpha }$). If $f\in
I_{a^{+}}^{\alpha }(L^{p})$ (resp. $f\in I_{b^{-}}^{\alpha }(L^{p})$) and $%
0<\alpha <1$ then we can define the \emph{left-} and \emph{right-sided
Riemann-Liouville fractional derivatives} by 

\begin{equation*}
D_{a^{+}}^{\alpha }f(x)=\frac{1}{\Gamma (1-\alpha )}\frac{d}{dx}\int_{a}^{x}%
\frac{f(y)}{(x-y)^{\alpha }}dy
\end{equation*}%
and 
\begin{equation*}
D_{b^{-}}^{\alpha }f(x)=\frac{1}{\Gamma (1-\alpha )}\frac{d}{dx}\int_{x}^{b}%
\frac{f(y)}{(y-x)^{\alpha }}dy.
\end{equation*}

The left- and right-sided derivatives of $f$ can be also represented as 
\begin{equation*}
D_{a^{+}}^{\alpha }f(x)=\frac{1}{\Gamma (1-\alpha )}\left( \frac{f(x)}{%
(x-a)^{\alpha }}+\alpha \int_{a}^{x}\frac{f(x)-f(y)}{(x-y)^{\alpha +1}}%
dy\right)
\end{equation*}%
and 
\begin{equation*}
D_{b^{-}}^{\alpha }f(x)=\frac{1}{\Gamma (1-\alpha )}\left( \frac{f(x)}{%
(b-x)^{\alpha }}+\alpha \int_{x}^{b}\frac{f(x)-f(y)}{(y-x)^{\alpha +1}}%
dy\right) .
\end{equation*}

Using the above definitions, one obtains that 
\begin{equation*}
I_{a^{+}}^{\alpha }(D_{a^{+}}^{\alpha }f)=f
\end{equation*}%
for all $f\in I_{a^{+}}^{\alpha }(L^{p})$ and 
\begin{equation*}
D_{a^{+}}^{\alpha }(I_{a^{+}}^{\alpha }f)=f
\end{equation*}%
for all $f\in L^{p}([a,b])$ and similarly for $I_{b^{-}}^{\alpha }$ and $%
D_{b^{-}}^{\alpha }$.

\bigskip

Let now $B^{H}=\{B_{t}^{H},t\in \lbrack 0,T]\}$ be a $d$-dimensional \emph{%
fractional Brownian motion} with Hurst parameter $H\in (0,1/2)$, that is $%
B^{H}$ is a centered Gaussian process with a covariance function given by 
\begin{equation*}
(R_{H}(t,s))_{i,j}:=E[B_{t}^{H,(i)}B_{s}^{H,(j)}]=\delta _{ij}\frac{1}{2}%
\left( t^{2H}+s^{2H}-|t-s|^{2H}\right) ,\quad i,j=1,\dots ,d,
\end{equation*}%
where $\delta _{ij}$ is one, if $i=j$, or zero else.

In the sequel we briefly recall the construction of the fractional Brownian
motion, which can be found in \cite{5Nualart10}. For simplicity, consider the
case $d=1$.

Let $\mathcal{E}$ be the set of step functions on $[0,T]$ and $\mathcal{H}$
be the Hilbert space given by the completion of $\mathcal{E}$ with respect
to the inner product 
\begin{equation*}
\langle 1_{[0,t]},1_{[0,s]}\rangle _{\mathcal{H}}=R_{H}(t,s).
\end{equation*}%
From that we get an extension of the mapping $1_{[0,t]}\mapsto B_{t}$ to an
isometry between $\mathcal{H}$ and a Gaussian subspace of $L^{2}(\Omega )$
with respect to $B^{H}$. We denote by $\varphi \mapsto B^{H}(\varphi )$ this
isometry.

If $H<1/2$, one shows that the covariance function $R_{H}(t,s)$ has the
representation

\bigskip\ 
\begin{equation}
R_{H}(t,s)=\int_{0}^{t\wedge s}K_{H}(t,u)K_{H}(s,u)du,  \label{52.2}
\end{equation}%
where 
\begin{equation}
K_{H}(t,s)=c_{H}\left[ \left( \frac{t}{s}\right) ^{H-\frac{1}{2}}(t-s)^{H-%
\frac{1}{2}}+\left( \frac{1}{2}-H\right) s^{\frac{1}{2}-H}\int_{s}^{t}u^{H-%
\frac{3}{2}}(u-s)^{H-\frac{1}{2}}du\right] .  \label{5KH}
\end{equation}%
Here $c_{H}=\sqrt{\frac{2H}{(1-2H)\beta (1-2H,H+1/2)}}$ and $\beta $ is the
Beta function. See \cite[Proposition 5.1.3]{5Nualart10}.

Based on the kernel $K_{H}$, one can introduce by means (\ref{52.2}) an
isometry $K_{H}^{\ast }$ between $\mathcal{E}$ and $L^{2}([0,T])$ such that $%
(K_{H}^{\ast }1_{[0,t]})(s)=K_{H}(t,s)1_{[0,t]}(s).$ This isometry has an
extension to the Hilbert space $\mathcal{H}$, which has the following
representations by means of fractional derivatives

\begin{equation*}
(K_{H}^{\ast }\varphi )(s)=c_{H}\Gamma \left( H+\frac{1}{2}\right) s^{\frac{1%
}{2}-H}\left( D_{T^{-}}^{\frac{1}{2}-H}u^{H-\frac{1}{2}}\varphi (u)\right)
(s)
\end{equation*}%
and 
\begin{align*}
(K_{H}^{\ast }\varphi )(s)=& \,c_{H}\Gamma \left( H+\frac{1}{2}\right)
\left( D_{T^{-}}^{\frac{1}{2}-H}\varphi (s)\right) (s) \\
& +c_{H}\left( \frac{1}{2}-H\right) \int_{s}^{T}\varphi (t)(t-s)^{H-\frac{3}{%
2}}\left( 1-\left( \frac{t}{s}\right) ^{H-\frac{1}{2}}\right) dt.
\end{align*}%
for $\varphi \in \mathcal{H}$. One also proves that $\mathcal{H}=I_{T^{-}}^{%
\frac{1}{2}-H}(L^{2})$. See \cite{5DecreusefondUstunel98} and \cite[Proposition 6]{5AMN01}.

Since $K_{H}^{\ast }$ is an isometry from $\mathcal{H}$ into $L^{2}([0,T])$,
the $d$-dimensional process $W=\{W_{t},t\in \lbrack 0,T]\}$ defined by 
\begin{equation}
W_{t}:=B^{H}((K_{H}^{\ast })^{-1}(1_{[0,t]}))  \label{5WBH}
\end{equation}%
is a Wiener process and the process $B^{H}$ can be represented as 
\begin{equation}
B_{t}^{H}=\int_{0}^{t}K_{H}(t,s)dW_{s}.  \label{5BHW}
\end{equation}%
See \cite{5AMN01}.

In what follows we also need the Definition of a fractional Brownian motion
with respect to a filtration.

\begin{definition}
Let $\mathcal{G}=\left\{ \mathcal{G}_{t}\right\} _{t\in \left[ 0,T\right] }$
be a filtration on $\left( \Omega ,\mathcal{F},P\right) $ satisfying the
usual conditions. A fractional Brownian motion $B^{H}$ is called a $\mathcal{%
G}$-fractional Brownian motion if the process $W$ defined by (\ref{5WBH}) is
a $\mathcal{G}$-Brownian motion.
\end{definition}

\bigskip

In the following, let $W$ be a standard Wiener process on a filtered
probability space $(\Omega ,\mathfrak{A},P),\{\mathcal{F}_{t}\}_{t\in
\lbrack 0,T]},$ where $\mathcal{F}=\{\mathcal{F}_{t}\}_{t\in \lbrack 0,T]}$
is the natural filtration generated by $W$ and augmented by all $P$-null
sets. Denote by $B:=B^{H}$ the fractional Brownian motion with Hurst
parameter $H\in (0,1/2)$ as in (\ref{5BHW}).

We aim at using a version of Girsanov's theorem for fractional Brownian
motion which is due to \cite[Theorem 4.9]{5DecreusefondUstunel98}. The version stated here
corresponds to that in \cite[Theorem 2]{5Nualart10}. To this end, we need the
definition of an isomorphism $K_{H}$ from $L^{2}([0,T])$ onto $I_{0+}^{H+%
\frac{1}{2}}(L^{2})$ with respect to the kernel $K_{H}(t,s)$ in terms of the
fractional integrals as follows (see \cite[Theorem 2.1]{5DecreusefondUstunel98}): 
\begin{equation*}
(K_{H}\varphi )(s)=I_{0^{+}}^{2H}s^{\frac{1}{2}-H}I_{0^{+}}^{\frac{1}{2}%
-H}s^{H-\frac{1}{2}}\varphi ,\quad \varphi \in L^{2}([0,T]).
\end{equation*}

Using this and the properties of the Riemann-Liouville fractional integrals
and derivatives, one can show that the inverse of $K_{H}$ can be represented
as 
\begin{equation*}
(K_{H}^{-1}\varphi )(s)=s^{\frac{1}{2}-H}D_{0^{+}}^{\frac{1}{2}-H}s^{H-\frac{%
1}{2}}D_{0^{+}}^{2H}\varphi (s),\quad \varphi \in I_{0+}^{H+\frac{1}{2}%
}(L^{2}).
\end{equation*}

From this one obtains for absolutely continuous functions $\varphi $ (see 
\cite{5NO02}) that 
\begin{equation*}
(K_{H}^{-1}\varphi )(s)=s^{H-\frac{1}{2}}I_{0^{+}}^{\frac{1}{2}-H}s^{\frac{1%
}{2}-H}\varphi ^{\prime }(s).
\end{equation*}

\begin{theorem}[Girsanov's theorem for fBm]
\label{5girsanov} Let $u=\{u_{t},t\in \lbrack 0,T]\}$ be an $\mathcal{F}$%
-adapted process with integrable trajectories and set $\widetilde{B}%
_{t}^{H}=B_{t}^{H}+\int_{0}^{t}u_{s}ds,\quad t\in \lbrack 0,T].$ Suppose that

\begin{itemize}
\item[(i)] $\int_{0}^{\cdot }u_{s}ds\in I_{0+}^{H+\frac{1}{2}}(L^{2}([0,T]))$%
, $P$-a.s.

\item[(ii)] $E[\xi_T]=1$ where 
\begin{equation*}
\xi_T := \exp\left\{-\int_0^T K_H^{-1}\left( \int_0^{\cdot} u_r
dr\right)(s)dW_s - \frac{1}{2} \int_0^T K_H^{-1} \left( \int_0^{\cdot} u_r
dr \right)^2(s)ds \right\}.
\end{equation*}
\end{itemize}

Then the shifted process $\widetilde{B}^H$ is an $\mathcal{F}$-fractional
Brownian motion with Hurst parameter $H$ under the new probability $%
\widetilde{P}$ defined by $\frac{d\widetilde{P}}{dP}=\xi_T$.
\end{theorem}

\begin{remark}
In the the multi-dimensional case, we define 
\begin{equation*}
(K_{H}\varphi )(s):=((K_{H}\varphi ^{(1)})(s),\dots ,(K_{H}\varphi
^{(d)})(s))^{\ast },\quad \varphi \in L^{2}([0,T];\mathbb{R}^{d}),
\end{equation*}%
where $\ast $ denotes transposition. Similarly for $K_{H}^{-1}$ and $%
K_{H}^{\ast }$.
\end{remark}

	\subsection{Malliavin calculus}

	Let $\mathcal{S}$ be the set of smooth and cylindrical random variables of the form \[ F = f(B^H(\phi_1),\ldots ,B^H(\phi_n)) \]
	where $n\geq 1$, $f\in C_b^{\infty}(\mathbb{R}^n)$ and $\phi_1,\ldots ,\phi_n \in\mathcal{H}$ ($\mathcal{H}$ is defined in the previous section). Given a random variable $F\in\mathcal{S}$ we define its derivative, as an element in $\mathcal{H}$, to be
	\[ D^H F = \sum_{i=1}^n \frac{\partial f}{\partial x_j} (B^H(\phi_1),\ldots ,B^H(\phi_n)) \phi_j \]
	For any $p\geq 1$, we define the Sobolev space $\mathbb{D}^{1,p}_H$ as the completion of $\mathcal{S}$ with respect to the norm \[ ||F||^p_{1,p} = \mathbb{E} |F|^p + \mathbb{E}||D^H F||^p_{\mathcal{H}}\]
	
	Note that that the previous holds for any $H\in (0,1)$ and in particular for $H=\frac{1}{2}$. Denote by $D:= D^{\frac{1}{2}}$ the Malliavin derivative with respect to $W$ and let $\mathbb{D}^{1,p} $ be its corresponding Sobolev space. We restate the following transfer principle, Proposition 5.2.1, \cite{5Nualart10}, which links the $D$ and $D^H$.
	\begin{proposition}
		For any $F\in \mathbb{D}^{1,p} $ \[K^*_H D^H F = D F\]
	\end{proposition}
	
	A corollary of the previous is the following:
	
	\begin{lemma}
		Let $H\in (0,1)$ and $p > 1$, then $B^H_t$ belongs to $\mathbb{D}^{1,p} $ for all $t>0$ and its Malliavin derivative is given by:
		
		\[D B^H_t (s) = \int_0^{\min (s,t)} K_H(t,u) du \]
		and hence
		\[D_\theta B^H_t = K_H (t,\theta) I_d \]
		for any $\theta\in (0,t)$ and where $I_d$ is the identity matrix.
	\end{lemma}

	\subsection{Technical results}
	In this article we also resort to the following technical lemma (see \cite[%
	Lemma 4.3]{5BNP19}):

\begin{lemma}
\label{5Novikov} Let $\tilde{B}_{t}^{H}$ be a $d$-dimensional fractional
Brownian motion with respect to $(\Omega ,\mathfrak{A},\tilde{P})$. Then for
every $k\in \mathbb{R}$ we have 
\begin{equation*}
\tilde{E}\left[ \exp \left\{ k\int_{0}^{T}\left\vert K_{H}^{-1}\left(
\int_{0}^{\cdot }b(r,\tilde{B}_{r}^{H})dr\right) (s)\right\vert
^{2}ds\right\} \right] \leq C_{H,d,\mu ,T}(\Vert b\Vert _{L_{\infty
}^{\infty }})
\end{equation*}%
for some continuous increasing function $C_{H,d,k,T}$ depending only on $H$, 
$d$, $T$ and $k$.

In particular, 
\begin{equation*}
\tilde{E}\left[ \mathcal{E}\left( \int_{0}^{T}K_{H}^{-1}\left(
\int_{0}^{\cdot }b(r,\tilde{B}_{r}^{H})dr\right) ^{\ast }(s)dW_{s}\right)
^{p}\right] \leq C_{H,d,\mu ,T}(\Vert b\Vert _{L_{\infty }^{\infty }}),
\end{equation*}%
where $\mathcal{E}(M_{t})$ is the Dolean-Dade exponential of a local
martingale $M_{t},0\leq t\leq T$ and where $\tilde{E}$ denotes expectation
under $\tilde{P}$ and $\ast $ transposition.
\end{lemma}

\bigskip

In this paper, we will also make use of an integration by parts formula for
iterated integrals based on \emph{shuffle permutations}. For this purpose,
let $m$ and $n$ be integers. Denote by $S(m,n)$ the set of shuffle
permutations, i.e. the set of permutations $\sigma :\{1,\dots
,m+n\}\rightarrow \{1,\dots ,m+n\}$ such that $\sigma (1)<\dots <\sigma (m)$
and $\sigma (m+1)<\dots <\sigma (m+n)$.

Introduce the $m$-dimensional simplex for $0\leq \theta <t\leq T$, 
\begin{equation*}
\Delta _{\theta ,t}^{m}:=\{(s_{m},\dots ,s_{1})\in \lbrack 0,T]^{m}:\,\theta
<s_{m}<\cdots <s_{1}<t\}.
\end{equation*}%
The product of two simplices can be represented as follows 
\begin{equation*}
\Delta _{\theta ,t}^{m}\times \Delta _{\theta ,t}^{n}=%
\mbox{\footnotesize
$\bigcup_{\sigma \in S(m,n)} \{(w_{m+n},\dots,w_1)\in [0,T]^{m+n} : \,
\theta< w_{\sigma(m+n)} <\cdots < w_{\sigma(1)} <t\} \cup \mathcal{N}$
\normalsize},
\end{equation*}%
where the set $\mathcal{N}$ has null Lebesgue measure. So, if $%
f_{i}:[0,T]\rightarrow \mathbb{R}$, $i=1,\dots ,m+n$ are integrable
functions we get that 
\begin{align}
\int_{\Delta _{\theta ,t}^{m}}\prod_{j=1}^{m}f_{j}(s_{j})ds_{m}\dots ds_{1}&
\int_{\Delta _{\theta ,t}^{n}}\prod_{j=m+1}^{m+n}f_{j}(s_{j})ds_{m+n}\dots
ds_{m+1}  \notag \\
& =\sum_{\sigma \in S(m,n)}\int_{\Delta _{\theta
,t}^{m+n}}\prod_{j=1}^{m+n}f_{\sigma (j)}(w_{j})dw_{m+n}\cdots dw_{1}.
\label{5shuffleIntegral}
\end{align}

A generalization of the latter relation is the following (see \cite{5BNP19}):

\begin{lemma}
\label{5partialshuffle} Let $n,$ $p$ and $k$ be non-negative integers, $k\leq
n$. Suppose we have integrable functions $f_{j}:[0,T]\rightarrow \mathbb{R}$%
, $j=1,\dots ,n$ and $g_{i}:[0,T]\rightarrow \mathbb{R}$, $i=1,\dots ,p$. We
may then write 
\begin{align*}
& \int_{\Delta _{\theta ,t}^{n}}f_{1}(s_{1})\dots f_{k}(s_{k})\int_{\Delta
_{\theta ,s_{k}}^{p}}g_{1}(r_{1})\dots g_{p}(r_{p})dr_{p}\dots
dr_{1}f_{k+1}(s_{k+1})\dots f_{n}(s_{n})ds_{n}\dots ds_{1} \\
& =\sum_{\sigma \in A_{n,p}}\int_{\Delta _{\theta ,t}^{n+p}}h_{1}^{\sigma
}(w_{1})\dots h_{n+p}^{\sigma }(w_{n+p})dw_{n+p}\dots dw_{1},
\end{align*}%
where $h_{l}^{\sigma }\in \{f_{j},g_{i}:1\leq j\leq n,1\leq i\leq p\}$.
Above $A_{n,p}$ stands for a subset of permutations of $\{1,\dots ,n+p\}$
such that $\#A_{n,p}\leq C^{n+p}$ for an appropriate constant $C\geq 1$.
Here $s_{0}:=\theta $.
\end{lemma}

\bigskip

In what follows we need an important
estimate (see e.g. Proposition 3.3 in the second revision of \cite{5BLPP18} for the improved version of the result that we state below).
In order to state this result, we need some notation. Let $m$ be an integer
and let $f:[0,T]^{m}\times (\mathbb{R}^{d})^{m}\rightarrow \mathbb{R}$ be a
function of the form 
\begin{equation}
f(s,z)=\prod_{j=1}^{m}f_{j}(s_{j},z_{j}),\quad s=(s_{1},\dots ,s_{m})\in
\lbrack 0,T]^{m},\quad z=(z_{1},\dots ,z_{m})\in (\mathbb{R}^{d})^{m},
\label{5ffnew}
\end{equation}%
where $f_{j}:[0,T]\times \mathbb{R}^{d}\rightarrow \mathbb{R}$, $j=1,\dots
,m $ are smooth functions with compact support. In addition, let $\varkappa
:[0,T]^{m}\rightarrow \mathbb{R}$ be a function of the form 
\begin{equation}
\varkappa (s)=\prod_{j=1}^{m}\varkappa _{j}(s_{j}),\quad s\in \lbrack
0,T]^{m},  \label{5kappa}
\end{equation}%
where $\varkappa _{j}:[0,T]\rightarrow \mathbb{R}$, $j=1,\dots ,m$ are
integrable functions.

Further, denote by $\alpha _{j}$ a multi-index and $D^{\alpha _{j}}$ its
corresponding differential operator. For $\alpha =(\alpha _{1},\dots ,\alpha
_{m})$ as an element of $\mathbb{N}_{0}^{d\times m}$ with $|\alpha
|:=\sum_{j=1}^{m}\sum_{l=1}^{d}\alpha _{j}^{(l)}$, we write 
\begin{equation*}
D^{\alpha }f(s,z)=\prod_{j=1}^{m}D^{\alpha _{j}}f_{j}(s_{j},z_{j}).
\end{equation*}

In \cite{5BNP19} the following integration by parts formula was shown:%
\begin{equation}
\int_{\Delta _{\theta ,t}^{m}}\prod\limits_{l=1}^{m}D^{\alpha
}f(s,z)du_{m}...du_{1}=\int_{(\mathbb{R}^{d})^{m}}\Lambda _{\alpha
}^{f}(\theta ,t,z)dz\text{,}  \label{IntPart}
\end{equation}%
where%
\begin{equation}
\Lambda _{\alpha }^{f}(\theta ,t,z)=(2\pi )^{-dm}\int_{(\mathbb{R}%
^{d})^{m}}\int_{\Delta _{\theta
,t}^{m}}\prod\limits_{j=1}^{m}f_{j}(s_{j},z_{j})(-iu_{j})^{\alpha _{j}}\exp
(-i\left\langle u_{j},B_{s_{j}}-z_{j}\right\rangle dsdu\text{.}
\label{Lambda}
\end{equation}

Let us also introduce the following notation: For $%
(s,z)=(s_{1},...,s_{m},z_{1},...,z_{m})\in \left[ 0,T\right] ^{m}\times (%
\mathbb{R}^{d})^{m}$ and a shuffle $\sigma \in S(m,m)$ we define%
\begin{equation*}
f_{\sigma }(s,z)=\prod\limits_{j=1}^{2m}f_{\left[ \sigma (j)\right]
}(s_{j},z_{\left[ \sigma (j)\right] })
\end{equation*}%
and%
\begin{equation*}
\varkappa _{\sigma }(s)=\prod\limits_{j=1}^{2m}\varkappa _{\left[ \sigma (j)%
\right] }(s_{j}),
\end{equation*}%
where $\left[ j\right] $ is equal to $j$, if $1\leq j\leq m$ and $j-m$, if $%
m+1\leq j\leq 2m$.

For a multiindex $\alpha $ we also define%
\begin{eqnarray*}
&&\Psi _{\alpha }^{f}(\theta ,t,z) \\
&=&\dprod\limits_{l=1}^{d}\sqrt{(2\left\vert \alpha ^{(l)}\right\vert )!}%
\sum_{\sigma \in S(m,m)}\int_{\Delta _{0,t_{2}}^{2m}}\left\vert f_{\sigma
}(s,z)\right\vert \dprod\limits_{j=1}^{2m}\frac{1}{\left\vert
s_{j}-s_{j-1}\right\vert ^{H(d+2\sum_{i=1}^{d}\alpha _{\left[ \sigma (j)%
\right] }^{(i)})}}ds_{1}...ds_{2m}.
\end{eqnarray*}

\bigskip

\begin{lemma}
\label{LocalTimeEstimate}Suppose that $\Psi _{\alpha }^{\varkappa f}(\theta
,t,z)<\infty $. Then, $\Lambda _{\alpha }^{g}(\theta ,t,z)$ given by (\ref%
{Lambda}) for $g=\varkappa f$ is a random variable in $L^{2}(\Omega )$ and
there exists a constant $C=C(T,H,d)>0$ such that%
\begin{equation*}
\left\vert E\left[ \int_{(\mathbb{R}^{d})^{m}}\Lambda _{\alpha }^{g}(\theta
,t,z)dz\right] \right\vert \leq C^{m/2+\left\vert \alpha \right\vert
/2}\int_{(\mathbb{R}^{d})^{m}}(\Psi _{\alpha }^{g}(\theta ,t,z))^{1/2}dz.
\end{equation*}
\end{lemma}

\begin{proof}
See the proof of Theorem 3.1 in \cite{5BNP19}.
\end{proof}

\bigskip

\begin{theorem}
\label{5mainestimate2} Let $B^{H},H\in (0,1/2)$ be a standard $d-$dimensional
fractional Brownian motion and functions $f$ and $\varkappa $ as in (\ref%
{5ffnew}), respectively as in (\ref{5kappa}). Let $\theta ,t\in \lbrack 0,T]$
with $\theta <t$ and%
\begin{equation*}
\varkappa _{j}(s)=(K_{H}(s,\theta ))^{\varepsilon _{j}},\theta <s<t
\end{equation*}%
for every $j=1,...,m$ with $(\varepsilon _{1},...,\varepsilon _{m})\in
\{0,1\}^{m}.$ Let $\alpha \in (\mathbb{N}_{0}^{d})^{m}$ be a multi-index. If 
\begin{equation*}
H<\frac{\frac{1}{2}-\gamma }{(d-1+2\sum_{l=1}^{d}\alpha _{j}^{(l)})}
\end{equation*}%
for all $j$, where $\gamma \in (0,H)$ is sufficiently small, then there
exists a universal constant $C$ (depending on $H$, $T$ and $d$, but
independent of $m$, $\{f_{i}\}_{i=1,...,m}$ and $\alpha $) such that for any 
$\theta ,t\in \lbrack 0,T]$ with $\theta <t$ we have%

\begin{equation*}
\begin{split}
&\left\vert E\int_{\Delta _{\theta ,t}^{m}}\left( \prod_{j=1}^{m}D^{\alpha
_{j}}f_{j}(s_{j},B_{s_{j}}^{H})\varkappa _{j}(s_{j})\right) ds\right\vert \\
\leq &C^{m+\left\vert \alpha \right\vert }\prod_{j=1}^{m}\left\Vert
f_{j}(\cdot ,z_{j})\right\Vert _{L^{1}(\mathbb{R}^{d};L^{\infty
}([0,T]))}\theta ^{(H-\frac{1}{2})\sum_{j=1}^{m}\varepsilon _{j}} \\
&\times \frac{(\prod_{l=1}^{d}(2\left\vert \alpha ^{(l)}\right\vert
)!)^{1/4}(t-\theta )^{-H(md+2\left\vert \alpha \right\vert )+(H-\frac{1}{2}%
-\gamma )\sum_{j=1}^{m}\varepsilon _{j}+m}}{\Gamma (-H(2md+4\left\vert
\alpha \right\vert )+2(H-\frac{1}{2}-\gamma )\sum_{j=1}^{m}\varepsilon
_{j}+2m)^{1/2}}.
\end{split}
\end{equation*}
\end{theorem}

\bigskip

The following auxiliary result is very useful in connection with iterated integrals:

\begin{lemma}
\label{5OrderDerivatives}Let $n,$ $p$ and $k$ be non-negative integers, $%
k\leq n$. Assume we have functions $f_{j}:[0,T]\rightarrow \mathbb{R}$, $%
j=1,\dots ,n$ and $g_{i}:[0,T]\rightarrow \mathbb{R}$, $i=1,\dots ,p$ such
that 
\begin{equation*}
f_{j}\in \left\{ \frac{\partial ^{\alpha _{j}^{(1)}+...+\alpha _{j}^{(d)}}}{%
\partial ^{\alpha _{j}^{(1)}}x_{1}...\partial ^{\alpha _{j}^{(d)}}x_{d}}%
b^{(r)}(u,X_{u}^{x}),\text{ }r=1,...,d\right\} ,\text{ }j=1,...,n
\end{equation*}%
and 
\begin{equation*}
g_{i}\in \left\{ \frac{\partial ^{\beta _{i}^{(1)}+...+\beta _{i}^{(d)}}}{%
\partial ^{\beta _{i}^{(1)}}x_{1}...\partial ^{\beta _{i}^{(d)}}x_{d}}%
b^{(r)}(u,X_{u}^{x}),\text{ }r=1,...,d\right\} ,\text{ }i=1,...,p
\end{equation*}%
for $\alpha :=(\alpha _{j}^{(l)})\in \mathbb{N}_{0}^{d\times n}$ and $\beta
:=(\beta _{i}^{(l)})\in \mathbb{N}_{0}^{d\times p},$ where $X_{\cdot }^{x}$
is the strong solution to 
\begin{equation*}
X_{t}^{x}=x+\int_{0}^{t}b(u,X_{u}^{x})du+B_{t}^{H},\text{ }0\leq t\leq T
\end{equation*}%
for $b=(b^{(1)},...,b^{(d)})$ with $b^{(r)}\in C_{c}((0,T)\times \mathbb{R}%
^{d})$ for all $r=1,...,d$. So (as we shall say in the sequel) the product $%
g_{1}(r_{1})\cdot \dots \cdot g_{p}(r_{p})$ has a total order of derivatives 
$\left\vert \beta \right\vert =\sum_{l=1}^{d}\sum_{i=1}^{p}\beta _{i}^{(l)}$%
. We know from Lemma \ref{5partialshuffle} that 

\begin{align}
& \int_{\Delta _{\theta ,t}^{n}}f_{1}(s_{1})\dots f_{k}(s_{k})\int_{\Delta
_{\theta ,s_{k}}^{p}}g_{1}(r_{1})\dots g_{p}(r_{p})dr_{p}\dots
dr_{1}f_{k+1}(s_{k+1})\dots f_{n}(s_{n})ds_{n}\dots ds_{1}  \notag \\
& =\sum_{\sigma \in A_{n,p}}\int_{\Delta _{\theta ,t}^{n+p}}h_{1}^{\sigma
}(w_{1})\dots h_{n+p}^{\sigma }(w_{n+p})dw_{n+p}\dots dw_{1},  \label{5h}
\end{align}%

where $h_{l}^{\sigma }\in \{f_{j},g_{i}:1\leq j\leq n,$ $1\leq i\leq p\}$, $%
A_{n,p}$ is a subset of permutations of $\{1,\dots ,n+p\}$ such that $%
\#A_{n,p}\leq C^{n+p}$ for an appropriate constant $C\geq 1$, and $%
s_{0}=\theta $. Then the products%
\begin{equation*}
h_{1}^{\sigma }(w_{1})\cdot \dots \cdot h_{n+p}^{\sigma }(w_{n+p})
\end{equation*}%
have a total order of derivatives given by $\left\vert \alpha \right\vert
+\left\vert \beta \right\vert .$
\end{lemma}

\begin{proof}
The result is proved by induction on $n$. For $n=1$ and $k=0$ the result is
trivial. For $k=1$ we have

\begin{equation*}
\begin{split}
	\int_{\theta }^{t}f_{1}(s_{1})\int_{\Delta _{\theta
		,s_{1}}^{p}}g_{1}(r_{1})\dots g_{p}(r_{p}) &dr_{p}\dots dr_{1}ds_{1} \\
=&\int_{\Delta _{\theta ,t}^{p+1}}f_{1}(w_{1})g_{1}(w_{2})\dots
g_{p}(w_{p+1})dw_{p+1}\dots dw_{1},
\end{split}
\end{equation*}%

where we have put $w_{1}=s_{1},$ $w_{2}=r_{1},\dots ,w_{p+1}=r_{p}$. Hence
the total order of derivatives involved in the product of the last integral
is given by $\sum_{l=1}^{d}\alpha
_{1}^{(l)}+\sum_{l=1}^{d}\sum_{i=1}^{p}\beta _{i}^{(l)}=\left\vert \alpha
\right\vert +\left\vert \beta \right\vert .$

Assume the result holds for $n$ and let us show that this implies that the
result is true for $n+1$. Either $k=0,1$ or $2\leq k\leq n+1$. For $k=0$ the
result is trivial. For $k=1$ we have 
\begin{align*}
&\int_{\Delta _{\theta ,t}^{n+1}} f_{1}(s_{1})\int_{\Delta _{\theta
,s_{1}}^{p}}g_{1}(r_{1})\dots g_{p}(r_{p})dr_{p}\dots
dr_{1}f_{2}(s_{2})\dots f_{n+1}(s_{n+1})ds_{n+1}\dots ds_{1} \\
& =\int_{\theta }^{t}f_{1}(s_{1}) \\ &\left( \int_{\Delta _{\theta
,s_{1}}^{n}}\int_{\Delta _{\theta ,s_{1}}^{p}}g_{1}(r_{1})\dots
g_{p}(r_{p})dr_{p}\dots dr_{1}f_{2}(s_{2})\dots
f_{n+1}(s_{n+1})ds_{n+1}\dots ds_{2}\right) ds_{1}.
\end{align*}%
From (\ref{5shuffleIntegral}) we observe by using the shuffle permutations
that the latter inner double integral on diagonals can be written as a sum
of integrals on diagonals of length $p+n$ with products having a total order
of derivatives given by $\sum_{l=1}\sum_{j=2}^{n+1}\alpha
_{j}^{(l)}+\sum_{l=1}^{d}\sum_{i=1}^{p}\beta _{i}^{(l)}$. Hence we obtain a
sum of products, whose total order of derivatives is $\sum_{l=1}^{d}%
\sum_{j=2}^{n+1}\alpha _{j}^{(l)}+\sum_{l=1}^{d}\sum_{i=1}^{p}\beta
_{i}^{(l)}+\sum_{l=1}^{d}\alpha _{1}^{(l)}=\left\vert \alpha \right\vert
+\left\vert \beta \right\vert .$

For $k\geq 2$ we have (in connection with Lemma \ref{5partialshuffle}) from
the induction hypothesis that%
{\small
\begin{align*}
\int_{\Delta _{\theta ,t}^{n+1}}f_{1}(s_{1})\dots f_{k}(s_{k})\int_{\Delta
_{\theta ,s_{k}}^{p}}g_{1}(r_{1})\dots g_{p}(r_{p})& dr_{p}\dots
dr_{1}f_{k+1}(s_{k+1})\dots f_{n+1}(s_{n+1})ds_{n+1}\dots ds_{1} \\
=\int_{\theta }^{t}f_{1}(s_{1})\int_{\Delta _{\theta
,s_{1}}^{n}}f_{2}(s_{2})\dots f_{k}(s_{k})& \int_{\Delta _{\theta
,s_{k}}^{p}}g_{1}(r_{1})\dots g_{p}(r_{p})dr_{p}\dots dr_{1} \\
& \times f_{k+1}(s_{k+1})\dots f_{n+1}(s_{n+1})ds_{n+1}\dots ds_{2}ds_{1} \\
=\sum_{\sigma \in A_{n,p}}\int_{\theta }^{t}f_{1}(s_{1})\int_{\Delta
_{\theta ,s_{1}}^{n+p}}& h_{1}^{\sigma }(w_{1})\dots h_{n+p}^{\sigma
}(w_{n+p})dw_{n+p}\dots dw_{1}ds_{1},
\end{align*}}%
where each of the products $h_{1}^{\sigma }(w_{1})\cdot \dots \cdot
h_{n+p}^{\sigma }(w_{n+p})$ has a total order of derivatives given by $%
\sum_{l=1}\sum_{j=2}^{n+1}\alpha
_{j}^{(l)}+\sum_{l=1}^{d}\sum_{i=1}^{p}\beta _{i}^{(l)}.$ Thus we get a sum
with respect to a set of permutations $A_{n+1,p}$ with products having a
total order of derivatives which is%
\begin{equation*}
\sum_{l=1}^{d}\sum_{j=2}^{n+1}\alpha
_{j}^{(l)}+\sum_{l=1}^{d}\sum_{i=1}^{p}\beta _{i}^{(l)}+\sum_{l=1}^{d}\alpha
_{1}^{(l)}=\left\vert \alpha \right\vert +\left\vert \beta \right\vert .
\end{equation*}
\end{proof}

\bigskip

The following result can be found in \cite{5ACHP18} or \cite{5BNP19}:

\begin{lemma}
\bigskip \label{5BoundedDerivatives}Let $b\in C_{c}^{\infty }((0,T)\times 
\mathbb{R}^{d}).$ Fix integers $p\geq 2$ and $k\geq 1$. Then, if $H<\frac{1}{2(d-1+2k)},$ we
have 
\begin{equation*}
\sup_{s,t\in [0,T]}\sup_{x\in \mathbb{R}^{d}}E[\left\Vert \frac{\partial^k }{\partial x^k}%
X_{t}^{s,x}\right\Vert ^{p}]\leq C_{k,p,H,d,T}(\left\Vert b\right\Vert
_{L_{\infty }^{\infty }},\left\Vert b\right\Vert _{L_{\infty }^{1}})<\infty
\end{equation*}%
for some continuous function $C_{k,p,H,d,T}:[0,\infty )^{2}\longrightarrow
\lbrack 0,\infty )$.
\end{lemma}

\bigskip

The following result which is due to \cite[Theorem 1]{5DMN92} provides a
compactness criterion for subsets of $L^{2}(\Omega)$ using Malliavin
calculus.

\begin{theorem}
\label{5VI_MCompactness}Let $\left\{ \left( \Omega ,\mathcal{A},P\right)
;H\right\} $ be a Gaussian probability space, that is $\left( \Omega ,%
\mathcal{A},P\right) $ is a probability space and $H$ a separable closed
subspace of Gaussian random variables of $L^{2}(\Omega )$, which generate
the $\sigma $-field $\mathcal{A}$. Denote by $\mathbf{D}$ the derivative
operator acting on elementary smooth random variables in the sense that%
\begin{equation*}
\mathbf{D}(f(h_{1},\ldots,h_{n}))=\sum_{i=1}^{n}\partial
_{i}f(h_{1},\ldots,h_{n})h_{i},\text{ }h_{i}\in H,f\in C_{b}^{\infty }(%
\mathbb{R}^{n}).
\end{equation*}%
Further let $\mathbb{D}^{1,2}$ be the closure of the family of elementary
smooth random variables with respect to the norm%
\begin{align*}
\left\Vert F\right\Vert _{1,2}:=\left\Vert F\right\Vert _{L^{2}(\Omega
)}+\left\Vert \mathbf{D}F\right\Vert _{L^{2}(\Omega ;H)}.
\end{align*}%
Assume that $C$ is a self-adjoint compact operator on $H$ with dense image.
Then for any $c>0$ the set 
\begin{equation*}
\mathcal{G}=\left\{ G\in \mathbb{D}^{1,2}:\left\Vert G\right\Vert
_{L^{2}(\Omega )}+\left\Vert C^{-1} \mathbf{D} \,G\right\Vert _{L^{2}(\Omega
;H)}\leq c\right\}
\end{equation*}%
is relatively compact in $L^{2}(\Omega )$.
\end{theorem}

In order to formulate compactness criteria useful for our purposes, we need
the following technical result which also can be found in \cite{5DMN92}.

\begin{lemma}
\label{5VI_DaPMN} Let $v_{s},s\geq 0$ be the Haar basis of $L^{2}([0,T])$.
For any $0<\alpha <1/2$ define the operator $A_{\alpha }$ on $L^{2}([0,T])$
by%
\begin{equation*}
A_{\alpha }v_{s}=2^{k\alpha }v_{s}\text{, if }s=2^{k}+j\text{ }
\end{equation*}%
for $k\geq 0,0\leq j\leq 2^{k}$ and%
\begin{equation*}
A_{\alpha }1=1.
\end{equation*}%
Then for all $\beta $ with $\alpha <\beta <(1/2),$ there exists a constant $%
c_{1}$ such that%
\begin{equation*}
\left\Vert A_{\alpha }f\right\Vert \leq c_{1}\left\{ \left\Vert f\right\Vert
_{L^{2}([0,T])}+\left(\int_{0}^{T}\int_{0}^{T}\frac{\left| f(t)-f(t^{\prime
})\right|^2}{\left\vert t-t^{\prime }\right\vert ^{1+2\beta }}dt\,dt^{\prime
}\right)^{1/2}\right\} .
\end{equation*}
\end{lemma}

A direct consequence of Theorem \ref{5VI_MCompactness} and Lemma \ref%
{5VI_DaPMN} is now the following compactness criteria. See \cite{5DMN92} for
a proof.

\begin{corollary}
\label{5VI_compactcrit} Let a sequence of $\mathcal{F}_T$-measurable random
variables $X_n\in\mathbb{D}^{1,2}$, $n=1,2...$, be such that there exists a
constant $C>0$ with 
\begin{equation*}
\sup_n \text{\emph{E}}[|X_n|^2] \leq C ,
\end{equation*}
\begin{equation*}
\sup_n \text{\emph{E}}\left[ \| D_t X_n \|_{L^2([0,T])}^2 \right] \leq C \,
\end{equation*}
and there exists a $\beta \in (0,1/2)$ such that 
\begin{equation*}
\sup_n \int_0^T \int_0^T \frac{\text{\emph{E}}\left[ \| D_t X_n -
D_{t^{\prime}} X_n \|^2 \right]}{|t-t^{\prime}|^{1+2\beta}}
dtdt^{\prime}<\infty
\end{equation*}
where $\|\cdot\|$ denotes any matrix norm.

Then the sequence $X_n$, $n=1,2...$, is relatively compact in $L^{2}(\Omega
) $.
\end{corollary}

We also need the following special version of a lemma due to Garcia,
Rodemich and Rumsey (see e.g. \cite{Kwapien} or \cite{Stroock})

\begin{lemma}
\label{GarciaRodemichRumsey}Let $\Lambda $ be a compact interval endowed
with a metric $d$. Define $\sigma (r)=\inf_{x\in \Lambda }\lambda (B(x,r))$,
where $B(x,r):=\left\{ y\in \Lambda :d(x,y)\leq r\right\} $ denotes the ball
of radius $r$ centered in $x\in \Lambda $ and where $\lambda $ is the
Lebesgue measure. Assume that $\Psi :\left[ 0,\infty \right) \longrightarrow %
\left[ 0,\infty \right) $ is positive, increasing and convex with $\Psi
(0)=0 $ and denote by $\Psi ^{-1}$ its inverse, which is a positive,
increasing and concave function. Suppose that $f:$ $\left[ 0,\infty \right)
\longrightarrow \left[ 0,\infty \right) $ is continuous on $\left( \Lambda
,d\right) $ and let%
\begin{equation*}
U=\int_{\Lambda \times \Lambda }\Psi \left( \frac{\left\vert
f(t)-f(s)\right\vert }{d(t,s)}\right) dtds\text{.}
\end{equation*}%
Then%
\begin{equation*}
\left\vert f(t)-f(s)\right\vert \leq 18\int_{0}^{d(t,s)/2}\Psi ^{-1}\left( 
\frac{U}{(\sigma (r))^{2}}\right) dr.
\end{equation*}
\end{lemma}

The following lemma is used in the proof of Proposition \ref{5Holder flow}

\begin{lemma}\label{5ExtKolmogorov}
	Let $X(s, x)$ be a continuous with respect to $x$ process with values in a complete metric space $(M, \varrho_M)$ on 
	$\mathcal{S} \times [0, 1]^d$.
	Assume that for some $a, b > 0$
	$$ 
	\sup_{s \in \mathcal{S}}\mathbb{E} \varrho_{M} (X_{s}(u), X_{s}(v))^{a} \leq |u - v|^{d + b}, \ u, v \in [0, 1]^d
	$$
	For any $\alpha \in (0, b/a), \eta \in (0, b - \alpha a)$ 
	and any increasing sequence $S$
	of finite subsets $\{S_n\}_{n = 0}^{\infty}$ with  
	$|S_{n}| \leq 2^{\eta  n}$
	there exists a set $\Omega'$ of probability $1$ such that
	$$ 
	\varrho_{M} (X_{s}(u), X_{s}(v)) \leq C(\alpha, \eta, S, \omega)|u - v|^{\alpha} \ s \in S_n,  u, v \in [0, 1]^d, 
	|u - v| \leq 2^{-n},
	\omega \in \Omega',
	$$ 
\end{lemma}

\begin{proof}
	See \cite{5Shaposhnikov17}.
\end{proof}

\begin{proposition}\label{5FlowProperty}
	Let $b\in L^{1,\infty}_{\infty,\infty}$ and $H<\frac{1}{2(d+2)}$.\par
	Further, let $\left((t,x)\mapsto X_t^{s,x}\right)$ be a locally H\"{o}lder continuous process (see Lemma \ref{5JointHolderContinuity}), which for all fixed $x\in\mathbb{R}^d$ uniquely satisfies the SDE (2.1) a.e.\par
	Then for all $s$ there exists a measurable set $\Omega^+$ with $\mu(\Omega^+) = 1$ such that for all $t\in[s,T]$, $x\in\mathbb{R}^d$ the following holds 
	\[X_t^x(\omega) = X_t^{s,X_s^x}(\omega)\]
\end{proposition}

\begin{proof}
	Let $\{b_n\}_{n\geq 1}\subset\mathcal{C}_c^{\infty}\left([0,T]\times\mathbb{R}^d\right)$ be an approximation sequence for $b$ as in Lemma \ref{5JointHolderContinuity}.
	Then we can find continuous processes $\left((t,x)\mapsto X_t^{s,x,n}\right)$, $n\geq 1$, which uniquely solve the SDE (\ref{5SDE}) with respect to the vector fields $b_n$ $\mu$-a.e. uniformly in $x\in\mathbb{R}^d$.\par
	Using uniqueness and the fact that $b_n\in\mathcal{C}_c^{\infty}\left([0,T]\times\mathbb{R}^d\right)$, we can find an $\Omega^{'}$ with $\mu(\Omega^{'})=1$ such that for all $\omega\in\Omega^{'}$, $n\geq 1$, $t\in[0,T]$ and $x\in\mathbb{R}^d$ we have 
		\[X_t^{x,n}(\omega) = X_t^{s,X_s^{x,n},n}(\omega)\]
		
	Let $\psi\in \mathcal{C}_c^{\infty}\left(\mathbb{R}^d\right) $ be non-negative and $t\in\left(s,T\right]$ fixed. Then using the change of variable formula (Theorem \ref{5ChangeOfVariableFormula}) we see that 
	\begin{equation}
		\begin{split}
			E\Big[\int_{\mathbb{R}^d} \Big\lvert  X_t^{s,X_s^{x,n},n}\psi(X_s^{x,n})-& X_t^{s,X_s^{x,n}}\psi(X_s^{x,n}) \Big\rvert dx\Big]\\
			&= E\left[\int_{\mathbb{R}^d} \left\lvert  X_t^{s,y,n} - X_t^{s,y} \right\rvert \psi(y) \left\lvert \text{det}\frac{\partial}{\partial y} Y_s^{y,n} \right\rvert dy\right]
		\end{split}
	\end{equation}
	
	where $Y_t^{x,n}$ is the inverse process of $X_t^{x,n}$.\par
	By applying H\"{o}lder's inequality we obtain that 
	
	\[E\left[\int_{\mathbb{R}^d} \left\lvert  X_t^{s,y,n} - X_t^{s,y} \right\rvert \psi(y) \left\lvert \text{det}\frac{\partial}{\partial y} Y_s^{y,n} \right\rvert dy \right] \leq J_1(n)\times J_2(n)\]

	where \[J_1(n):=E\left[\int_{\mathbb{R}^d} \left\lvert  X_t^{s,y,n} - X_t^{s,y} \right\rvert^2 \psi(y)  dy \right]^{\frac{1}{2}}  \]
	
	and \[J_2(n):=E\left[\int_{\mathbb{R}^d} \left\lvert  \text{det}\frac{\partial}{\partial y} Y_s^{y,n} \right\rvert^2 \psi(y)  dy \right]^{\frac{1}{2}}  \]

	Since $\psi$ has compact support, we see from Theorem 2.7 that $J_1(n) \xrightarrow{n\rightarrow\infty} 0 $\par
	
	On the other hand, since
	
	 \[J_2(n) =\left( \int_{\mathbb{R}^d} E\left[\left\lvert  \text{det}\frac{\partial}{\partial y} Y_s^{y,n} \right\rvert^2 \right] \psi(y)  dy  \right)^\frac{1}{2} \]
	Lemma \ref{5BoundedDerivatives} gives \[\sup_{n\in \mathbb{N},y\in\mathbb{R}^d}  E\left[\left\lvert  \text{det}\frac{\partial}{\partial y} Y_s^{y,n} \right\rvert^2 \right] <\infty \]
	Thus $J_2(n)\leq M$ for all $n$ for a constant $M$.\par
	Hence \[E\left[\int_{\mathbb{R}^d} \left\lvert  X_t^{s,y,n} - X_t^{s,y} \right\rvert \psi(y) \left\lvert \text{det}\frac{\partial}{\partial y} Y_s^{y,n} \right\rvert dy \right]\xrightarrow{n\rightarrow\infty} 0 \]
	and consequently
	\[E\left[\int_{\mathbb{R}^d} \left\lvert  X_t^{s,X_s^{x,n},n}\psi(X_s^{x,n})- X_t^{s,X_s^{x,n}}\psi(X_s^{x,n}) \right\rvert dx\right]\xrightarrow{n\rightarrow\infty} 0 \]
	
	It follows that there exists a subsequence $n_k$, $k\geq 1$ and a null set $\mathcal{N}$ such that for all $x\in\mathcal{U}:=\mathbb{R}^d\setminus\mathcal{N}$
	\[\left\lVert  X_t^{s,X_s^{x,n_k},n_k}\psi(X_s^{x,n_k})- X_t^{s,X_s^{x,n_k}}\psi(X_s^{x,n_k})  \right\rVert_{L^1(\Omega)}\xrightarrow{k\rightarrow\infty} 0 \]
	
	Thus for all $x\in\mathcal{U}$ there exists a subsequence $m_l=m_l(x)$ and an $\tilde{\Omega} = \tilde{\Omega}(x)$ with $\mu(\tilde{\Omega}) = 1$ such that for all $\omega\in\tilde{\Omega}$
	\begin{equation}\label{5proofFlow1}
\left( X_t^{s,X_s^{x,n_{m_l}},n_{m_l}}\psi(X_s^{x,n_{m_l}})\right)(\omega)- \left(X_t^{s,X_s^{x,n_{m_l}}}\psi(X_s^{x,n_{m_l}}) \right)(\omega)  \xrightarrow{l\rightarrow\infty} 0
	\end{equation} 
	
	Set $r_l = r_l(x) = n_{m_l}(x)$, $l\geq 1$.\par
	On the other hand, we know that for all $x\in\mathbb{R}^d$ \[\lVert X_s^{x,n} - X_s^x \rVert_{L^1(\Omega)}\xrightarrow{n\rightarrow\infty} 0 \]
	So in particular 
	\[\lVert X_s^{x,r_l(x)} - X_s^x \rVert_{L^1(\Omega)}\xrightarrow{l\rightarrow\infty} 0 \]
	for all $x\in\mathcal{U}$.\par
	Then we can find for all $x\in\mathcal{U}$ another subsequence, say, $z_i = z_i(x)$, $i\geq 1$ and an $\Omega^* = \Omega^*(x)$ such that for all $\omega\in\Omega^*$
	\[\left(X_s^{x,r_{z_i(x)}}\right)(\omega)\xrightarrow{i\rightarrow\infty} X_s^x(\omega)\]
	
	Hence, we find by continuity that for all $x\in\mathcal{U}$ and $\omega\in\Omega^{**} = \Omega^{**}(x) \subseteq \Omega^{*} = \Omega^{*}(x) $ with $\mu(\Omega^{**}) = 1$
	
	\begin{equation}\label{5proofFlow2}
		\left( X_t^{s,X_s^{x,r_{z_i(x)}}}\psi(X_s^{x,r_{z_i(x)}})\right)(\omega)- \left(X_t^{s,X_s^{x}}\psi(X_s^{x}) \right)(\omega)  \xrightarrow{i\rightarrow\infty} 0
	\end{equation}

		We also see from (\ref{5proofFlow1}) that for all $x\in\mathcal{U}$, $\omega\in\tilde{\Omega}(x)\bigcap\Omega^{**}(x)$
		\begin{equation}\label{5proofFlow3}
		\left( X_t^{s,X_s^{x,r_{z_i(x)}},r_{z_i(x)}}\psi(X_s^{x,r_{z_i(x)}})\right)(\omega)- \left(X_t^{s,X_s^{x,r_{z_i(x)}}}\psi(X_s^{x,r_{z_i(x)}}) \right)(\omega)  \xrightarrow{i\rightarrow\infty} 0
		\end{equation}
	
		Setting $\bar{r}_i = \bar{r}_i(x) = r_{z_i(x)}$, we further get that for all $x\in\mathcal{U}$ there exist $l_j =l_j(x)$, $j\geq 1$ and $\bar{\Omega} = \bar{\Omega}(x)$ with $\mu(\bar{\Omega})=1$ such that 
	
	\begin{equation}\label{5proofFlow4}
	\left( X_t^{x,\bar{r}_{l_j(x)}}\psi(X_s^{x,\bar{r}_{l_j(x)}})\right)(\omega) \xrightarrow{j\rightarrow\infty} X_t^x(\omega)\psi(X_s^x(\omega))
	\end{equation}
	So it follows from (\ref{5proofFlow2}), (\ref{5proofFlow3}) and (\ref{5proofFlow4}) that for all $x\in\mathcal{U}$, $\omega\in\bar{\Omega}(x)\bigcap\tilde{\Omega}(x)\bigcap\Omega^{**}(x)$

		\begin{equation}\label{5proofFlow5}
X_t^x(\omega)\psi(X_s^x(\omega)) = \left(X_t^{s,X_s^x}\psi(X_s^x)\right)(\omega)
	\end{equation}
	
	Consider now a sequence $\{\psi_n\}_{n\geq 1}\subset \mathcal{C}_c^{\infty}(\mathbb{R}^d)$ such that \[\psi_n(y)\xrightarrow{n\rightarrow\infty} 1\] for all $y\in\mathbb{R}^d$.
	
	Hence (\ref{5proofFlow5}) shows that for all $x\in\mathcal{U}$ there exists a $\Omega^\circ = \Omega^\circ(x)$ such that for all $\omega\in\Omega^\circ$ and $n\geq 1$ 
	\[X_t^x(\omega)\psi_n(X_s^x(\omega)) = \left(X_t^{s,X_s^x}\psi_n(X_s^x)\right)(\omega)\]
	
	It follows that for all $x\in\mathcal{U}$, $\omega\in\Omega^\circ(x)$
	\begin{equation}\label{5proofFlow6}
		X_t^x(\omega) = X_t^{s,X_s^x}(\omega)
	\end{equation}
	
	Since both sides of (\ref{5proofFlow6}) are continuous, in $t$ and $x$, we can find an $\Omega^+$ with $\mu(\Omega^+) = 1$ such that for all $\omega\in\Omega^+$, $t\in\left[s,T\right]$, $x\in\mathbb{R}^d$ \[X_t^x(\omega) = X_t^{s,X_s^x}(\omega).\]
\end{proof}

\begin{remark}\label{5remark:uniformFlow}
We mention that one can prove the existence of a continuous version of $\left( (s,t,x)\mapsto X_t^{s,x}\right)$ (compare e.g. \cite{5MNP14} in the Wiener case). Using this, one can even show the flow property in Proposition \ref{5FlowProperty} for a measurable set $\Omega^*$ of full mass, uniformly in $s,t$ and $x$.\par
\end{remark}

In the same way as in Proposition \ref{5FlowProperty}, one can prove the following result:
\begin{proposition}\label{5prop:homeomorphism}
    Under the conditions of Proposition 3.13 there exists a measurable set $\Omega^*$ with $\mu(\Omega^*) = 1$ such that for all $\omega\in\Omega^*$ the mapping
    \[\left(x\mapsto X_t^{s,x}(x)\right)\]
    is a homeomorphism for all $t\in\left(s,T\right]$.
\end{proposition}

\end{document}